\documentclass[11pt]{article}

\usepackage{amsmath, amsthm, amssymb, amscd, paralist}
\usepackage{mathtools}
\usepackage{scalerel}
\usepackage{graphicx}
\usepackage[dvipsnames]{xcolor}
\usepackage{bm}
\usepackage[title]{appendix}
\usepackage{hyperref}
\usepackage{enumitem}
\usepackage[right=3cm,left=3cm,top=2cm,bottom=2cm]{geometry}
\usepackage{tikz}
\usetikzlibrary{matrix,arrows,positioning}
\usepackage[ruled,linesnumbered]{algorithm2e}

\newcommand{\eps}{\varepsilon}

\newcommand{\sgn}{\mathop{\rm sgn}\nolimits}
\newcommand{\argmin}{\mathop{\rm arg\min}}

\newcommand{\maxnew}{\operatornamewithlimits{\mathstrut \rm \max}}
\newcommand{\infnew}{\operatornamewithlimits{\mathstrut \rm \inf}}

\numberwithin{equation}{section}
\newtheorem{thm}{Theorem}[section]
\newtheorem{lem}[thm]{Lemma}
\newtheorem{rem}[thm]{Remark}
\newtheorem{prop}[thm]{Proposition}
\newtheorem{cor}[thm]{Corollary}
\newtheorem{assump}[thm]{Assumption}
\newtheorem{defi}[thm]{Definition}

\newcommand{\Var}{\operatorname{Var}}
\newcommand{\Cov}{\operatorname{Cov}}

\newcommand{\E}{\mathbb{E}}

\renewcommand{\P}{\mathbb{P}}

\newcommand{\Risk}{\operatorname{Risk}}

\newcommand{\B}{\mathbb{B}}

\newcommand{\R}{\mathbb{R}}

\newcommand{\N}{\mathbb{N}}
\newcommand{\X}{\mathbb{X}}

\newcommand{\mA}{\mathcal{A}}

\newcommand{\mC}{\mathcal{C}}
\newcommand{\mD}{\mathcal{D}}
\newcommand{\mE}{\mathcal{E}}
\newcommand{\mF}{\mathcal{F}}

\newcommand{\mH}{\mathcal{H}}
\newcommand{\mI}{\mathcal{I}}

\newcommand{\mM}{\mathcal{M}}
\newcommand{\mN}{\mathcal{N}}
\newcommand{\mO}{\mathcal{O}}
\newcommand{\mP}{\mathcal{P}}
\newcommand{\mQ}{\mathcal{Q}}

\newcommand{\mX}{\mathcal{X}}

\newcommand{\frakr}{{\mathfrak r}}
\newcommand{\frakm}{{\mathfrak m}}

\newcommand{\Ind}{\mathbf{1}}

\newcommand{\be}{\mathbf{e}}

\newcommand{\bv}{\mathbf{v}}

\newcommand{\bx}{\mathbf{x}}
\newcommand{\bX}{\mathbf{X}}
\newcommand{\by}{\mathbf{y}}

\newcommand{\bbeta}{\bm{\beta}}

\newcommand{\footremember}[2]{%
   \footnote{#2}
    \newcounter{#1}
    \setcounter{#1}{\value{footnote}}%
}
\newcommand{\footrecall}[1]{%
    \footnotemark[\value{#1}]%
} 

\begin{document}


\title{Robust-to-outliers square-root LASSO, \\ simultaneous inference with a MOM approach}

\author{%
    Gianluca Finocchio\footremember{utwente}{University of Twente, Enschede, the Netherlands.}\footremember{trailer}{Research supported by TOP grant and a Vidi grant from the Dutch science organization (NWO).},
    \, Alexis Derumigny\footremember{tudelft}{Department of Applied Mathematics, Delft University of Technology, Delft, the Netherlands.}%
    \, and Katharina Proksch\footrecall{utwente}
}

\date{\today}
\maketitle

\begin{abstract}
    We consider the least-squares regression problem with unknown noise variance, where the observed data points are allowed to be corrupted by outliers.
    Building on the median-of-means (MOM) method introduced by Lecue and Lerasle~\cite{lecue2020robustML} in the case of known noise variance, we propose a general MOM approach for simultaneous inference of both the regression function and the noise variance, requiring only an upper bound on the noise level.
    Interestingly, this generalization requires care due to regularity issues that are intrinsic to the underlying convex-concave optimization problem.
    In the general case where the regression function belongs to a convex class, we show that our simultaneous estimator achieves with high probability the same convergence rates and a similar risk bound as if the noise level was unknown, as well as convergence rates for the estimated noise standard deviation.
    
    In the high-dimensional sparse linear setting, our estimator yields a robust analog of the square-root LASSO. Under weak moment conditions, it jointly achieves with high probability the minimax rates of estimation $s^{1/p} \sqrt{(1/n) \log(p/s)}$ for the $\ell_p$-norm of the coefficient vector, and the rate $\sqrt{(s/n) \log(p/s)}$ for the estimation of the noise standard deviation. Here $n$ denotes the sample size, $p$ the dimension and $s$ the sparsity level. We finally propose an extension to the case of unknown sparsity level $s$, providing a jointly adaptive estimator $(\widetilde \beta, \widetilde \sigma, \widetilde s)$. It simultaneously estimates the coefficient vector, the noise level and the sparsity level, with proven bounds on each of these three components that hold with high probability.
\end{abstract}



\section{Introduction}\label{sec.intro}

We consider the statistical learning problem of predicting a real random variable $Y$ by means of an explanatory variable $\bX$ belonging to some measurable space $\mX.$ Given a dataset $\mD$ of observations and a function class $\mF,$ the goal is to choose a function $\widehat f\in\mF$ in such a way that $\widehat f(\bX)$ approximates $Y$ as well as possible.
In particular, we study the problem of predicting $Y$ with the mean-squared loss, which corresponds to the estimation of an \textit{oracle function} $f^* \in \argmin_{f \in \mF} \E[(Y-f(\bX))^2].$
This setting has been formalized by~\cite{lecue2020robustML} in the context of robust machine learning. In this framework, one observes a (possibly) contaminated dataset consisting of \textit{informative observations} (sometimes called \textit{inliers}), and \textit{outliers}. The statistician does not know which data points are corrupted and nothing is usually assumed about the outliers, however one expects the informative observations to be sufficient to solve the problem at hand if the number of outliers is not too large.
Even when the inliers are a sample of i.i.d. observations with finite second-moment, such a corrupted dataset can break naive estimators even in the simplest of problems: a single big outlier can push  an empirical average towards infinity when estimating the mean of a real random variable.
A much better choice of estimator in the presence of outliers is the so-called median-of-means, which is constructed as follows: given a partition of the dataset into some number $K$ of blocks, one computes the empirical average relative to each block, and then takes the median of all these empirical averages. The resulting object is robust to $K/2$ outliers and has good performance even when the underlying distribution has no second moment, see \cite[Section 4.1]{devroye2016sub}. Some of the key ideas behind the median-of-means construction can be traced back to the work on stochastic optimization~\cite{NemYud1985,Levin2005}, sampling from large discrete structures~\cite{Jerrum1986}, and sketching algorithms~\cite{Alon1999}.

Our work builds on the MOM method introduced in~\cite{lecue2020robustML}, which solves the least-squares problem by implementing a convex-concave optimization of a suitable functional.
In the sparse linear case, this problem can be rewritten as the estimation of $\bbeta^*$ in the model $Y=\bX^T \bbeta^* + \zeta$ for some noise $\zeta$, where $\mF_{s^*} = \{ \bx \mapsto \bx^T \bbeta: \bbeta \in \R^d,\ |\bbeta|_0\leq s^*\}$ for some sparsity level $s^*>0$ and $|\bbeta|_0$ is the number of non-zero components of $\bbeta$.
The MOM-LASSO method~\cite{lecue2020robustML} yields there a robust version of the LASSO estimator, which is known to be minimax optimal, see \cite{bellec2017, bellec2018,bellec2016}, but its optimal penalization parameter has to be proportional to the noise standard deviation $\sigma^*.$
However, in practical applications this noise level $\sigma^*$ is often unknown to the statistician, and, as a consequence, it may be difficult to apply the MOM-LASSO.
We extend this MOM approach to the case of unknown noise variance and highlight the challenges that arise from this formulation of the problem.
The main contribution of our paper is the choice of a new functional in the convex-concave procedure that yields, in the sparse linear case, a robust version of the square-root LASSO introduced in~\cite{belloni2010}, which was shown to be minimax optimal by~\cite{derumigny2018improved}, while its penalization parameter does not require knowledge of $\sigma^*.$ Interestingly, intuitive and seemingly innocuous choices of functional end up requiring too restrictive assumptions, such as a known (or estimated) lower bound $\sigma_{-}>0$ on the noise standard deviation as in \cite{derumigny2019thesis}, whereas in this article, we only require a known (or estimated) upper bound $\sigma_{+}.$

Our main results deal with the simultaneous estimation of the oracle function $f^*$ and standard deviation $\sigma^*$ of the residual $\zeta := Y - f^*(\X)$.
In the high-dimensional sparse linear regression setting with unknown $\sigma^*$, if the sparsity level $s^*\leq d$ is known and the number of outliers is no more than $O(s^*\log(ed/s^*)),$ we prove that our MOM achieves the optimal rates of estimation of $\bbeta^*$ using a number of blocks $K$ of order $O(s^*\log(ed/s^*)).$
We also prove that our estimator of the noise standard deviation satisfies $|\widehat\sigma_{K,\mu} - \sigma^* | \lesssim \sigma_+ \sqrt{\frac{s^*}{n} \log \left( \frac{ed}{s^*} \right)}$ with high probability, improving the rates compared to the previous best estimator $\hat \sigma$, see ~\cite[Corollary 2]{belloni2014pivotal}, which satisfies $|\hat \sigma^2 - \sigma^2| \lesssim \sigma^*{}^2 \Big( \frac{s^* \log(n \vee d \log n)}{n} + \sqrt{\frac{s^* \log(d \, \vee \, n)}{n}} + \frac{1}{\sqrt{n}} \Big)$ whenever the noise has a finite fourth moment. Note that these rates for the estimation of $\sigma^*$ derived in \cite{belloni2014pivotal} correspond to a different penalty level than the one used in \cite{derumigny2018improved} that allows to derive optimal rates for the estimation of $\bbeta^*$.
A related paper is~\cite{comminges2018adaptive}, which studies optimal noise level estimation for the sparse Gaussian sequence model.

Since the sparsity level may be unknown in practice, we provide an aggregated adaptive procedure based on Lepski's method, that is, we first infer an estimated sparsity $\widetilde s$ and then an estimated number of blocks $\widetilde K$ of order $O(\widetilde s\log(ed/\widetilde s)).$
We show that the resulting adaptive estimator $(\widetilde\bbeta, \widetilde\sigma, \widetilde s)$ attains the minimax rates for the estimation of $\bbeta^*$ while still being adaptive to the unknown noise variance $\sigma^2$ and selecting a sparse model ($\widetilde s \leq s^*$) with high probability.

\begin{table}[thb]
    \centering
    \resizebox{\textwidth}{!}{%
    \begin{tabular}{c|cccc}
        Estimator & Rate on $\bbeta$ & Adapt. to $s$ &
        Rate and adapt. to $\sigma^*$ & Robustness \\
        \hline
        Lasso & Optimal \cite{bellec2018} & - & - & - \\
        Aggreg. Lasso & Optimal \cite{bellec2018} & Yes & - & - \\
        Square-root Lasso & Optimal \cite{derumigny2018improved} & - & Yes, complicated rate \cite{belloni2014pivotal} & - \\
        Aggreg. Square-root Lasso & Optimal \cite{derumigny2018improved} & Yes & Yes, but no rate & - \\
        MOM-Lasso & Optimal \cite{lecue2020robustML} & - & - & Yes \\
        Aggreg. MOM-Lasso & Optimal \cite{lecue2020robustML} & Yes & - & Yes \\
        \textbf{Robust SR-Lasso} & Optimal (Th. \ref{thm.rates_AMOM_sigmaplus}) & - & $\sqrt{\frac{s^*}{n}\log\left(\frac{ed}{s^*} \right)}$ (Th. \ref{thm.rates_AMOM_sigmaplus}) & Yes \\
        \textbf{Aggreg. Robust SR-Lasso} & Optimal (Th. \ref{thm.MOM_sqrt_lasso_adaptive}) & Yes & $\sqrt{\frac{s^*}{n}\log\left(\frac{ed}{s^*} \right)}$ (Th. \ref{thm.MOM_sqrt_lasso_adaptive}) & Yes \\
    \end{tabular}
    }
    \caption{Comparison of estimators of sparse high-dimensional regressions and their main theoretical properties. Names in bold print refer to the new estimators that we propose in this article.}
    \label{tab:comp_esti}
\end{table}

In Table~\ref{tab:comp_esti}, we detail a comparison of the Lasso-type estimators and their different theoretical properties in this sparse high-dimensional regression framework.
The two new estimators that we propose solve the problem of minimax-optimal robust estimation of~$\bbeta$.
Even in the setting where no outliers are present, our estimators still improve the best-known bounds on the estimation of the noise variance $\sigma^{*2}$. Moreover, the second estimator $(\widetilde\bbeta, \widetilde \sigma, \widetilde s)$ attains the same rate of simultaneous estimation of $\bbeta^*$ and $\sigma^*$ adaptively to the sparsity level $s^*$.
Finally, the estimator $(\widetilde\bbeta, \widetilde \sigma)$ is robust to the same number of outliers as the estimator which uses the knowledge of the true sparsity level $s^*$.
For every $\sigma^* > 0$, let $\mP(\sigma^*)$ be a class of distributions of $(\bX,\zeta)$ such that the kurtosis of $\zeta$ is bounded, $\Var[\zeta]={\sigma^*}^2$ and $\bX$ is isotropic, satisfies a weak moment condition and is such that the weighted norms $L^1(\P_{\bX}),L^2(\P_{\bX}),$ and $L^4(\P_{\bX})$ are equivalent on $\R^d.$ We work with a dataset $\mD = (\bX_i,Y_i)_{i=1,\ldots,n}$ that might be contaminated by a set of outliers $(\bX_i,Y_i)_{i\in\mO}$ (for some $\mO \subset \{1,\ldots,n\}$) in the sense that, for $i \in \mO$, $(\bX_i,Y_i)$ is an arbitrary outlier while for $i \notin \mO$, $(\bX_i,Y_i)$ is i.i.d. distributed as $(\bX, Y)$. We denote by $\mD(N)$ the set of all possible modifications of $\mD$ by at most $N$ observations. 
To sum up, our joint estimator $(\widetilde\bbeta, \widetilde \sigma, \widetilde s)$ satisfies the following worst-case simultaneous deviation bound
\begin{align*}
    \infnew_{s^* = 1, \dots, s_+}
    &\inf_{ \scaleto{\begin{array}{c}
        \bbeta^* \in \mF_{s^*} \\ \sigma^* < \sigma_+
    \end{array} }{25pt} }
    \infnew_{P_{\bX,\zeta} \in \mP(\sigma^*)}
    P_{\beta^*,P_{\bX,\zeta}}^{\otimes n}
    \Big( \mA_{\sigma^*,\bbeta^*,s^*}(\mD) \Big)
    \geq 1 - \phi(s_+,d),
\end{align*}
where the event $\mA_{\sigma^*,\bbeta^*,s^*}(\mD)$ describes the performance of the aggregated estimator over a class of contaminations of the dataset $\mD$ by arbitrary outliers. Formally,
\begin{align*}
    \mA_{\sigma^*,\bbeta^*,s^*}(\mD)
    &:= \bigcap_{\mD' \in \mD \big(cs^*\log ( ed/s^*)\big)} 
    \mA_{\sigma^*}(\mD') \cap \mA_{\bbeta^*}(\mD') \cap \mA_{s^*}(\mD'),\\
    \mA_{\sigma^*}(\mD')
    &:= \left\{ \Big|\widetilde\sigma \big( \mD' \big) - \sigma^* \Big| \leq C \sigma_+ \sqrt{\frac{s^*}{n} \log \Big( \frac{ed}{s^*} \Big)}\right\}, \\
    \mA_{\bbeta^*}(\mD')
    &:= \left\{ \Big| \widetilde\bbeta \big( \mD' \big) - \bbeta^* \Big|_p
    \leq C \sigma_+  {s^*}^{1/p} \sqrt{\frac{1}{n} \log \Big( \frac{ed}{s^*} \Big)} \right\}, \\
    \mA_{s^*}(\mD')
    &:= \Big\{ \widetilde s \big( \mD' \big) \leq s^* \Big\},
\end{align*}
where $(\widetilde\bbeta(\mD'), \widetilde \sigma(\mD'), \widetilde s(\mD'))$ is the joint estimator obtained from the perturbed dataset $\mD'.$ Our method only requires the knowledge of the upper bounds $(\sigma_+, s_+)$,
where
$\mF_s$ is the set of $s$-sparse vectors, $|\cdot|_p$ is the $\ell_p$ norm, $\phi(s,d):=4 (\log_2(s)+1)^2 ( 2s / ed )^{C' s}$ for a universal constant $C' > 0$, the constants $c, C>0$ only depend on the class $\mP(\sigma^*)$, $|\mO|$ denotes the cardinality of the set $\mO$ and $P_{\beta^*,P_{\bX,\zeta}}$ is the distribution of $(\bX,Y)$ when $(\bX, \zeta) \sim P_{\bX,\zeta}$ and $Y=\bX^\top \bbeta^* + \zeta$.

The manuscript is organized as follows. In Section~\ref{sec.notation}, we introduce the main framework and notation, as well as the step-by-step construction of the MOM estimator.
In Section~\ref{sec.general_F} we present our results in the general situation of a convex class $\mF$ of regression functions. The results for the high-dimensional sparse linear regression framework are presented in Section~\ref{sec.high_dim_reg}.
In Section~\ref{sec.discussion} we discuss the contraction rates, the construction of the MOM estimator and some known results from the literature. The proofs are gathered in the appendix.

\section{Notation and framework}\label{sec.notation}

\subsection{General notation}

Vectors are denoted by bold letters, e.g. $\bx:=(x_1,\ldots,x_d)^\top.$ For $S\subseteq \{1,\ldots,d\},$ we write $|S|$ for the cardinality of $S.$ As usual, we define $|\bx|_p:= (\sum_{i=1}^d |\bx_i|^p)^{1/p},$ $|\bx|_\infty:= \max_i|\bx_i|,$ $|\bx|_0:= \sum_{i=1}^d \Ind(\bx_i \neq 0),$ where $\Ind$ is the indicator function and write $\|f\|_{L^p(D)}$ for the $L^p$ norm of $f$ on $D$. If there is no ambiguity concerning the domain $D,$ we also write $\|\cdot\|_p.$ We set $|\bx|_{2,n} := |\bx|_2/\sqrt{n}$ and, for a measure $\mu$ on $\R^d$ and a function $f$ in a class of functions $\mF,$ we define $\|f\|_{2,\nu} := \|f\|_{L^2(\nu)}.$ The expected value of a random variable $X$ with respect to a measure $P$ is denoted $PX.$ For two sequences $(a_n)_n$ and $(b_n)_n$ we write $a_n \lesssim b_n$ if there exists a constant $C$ such that $a_n \leq C b_n$ for all $n.$ Moreover, $a_n\asymp b_n$ means that $(a_n)_n\lesssim(b_n)_n$ and $(b_n)_n\lesssim(a_n)_n.$

\subsection{Mathematical framework}

The goal is to predict a square-integrable random variable $Y\in\R$ by means of an explanatory random variable $\bX,$ on a measurable space $\mX,$ and a dataset $\mD = \{(\bX_i,Y_i) \in\mX\times\R : i = 1,\ldots,n\}.$ Let $\P_{\bX}$ be the law of $\bX$ and $L^2(\P_{\bX})$ the corresponding weighted $L^2-$space. Let $\mF \subseteq L^2(\P_{\bX})$ be a convex class of functions from $\mX$ to $\R,$ so that, for any $f\in\mF,$ $\|f\|_{2,\bX}^2 := \int_\mX f(\bx)^2d\P_{\bX}(\bx)$ is finite. We consider the least-squares problem, which requires to minimize the \textit{risk} $\Risk(f) := \E[(Y-f(\bX))^2]$ among all possible predictions $f(\bX)$ for $Y,$ which in turn minimizes the variance of the residuals $\zeta_f := Y - f(\bX).$ The best predictor on $L^2(\P_{\bX})$ is the conditional mean $\overline f(\bX) = E[Y|\bX],$ which can only be computed when the joint distribution of $(\bX,Y)$ is given. Therefore, one solves the least-squares problem by estimating any \textit{oracle solution} 
\begin{align}\label{def.f*_F*}
    f^* \in \mF^* := \argmin_{f\in\mF} \E\big[(Y-f(\bX))^2\big],
\end{align}
which is unique, i.e. $\mF^*=\{f^*\},$ if the class $\mF\subseteq L^2(\P_{\bX})$ is closed (on top of being convex). The resulting representation is
\begin{align}\label{def.model}
    Y = f^*(\bX) + \zeta,\quad \zeta := Y - f^*(\bX),
\end{align}
where the residual $\zeta$ and $\bX$ may not be independent. 
\begin{assump}\label{ass.zeta_moments}
    We make the following assumptions on the residual $\zeta,$
    \begin{align} \label{def.zeta_moments}
        \E[\zeta] = 0, \quad
        \sigma^* := \E[\zeta^2]^{\frac{1}{2}} \leq \sigma_{+}, \quad
        \frakm^* := \E[\zeta^4]^{\frac{1}{4}} \leq \frakm_{+} := \sigma_{+}\kappa_{+}, \quad
        \kappa^* := \frac{\frakm^{*4}}{\sigma^{*4}} \leq \kappa_{+},
    \end{align}
with possibly unknown $\sigma^*,\frakm^*,\kappa^*$ and upper bounds $\sigma_{+},\kappa_{+}$ either given or estimated from the data. We use the convention that $\kappa^*=0$ if both $\sigma^*$ and $\frakm^*$ are zero.
\end{assump}
Without loss of generality we have $\sigma_{+} \leq \frakm_{+},$ since any upper bound on  $\frakm^{*}$ is also an upper bound on the standard deviation $\sigma^{*}.$ The requirement of a known upper bound on the fourth moment of the noise is natural when dealing with MOM procedures, this is in line with Assumption~3.1 in~\cite{lugosi2017regularization}. We aim at simultaneously estimating  $(f^*,\sigma^*)$ from the dataset $\mD,$ but the problem is made more difficult due to possible outliers in the observations.

\begin{assump}\label{ass.inlier_iid}
    We assume the dataset $\mD$ can be partitioned into an \textit{informative set} $\mD_{\mI}$ and an \textit{outlier set} $\mD_{\mO}$ satisfying the following. \vspace{-0.2cm}
    \begin{itemize}
        \item \textbf{Informative data.} We assume that the pairs $(\bX_i,Y_i)_{i\in\mI} =: \mD_{\mI}$ with $\mI \subseteq \{1,\ldots,n\}$ are independent and distributed as $(\bX,Y)$ in the regression model~\eqref{def.model}.
        
        \item \textbf{Outliers.} Nothing is assumed on the pairs $(\bX_i,Y_i)_{i\in\mO} =: \mD_{\mO}$ with $\mO \subseteq \{1,\ldots,n\}.$ They might be deterministic or even adversarial, in the sense that they might depend on the informative sample $(\bX_i,Y_i)_{i\in\mI}$ defined above, or on the choice of estimator.
    \end{itemize}
\end{assump}
The i.i.d. requirement on the informative data can be weakened, as in~\cite{lecue2020robustML}, by assuming that the observations $(\bX_i,Y_i)_{i \in \mI}$ are independent and, for all $i \in \mI$
\begin{align*}
    \E[(Y_i-f^*(\bX_i))(f-f^*)(\bX_i)] &= \E[(Y-f^*(\bX))(f-f^*)(\bX)],\\
    \E[(f-f^*)^2(\bX_i)] &= \E[(f-f^*)^2(\bX)].
\end{align*}
In other words, the distributions of $(\bX_i,Y_i)$ and $(\bX,Y)$ induce the same $L^2-$metric on the function space $\mF-f^*=\{f-f^*: f\in\mF\}.$

By construction, $\mI\cup\mO = \{1,\ldots,n\}$ and $\mI\cap\mO = \emptyset,$ but the statistician does not know whether any fixed index $i\in\{1,\ldots,n\}$ belongs to $\mI$ or $\mO.$ Otherwise, one could just remove this group from the dataset and perform the inference of the informative part. In order to achieve robust inference, we implement a median-of-means approach.

\textbf{The sparse linear case.} We highlight the special case when $\mX = \R^d,$ with a fixed dimension $d>0.$ For $\bbeta\in\R^d,$ set $f_{\bbeta}:\R^d\to\R$ the linear map $f_{\bbeta}(\bx)=\bx^\top\bbeta.$ For any $1 \leq s\leq d,$ we define
\begin{align*}
    \mF := \{f_{\bbeta} : \bbeta\in\R^d\},\quad \mF_s := \big\{ f_{\bbeta}\in\mF : \bbeta\in\R^d,\ |\bbeta|_0\leq s \big\},
\end{align*}
here $|\bbeta|_0$ is the number of non-zero entries of $\bbeta\in\R^d.$

\subsection{Convex-concave formulation} \label{sec.conv_conc}

We follow the formalization made in~\cite{lecue2020robustML}. For any function $f\in\mF,$ and any $(\bx,y)\in\mX\times\R,$ set $\ell_f(\bx,y) := (y-f(\bx))^2.$ In our setting we find
\begin{align*}
    f^* \in \argmin_{f\in\mF} \E\big[\ell_f(\bX,Y) \big],\quad \sigma^* = \E\big[\ell_{f^*}(\bX,Y) \big]^{\frac{1}{2}},
\end{align*}
since $\E[\ell_{f^*}(\bX,Y)] = \E[\zeta^2]$ is the risk of the oracle function $f^*.$ The oracle pair $(f^*,\sigma^*)$ is a solution of the convex-concave problem
\begin{align}\label{def.f*_minmax}
    f^* \in \argmin_{f\in\mF} \sup_{g\in\mF} \E\big[\ell_f(\bX,Y) - \ell_g(\bX,Y) \big],\quad \sigma^* = \E\big[\ell_{f^*}(\bX,Y) \big]^{\frac{1}{2}},
\end{align}
and the goal is to build an estimator $(\widehat f,\widehat\sigma)$ such that, with probability as high as possible, the quantities
\begin{align*}
    \Risk(\widehat f) - \Risk(f^*),\quad \|\widehat f-f^*\|_{2,\bX},\quad |\widehat\sigma - \sigma^*|,
\end{align*}
are as small as possible. The quantity $\Risk(\widehat f) - \Risk(f^*)$ is the \textit{excess risk}, whereas the quantity $\|\widehat f-f^*\|_{2,\bX}$ is the convergence rate in $L^2(\P_{\bX})-$norm of the random function $\widehat f$ to $f^*.$ Since $\widehat f$ is a function of the dataset $\mD,$ we always mean that the expectation is conditional on $\mD,$ i.e. $\|\widehat f-f^*\|_{2,\bX} = \E[(\widehat f - f^*)^2(\bX)|\mD].$ Finally, the quantity $|\widehat\sigma - \sigma^*|$ is the convergence rate of $\widehat\sigma$ to $\sigma^*.$ 

\subsection{Construction of the estimator}\label{sec.estimator}

The starting point of our approach is the regularized median-of-means (MOM) tournament introduced in \cite{lugosi2019regularization}, which has been proposed as a procedure to outperform the \textit{regularized empirical risk minimizer} (RERM)
\begin{align*}
    \widehat f_\lambda^{RERM} := \argmin_{f\in\mF}\left\{\frac{1}{n}\sum_{i=1}^n (Y_i - f(\bX_i))^2 + \lambda \|f\| \right\}, 
\end{align*}
with $\|\cdot\|$ a \textit{penalization} norm on the linear span of $\mF$ and $\lambda > 0$ a \textit{penalization parameter.} The penalization term reduces overfitting by assigning a higher cost to functions that are big with respect to $\|\cdot\|.$ The RERM estimator above is susceptible to outliers since it involves all the pairs $(\bX_i,Y_i)$ in the dataset $\mD,$ whereas replacing the empirical average by the corresponding median-of-means over a number of blocks leads to robustness. The MOM method in \cite{lecue2020robustML} builds directly on the theory of the MOM tournaments and it exploits the fact that $\widehat f_\lambda^{RERM}$ is computed by minimizing $n^{-1} \sum_{i=1}^n \ell_f(\bX_i,Y_i) + \lambda \|f\|.$ From this, the authors deal with the convex-concave equivalent
\begin{align*}
    \widehat f_\lambda^{RERM} := \argmin_{f\in\mF} \sup_{g\in\mF} \left\{\frac{1}{n}\sum_{i=1}^n \ell_f(\bX_i,Y_i) - \frac{1}{n}\sum_{i=1}^n \ell_g(\bX_i,Y_i) + \lambda (\|f\|-\|g\|) \right\}, 
\end{align*}
by replacing the empirical average $n^{-1} \sum_{i=1}^n \big(\ell_f(\bX_i,Y_i) - \ell_g(\bX_i,Y_i) \big)$ with the median-of-means over a chosen number of blocks. Our goal is to extend the scope of this procedure to the estimation of the unknown $\sigma^*.$ To this end, we modify the convex-concave RERM by replacing the functional $R(\ell_g,\ell_f) = \ell_f - \ell_g$ with a new $R_c(\ell_g,\chi,\ell_f,\sigma)$ that incorporates $\chi,\sigma\in I_{+} = (0,\sigma_{+}].$ This leads to a generalized empirical estimator
\begin{align*}
    (\widehat f_\mu, \widehat \sigma_\mu) := \argmin_{(f,\sigma)\in\mF\times I_{+}} \sup_{(g,\chi)\in\mF\times I_{+}} \left\{\frac{1}{n}\sum_{i=1}^n R_c\left(\ell_g(\bX_i,Y_i),\chi,\ell_f(\bX_i,Y_i),\sigma\right) + \mu (\|f\|-\|g\|) \right\}, 
\end{align*}
which we robustify using the MOM. The choice of the functional $R_c$ is crucial for the performance of the procedure and a main contribution of our paper is providing a suitable $R_c(\ell_g,\chi,\ell_f,\sigma),$ we refer to  Section~\ref{sec.discussion} for a detailed discussion motivating our choice.

We give the step-by-step construction of a family of MOM estimators for $(f^*,\sigma^*)$ from model~\eqref{def.f*_F*}--\eqref{def.zeta_moments}. We start with a preliminary definition.

\textbf{Quantiles.} For any $K\in\N,$ set $[K]=\{1,\ldots,K\}.$ For all $\alpha\in(0,1)$ and $\bx=(x_1,\ldots,x_K)\in\R^K,$ we call $\alpha-$\textit{quantile} of $\bx$ any element $Q_\alpha[\bx]$ of the set
\begin{align}\label{def.quantile_alpha}
    \mQ_\alpha[\bx] := \Big\{ u\in\R:\
    &\big|\{k = 1, \dots, K: x_k\geq u\} \big| \geq (1-\alpha)K, \nonumber \\
    &\text{and } \big|\{k = 1, \dots, K: x_k\leq u\} \big| \geq \alpha K \Big\}.
\end{align}
This means that $Q_\alpha[\bx]$ is a $\alpha-$\textit{quantile} of $\bx$ if at least $(1-\alpha)K$ components of $\bx$ are bigger than $Q_\alpha[\bx]$ and at least $\alpha K$ components of $\bx$ are smaller than $Q_\alpha[\bx].$ For all $t\in\R,$ we write $Q_\alpha[\bx]\geq t$ when there exists $J\subset[K]$ such that $|J|\geq(1-\alpha)K$ and, for all $k\in J,$ $x_k\geq t.$ We write $Q_\alpha[\bx]\leq t$ if there exists $J\subset[K]$ such that $|J|\geq \alpha K$ and, for all $k\in J,$ $x_k\leq t.$

\textbf{STEP 1. Partition of the dataset.} \\ 
Let $K \in \N$ be a fixed positive integer. Partition the dataset $\mD = \{(\bX_i,Y_i):i=1,\ldots,n\}$ into $K$ blocks $\mD_1,\ldots,\mD_K$ of size $n/K$ (assumed to be an integer). This corresponds to a partition of $\{1.\ldots,n\}$ into blocks $B_1,\ldots,B_K.$ 

\textbf{STEP 2. Local criterion.} \\
With $c>1$ and $f,g\in\mF,$ $\sigma,\chi\in\R_{+},$ define the functional
\begin{align}\label{def.R_functional}
    R_c(\ell_g,\chi,\ell_f,\sigma) := (\sigma-\chi)\bigg(1-2\frac{\ell_f+\ell_g}{(\sigma+\chi)^2} \bigg) + 2c\frac{\ell_f-\ell_g}{\sigma+\chi}.
\end{align}
Since $\ell_f(\bx,y) = (y-f(\bx))^2$ for all $(\bx,y)\in\mX\times\R,$ the latter definition induces the map $(\bx,y) \mapsto R_c(\ell_g(\bx,y),\chi,\ell_f(\bx,y),\sigma)$ over $(\bx,y)\in\mX\times\R.$ For each $k=[K],$ we define the \textit{criterion of} $(f,\sigma)$ \textit{against} $(g,\chi)$ \textit{on the block} $B_k$ as the empirical mean of the functional $R_c(\ell_g,\chi,\ell_f,\sigma)$ on that block, that is,
\begin{align}\label{def.local_crit}
    \P_{B_k} \Big( R_c(\ell_g,\chi,\ell_f,\sigma) \Big)
    := \frac{1}{|B_k|} \sum_{i \in B_k} 
    R_c\Big(\ell_g(\bX_i,Y_i),\chi,\ell_f(\bX_i,Y_i),\sigma\Big),
\end{align}
for all $(g,\chi,f,\sigma)\in\mF\times\R_{+}\times\mF\times\R_{+}.$ Here $|B_k| = n/K$ denotes the cardinality of $B_k.$

\textbf{STEP 3. Global criterion.} \\
For any $\alpha\in(0,1)$ and number of blocks $K,$ set
\begin{align*}
    Q_{\alpha,K} \Big[ R_c(\ell_g,\chi,\ell_f,\sigma) \Big] := Q_\alpha\Big[ \Big( \P_{B_k} \big( R_c(\ell_g,\chi,\ell_f,\sigma) \big)\Big)_{k\in[K]} \Big],
\end{align*}
the $\alpha-$quantile of the vector of local criteria defined in the previous step. For $\alpha=1/2$ we get the \textit{median}. We define the \textit{global criterion of} $(f,\sigma)$ \textit{against} $(g,\chi)$ as
\begin{align}\label{def.global_crit}
    MOM_K\Big( R_c(\ell_g,\chi,\ell_f,\sigma) \Big)
    := Q_{1/2,K} \Big[ R_c(\ell_g,\chi,\ell_f,\sigma) \Big],
\end{align}
for all $(g,\chi,f,\sigma)\in\mF\times\R_{+}\times\mF\times\R_{+}.$ With some norm $\|\cdot\|$ on the span of $\mF,$ we denote 
\begin{align}\label{def.T_functional}
     T_{K, \mu}(g, \chi, f, \sigma) := MOM_K\Big( R_c(\ell_g,\chi,\ell_f,\sigma) \Big) + \mu (\|f\| - \|g\|),
\end{align}
where $\mu>0$ is a tuning parameter, the functional $T_{K, \mu}$ is the penalized version of the global criterion.

\textbf{STEP 4. MOM estimator.} \\
With $\sigma_{+}$ the known upper bound in~\eqref{def.zeta_moments}, we define the \textit{MOM$-K$} estimator of $(f^*,\sigma^*)$ as
\begin{align}\label{def.MOM_est}
    (\widehat f_{K,\mu,\sigma_+}, \widehat\sigma_{K,\mu,\sigma_+})
    := \argmin_{f \in \mF,\ \sigma \leq \sigma_{+}} \,
    \maxnew_{g \in \mF,\ \chi \leq \sigma_{+}} T_{K, \mu}(g, \chi, f, \sigma),
\end{align}
where $T_{K, \mu}$ is the penalized functional in~\eqref{def.T_functional}. Furthermore, set 
\begin{align}\label{def.C_functional}
     \mC_{K,\mu} (f, \sigma) := \maxnew_{g \in \mF,\ \chi \leq \sigma_{+}} T_{K, \mu}(g, \chi, f, \sigma).
\end{align}
The estimator $(\widehat f_{K,\mu,\sigma_+},\widehat\sigma_{K,\mu,\sigma_+})$ only depends on the upper bound $\sigma_{+},$ the number $K$ of blocks and the tuning parameter $\mu.$

\section{Results for a general class \texorpdfstring{$\mF$}{F}} \label{sec.general_F}

We assume the following regularity condition on the function class $\mF$ and the inliers.
\begin{assump}\label{ass.main}
    There exist constants $\theta_0, \theta_1 > 1$ such that, for all $i \in \mI$ and $f \in \mF,$
    \begin{enumerate}
        \item $\|f-f^*\|_{2,\bX}^2 = \E[(f-f^*)^2(\bX_i)] \leq \theta_0^2 \E[|f-f^*|(\bX_i)]^2 = \theta_0^2 \|f - f^*\|_{1,\bX}^2.$
        \item $\|f-f^*\|_{4,\bX}^2 = \E[(f-f^*)^4(\bX_i)]^{1/2} \leq \theta_1^2 \E[(f-f^*)^2(\bX_i)] = \theta_1^2 \|f - f^*\|_{2,\bX}^2.$
    \end{enumerate}
\end{assump}
This assumption guarantees that the $L^1(\P_{\bX}),L^2(\P_{\bX}),L^4(\P_{\bX})-$norms are equivalent on the function space $\mF-f^*.$ The equivalence between $\|\cdot\|_{1,\bX}$ and $\|\cdot\|_{2,\bX}$ in the first condition matches Assumption 3 in~\cite{lecue2020robustML}. The equivalence between $\|\cdot\|_{2,\bX}$ and $\|\cdot\|_{4,\bX}$ in the second condition, together with the finiteness of fourth moment of the noise in Assumption~\ref{ass.zeta_moments}, helps controlling the dependence between $\zeta$ and $\bX;$ this also matches Assumption~3.1 in~\cite{lugosi2017regularization}. We do not necessarily assume that $\zeta$ is independent of $\bX,$ but the Cauchy-Schwarz inequality gives
\begin{align*}
    \|\zeta(f-f^*)\|_{2,\bX}^2 &= \E[\zeta^2(f-f^*)^2(\bX)] \\
    &\leq \E[\zeta^4]^{\frac{1}{2}} \E[(f-f^*)^4(\bX)]^{\frac{1}{2}} \\
    &\leq \theta_1^2 \frakm^{*2} \E[(f-f^*)^2(\bX)].
\end{align*}
The bound $\|\zeta(f-f^*)\|_{2,\bX}^2 \leq \theta_1^2 \frakm^{*2} \|f-f^*\|_{2,\bX}^2$ is Assumption 2 in~\cite{lecue2020robustML} with $\theta_m^2 = \theta_1^2 \frakm^{*2},$ whereas in our setting this is a consequence of Assumption~\ref{ass.zeta_moments} and Assumption~\ref{ass.main}.

\subsection{Complexity parameters} \label{sec.main_complexity}

With the introduction of MOM tournaments procedures, see \cite{lugosi2017regularization} and references therein, the authors have characterized the underlying geometric features that drive the performance of a learning method. For any $\rho>0,r>0,$ and $f\in\mF,$ we set
\begin{align*}
\begin{split}
    \B(f,\rho) := \big\{g\in\mF:\|g-f\|\leq\rho \big\},\quad \B_2(f,r) := \big\{g\in\mF:\|g-f\|_{2,\bX}\leq r \big\},
\end{split}
\end{align*}
respectively the $\|\cdot\|-$ball of radius $\rho$ and the $\|\cdot\|_{2,\bX}-$ball of radius $r,$ both centered around $f\in\mF.$ We denote by $\B(\rho)$ and $\B_2(r)$ the balls centered around zero. We define the \textit{regular ball around} $f^*$ of radii $\rho>0,r>0$ as
\begin{align*}
    \B(f^*,\rho,r) := \{f\in\mF: \|f-f^*\|\leq\rho,\ \|f-f^*\|_{2,\bX}\leq r \}.
\end{align*}
For any subset of inlier indexes $J\subseteq\mI,$ we define the \textit{standard empirical process on} $J$ as
\begin{align*}
    f \mapsto \P_J(f-f^*) := \frac{1}{|J|}\sum_{i\in J} (f-f^*)(\bX_i).
\end{align*}
Similarly, we define the \textit{quadratic empirical process on} $J$ and the \textit{multiplier empirical process on} $J$ as 
\begin{align*}
    f \mapsto \P_J\left((f-f^*)^2\right) &:= \frac{1}{|J|}\sum_{i\in J} (f-f^*)^2(\bX_i), \\
    f \mapsto \P_J\left(-2\zeta(f-f^*)\right) &:= -\frac{2}{|J|}\sum_{i\in J} \zeta_i (f-f^*)(\bX_i),
\end{align*}
where $\zeta_i = (Y_i-f^*(\bX_i)).$ These processes arise naturally when dealing with the \textit{empirical excess risk on} $J,$ which is
\begin{align*}
    \Risk_J(f) - \Risk_J(f^*) :&= \frac{1}{|J|}\sum_{i\in J} (Y_i - f(\bX_i))^2 - \frac{1}{|J|}\sum_{i\in J} (Y_i - f^*(\bX_i))^2 \\
    &= \frac{1}{|J|}\sum_{i\in J} (f-f^*)^2(\bX_i) -\frac{2}{|J|}\sum_{i\in J} \zeta_i (f-f^*)(\bX_i) \\
    &= \P_J\left((f-f^*)^2\right) + \P_J\left(-2\zeta(f-f^*)\right).
\end{align*}
The empirical processes defined above only involve observations that are not contaminated by outliers and we are interested in controlling them when the indexing function class is a regular ball $\B(f^*,\rho,r).$ 

Let $\xi_i$ be \textit{Rademacher variables}, that is, independent random variables uniformly distributed on $\{-1,1\},$ and independent from the dataset $\mD$. For any $r > 0$ and $\rho > 0$, consider the regular ball $\B(f^*, \rho, r)$ defined above. For every $\gamma_P, \gamma_Q, \gamma_M> 0,$ we define the \textit{complexity parameters}
\begin{align}
\begin{split}\label{def.complexity}
    r_P(\rho, \gamma_P) 
    &:= \inf \bigg\{ r > 0 : \sup_{ J \subset \mI, |J| \geq \frac{n}{2}} \E\bigg[ \sup_{f \in \B(f^*, \rho, r)}
    \Big|\frac{1}{|J|} \sum_{i \in J} \xi_i (f-f^*)(\bX_i)  \Big| \bigg] \leq \gamma_P r \bigg\}, \\
    r_Q(\rho, \gamma_Q) 
    &:= \inf \bigg\{ r > 0 : \sup_{ J \subset \mI, |J| \geq \frac{n}{2}} \E\bigg[ \sup_{f \in \B(f^*, \rho, r)}
    \Big|\frac{1}{|J|} \sum_{i \in J} \xi_i (f-f^*)^2(\bX_i)  \Big| \bigg] \leq \gamma_Q r^2 \bigg\}, \\
    r_M(\rho, \gamma_M) 
    &:= \inf \bigg\{ r > 0 : \sup_{ J \subset \mI, |J| \geq \frac{n}{2}} \E\bigg[ \sup_{f \in \B(f^*, \rho, r)}
    \Big|\frac{1}{|J|} \sum_{i \in J} \xi_i \zeta_i (f-f^*)(\bX_i)  \Big| \bigg] \leq \gamma_M r^2 \bigg\},
\end{split}
\end{align}
and let $r=r(\cdot, \gamma_P, \gamma_M)$ be a continuous non-decreasing function $r:\R_{+}\to\R_{+}$ depending on $\gamma_P, \gamma_M,$ such that
\begin{align}\label{def.r_complexity}
    r(\rho) \geq \max \big\{ r_P(\rho, \gamma_P), r_M(\rho, \gamma_M) \big\},
\end{align}
for every $\rho > 0.$ The definitions above depend on $f^*$ and require that $|\mI|\geq n/2$.  The function $r(\cdot)$ matches the one defined in Definition~3 in~\cite{lecue2020robustML}. We refer to Section~\ref{sec.discussion} for a detailed discussion on the role of complexity parameters, here we only mention that in the sub-Gaussian setting of \cite{lecue2013learning}, for some choice of $\gamma_P, \gamma_M,$ the quantity $r^*(\rho) = \max \{ r_P(\rho, \gamma_P), r_M(\rho, \gamma_M) \}$ is the minimax convergence rate over the function class $\B(f^*,\rho).$

\subsection{Sparsity equation} \label{sec.main_sparsity}

We follow the setup of \cite{lecue2020robustML}, that we restate here for convenience.

\textbf{Subdifferential.} Let $\mE$ be the vector space generated by $\mF$ and $\|\cdot\|$ a norm on $\mE.$ We denote by $(\mE^*, \|\cdot\|_*)$ the dual normed space of $(\mE, \|\cdot\|),$ that is, the space of all linear functionals $z^*$ from $\mE$ to $\R.$ The \textit{subdifferential} of $\|\cdot\|$ at any $f \in \mF$ is denoted by 
\begin{align*}
    (\partial\|\cdot\| )_f
    := \{ z^* \in \mE^* : \|f+h\| \geq \|f\| + z^*(h), \,
    \forall h \in \mE\}.
\end{align*}
The penalization term of the functional $T_{K,\mu}$ in Section~\ref{sec.estimator} is of the form $\mu(\|f\|-\|g\|),$ for $f,g\in\mF,$ and the subdifferential is useful in obtaining lower bounds for $\|f\|-\|f^*\|.$ For any $\rho > 0$ and complexity parameter $r(\rho)$ as in \eqref{def.r_complexity}, we denote $H_\rho = \{f\in\mF: \|f-f^*\| = \rho,\ \|f-f^*\|_{2,\bX} \leq r(\rho)\}.$ Furthermore, we set
\begin{align} 
\begin{split} \label{def.aux}
    \Gamma_{f^*}(\rho)
    &:= \bigcup_{f \in \mF: \, \|f-f^*\| \leq \rho/20}
    \big( \partial \|\cdot\| \big)_f, \\
    \Delta(\rho) &:= \infnew_{f \in H_\rho} \,
    \sup_{z^* \in \Gamma_{f^*}(\rho)} z^* (f - f^*).
\end{split}
\end{align}
The set $\Gamma_{f^*}(\rho)$ is the set of subdifferentials of all functions that are close to $f^*$ (no more than $\rho/20$) in penalization norm $\|\cdot\|.$ The quantity $\Delta(\rho)$ measures the smallest level $\Delta>0$ for which the chain $\|f\|-\|f^*\| \geq \Delta - \rho/20$ holds. In fact, if $f^{**}\in\mF$ is such that $\|f^*-f^{**}\|\leq\rho/20,$ then $\|f\|-\|f^*\| \geq \|f\| - \|f^{**}\| - \|f^{**}-f^*\| \geq z^* (f - f^{**}) - \rho/20,$ for any subdifferential $z^*\in (\partial \|\cdot\|)_{f^{**}}.$

\textbf{Sparsity equation.} The \textit{sparsity equation} and its smallest solution are
\begin{align}\label{def.sparsity_eq}
    \Delta(\rho) \geq \frac{4}{5}\rho,\quad \rho^* := \inf \left\{\rho>0: \Delta(\rho)\geq \frac{4\rho}{5} \right\},
\end{align}
if $\rho^*$ exists, the sparsity equation holds for any $\rho\geq\rho^*.$

\subsection{Main result in the general case}

We now present a result dealing with the simultaneous estimation of $(f^*,\sigma^*)$ by means of a family of MOM estimators constructed as in Section~\ref{sec.estimator}. Fix any constant $c>2$ in the definition on the functional $R_c$ in~\eqref{def.R_functional} and, with $\sigma_{+},\frakm_{+},\kappa_{+}$ the known bounds on the moments of the noise $\zeta=Y-f^*(\bX),$ set
\begin{align}
\begin{split}\label{def.constants}
    c_\mu &:= 200 (c+2) \kappa_{+}^{1/2}, \\
    \eps &:= \frac{c-2}{192 \, \theta_0^2 (c+2) \big(8 + 134 \, \kappa_{+}^{1/2} ((1+\frac{\sigma_{+}}{\sigma^*})\vee \frac{6}{5}) \big)}, \\
    c_\alpha^2 &:= \frac{3 (c-2)}{5 \theta_0^2},
\end{split}
\end{align}
and $\gamma_P = 1/(1488 \, \theta_0^2),$ $\gamma_M = \eps/744$ and $\gamma_Q = \eps/372.$ Let $\rho^*$ be the smallest solution of the sparsity equation in~\eqref{def.sparsity_eq} and $r(\cdot)$ any function such that $r(\rho) \geq \max\{r_P(\rho,\gamma_P),r_M(\rho,\gamma_M)\}$ as in~\eqref{def.r_complexity}.  Define $K^*$ as the smallest integer satisfying
\begin{align}\label{def.K*}
    K^* \geq \frac{n \eps^2 r^2(\rho^*)}{384 \, \theta_1^2\frakm^{*2}},
\end{align}
and, for any integer $K\geq K^*,$ also define $\rho_K$ as the implicit solution of
\begin{align}\label{eq.rhoK}
    r^2(\rho_K) = \frac{384 \, \theta_1^2 \frakm^{*2} K}{n \eps^2}.
\end{align}
\begin{assump}\label{ass.r2rho_rrho}
    We assume that there exists an absolute constant $c_r\geq1$ such that, for all $\rho>0,$ we have $r(\rho) \leq r(2\rho) \leq c_r r(\rho).$
\end{assump}
The role of the latter assumption is to simplify the statement of the main result. We are mainly interested in the sparse linear case, where this holds with $c_r = 2$ by construction of the function $r(\cdot),$ see Section~\ref{sec.disc_complexity_sparse}.

\begin{thm}\label{thm.main_theorem}
    With the notation above, let Assumptions~\ref{ass.zeta_moments}--\ref{ass.main} and Assumption~\ref{ass.r2rho_rrho} hold. With $C^2 := 384 \, \theta_1^2 c_r^2 c_\alpha^2 \kappa_{+}^{1/2},$ suppose that $n \eps^2 > 32 C^2$ and $|\mO| \leq n \eps^2 /(32C^2).$
    Then, for any integer $K \in \left[K^* \vee 32|\mO|,\ n\eps^2/C^2 \right],$
    and for every $\iota_\mu \in [1/4, 4],$ the MOM$-K$ estimator $(\widehat f_{K,\mu,\sigma_+}, \widehat\sigma_{K,\mu,\sigma_+})$ defined in \eqref{def.MOM_est} with $K$ blocks and penalization parameter
    \begin{align}\label{def.mu}
        \mu := \iota_\mu c_\mu \eps \frac{r^2(\rho_K)}{\frakm^* \rho_K},
    \end{align}
    satisfies, with probability at least $1 - 4 \exp( - K / 8920 ),$ for any possible $|\mO|$ outliers,
    \begin{align} 
        \|\widehat f_{K,\mu,\sigma_+} - f^* \|
        &\leq 2 \, \rho_K ,\quad
        \|\widehat f_{K,\mu,\sigma_+} - f^*\|_{2,\bX}
        \leq r(2 \rho_K)  ,\quad
        |\widehat\sigma_{K,\mu,\sigma_+} - \sigma^*|
        \leq c_\alpha r(2\rho_K), \label{eq.bound_thm} \\
        \begin{split}
            R(\widehat f_{K,\mu,\sigma_+})
            &\leq R(f^*) + \bigg(2 + 2 c_\alpha +
            \left(44 + 5 c_\mu \right)\eps + \frac{25 \kappa^{*1/2}}{8 \theta_1^2}\eps^2\bigg) r^2(2\rho_K) \\
            &\quad + 4 \, \theta_1^2\eps \left(r^2(2\rho_K) \vee r_Q^2(2\rho_K,\gamma_Q) \right).
            \label{eq.bound_excess_risk}
        \end{split}
    \end{align}
    
\end{thm}

The proof of Theorem~\ref{thm.main_theorem} is given in Appendix~\ref{sec.proofs_main}. It provides theoretical guarantees for the MOM$-K$ estimator $(\widehat f_{K,\mu,\sigma_+}, \widehat\sigma_{K,\mu,\sigma_+})$: this estimator recovers $(f^*,\sigma^*),$ with high probability, whenever the number $K$ of blocks is chosen to be at least $K^* \vee 32|\mO|$ and at most $n\eps^2/C^2.$
Specifically, the random function $\widehat f_{K,\mu,\sigma_+}$ belongs to the regular ball $\B(f^*,2\rho_K,r(2\rho_K)),$ whereas the random standard deviation $\widehat\sigma_{K,\mu,\sigma_+}$ is at most $c_\alpha r(2\rho_K)$ away from $\sigma^*.$ The best achievable rates are obtained for $K=K^*$ when $|\mO| \leq K^*/32.$ Any estimator $(\widehat f_{K,\mu,\sigma_+}, \widehat\sigma_{K,\mu,\sigma_+})$ only depends on the penalization parameter $\mu,$ the number of blocks $K$ and the upper bound $\sigma_+,$ thus the result is mainly of interest when these quantities can be chosen without knowledge of $(f^*,\sigma^*).$ Our Theorem~\ref{thm.main_theorem} extends the scope of Theorem~1 in~\cite{lecue2020robustML} to the case of unknown noise variance. In the latter reference, the authors obtain the same convergence rates for a MOM$-K$ estimator $\widehat f_{K,\lambda}$ defined by using a penalization parameter $\lambda$ that we compare to our $\mu,$
\begin{align*}
    \lambda := 16 \eps \frac{r^2(\rho_K)}{\rho_K},\quad \mu := c_{\mu} \eps \frac{r^2(\rho_K)}{\frakm^* \rho_K},
\end{align*}
so that $\mu$ is proportional to $\lambda/\frakm^*.$ For the sparse linear case, \cite{lecue2020robustML} shows that the optimal choice is $\lambda \sim \frakm^* \sqrt{\log(ed/s^*)/n},$ which is proportional to the noise level $\sigma^*.$ This in turn guarantees that our penalization parameter can be chosen of the form $\mu \sim \sqrt{\log(ed/s^*)/n}$ to obtain the optimal rates, and that such a choice does not depend on the moments of the noise.

\section{The high-dimensional sparse linear regression} \label{sec.high_dim_reg}

\subsection{Results for known sparsity}

In this section, we will give non-asymptotic bounds that will hold adaptively and uniformly over a certain class of joint distributions for $(\bX, \zeta)$. We now define the class of interest $\mP_I$, parametrized by an interval $I$. This interval $I$ represents the set of possible values for the standard deviation $\sigma^*$ of the noise $\zeta$.

\begin{defi}[Class of distributions of interest] \label{def.distr_class}
    For $I \subset \R_+$, $\theta_0, \theta_1, c_0, L, \kappa_+ > 1$,
    let us define $\mP_I = \mP_I(\theta_0, \theta_1, c_0, L, \kappa_+)$ to be the class of distributions $P_{\bX,\zeta}$ on $\R^{d+1}$ satisfying:
    \begin{enumerate}
        \item The standard deviation $\sigma^*$ of $\zeta$ belongs to $I$ and the kurtosis of $\zeta$ is smaller than $\kappa_+$.
        
        \item For all $\bbeta \in \R^d$,
        $\E\big[(\bX^\top \bbeta)^2 \big]^{\frac{1}{2}}
        \leq \theta_0 \E\big[|\bX^\top \bbeta|\big],$ and
        $\E\big[(\bX^\top \bbeta)^4\big]^{\frac{1}{2}} \leq \theta_1^2 \E\big[(\bX^\top \bbeta)^2\big]$.
        
        \item $\bX$ is isotropic: for all $\bbeta\in\R^d,$
        $\|f_{\bbeta}\|_{2,\bX} := \E[(\bX^\top \bbeta)^2] = |\bbeta|_2,$ where $f_{\bbeta}(\bx) = \bx^\top\bbeta$.
        
        \item $\bX$ satisfies the weak moment condition: for all $1 \leq p \leq c_0 \log(ed),$ $1 \leq j \leq d,$
        $\E\big[|\bX^\top \be_j|^p \big]^{\frac{1}{p}}
        \leq L \sqrt{p} \, \E\big[|\bX^\top \be_j|^2 \big]^{\frac{1}{2}}.$
    \end{enumerate}
\end{defi}

The class $\mP_I$ only requires a finite fourth moment on $\zeta$, allowing it to follow heavy-tailed distributions. The weak moment condition only bounds moments of $\bX$ up to the order $\log(d)$, which is weaker than the sub-Gaussian assumption, see~\cite{lecue2013learning} and the references therein for a discussion and a list of examples.

\begin{defi}[Contaminated datasets]
    For a dataset $\mD = (\bx_i, y_i)_{i=1,\dots,n} \in \R^{(d+1) \times n}$ and for $N \in [n]$, we denote by $\mD(N)$ the set of all datasets $\mD' = (\bx'_i, y'_i)_{i=1,\dots,n} \in \R^{(d+1) \times n}$ that differ from $\mD$ by at most $N$ observations, i.e.
    \begin{align*}
        \mD(N) &:= \left\{\mD' \in \R^{(d+1) \times n} : \left|\mD \setminus \mD' \right| \leq N  \right\},
    \end{align*}
    where $\mD \setminus \mD'$ is defined as the difference between the (multi-)sets $\mD$ and $\mD'$, meaning that if there exists duplicated observations in $\mD$ that appear also in $\mD'$, they are removed from $\mD$ up to their multiplicities in $\mD'$.
    This encodes all the possible corrupted versions of $\mD$ by means of up to $N$ arbitrary outliers. 
\end{defi}

\begin{defi}
    Let $P_{\bbeta^*,P_{\bX,\zeta}}$ be the distribution of $(\bX,Y)$ when $(\bX, \zeta) \sim P_{\bX,\zeta}$ and $Y := \bX^\top \bbeta^* + \zeta$.
    
    In the following, we will use the minimax optimal rates of convergence defined for $p \in [1,2]$ by $\frakr_p := {s^*}^{1/p} \sqrt{(1/n) \log(ed/s^*)}$ and the allowed maximum number of outliers defined by $\frakr_\mO := s^*\log(ed/s^*) = n \frakr_2^2$.
\end{defi}

\begin{thm} \label{thm.rates_AMOM_sigmaplus}
    Assume that $\frakr_2 < 1$. For every $\theta_0, \theta_1, c_0, L, \kappa_+ > 1$, there exists universal constants $\widetilde c_1, \dots, \widetilde c_5 > 0$ such that for every $\sigma_+$ and for every $\iota_K, \iota_\mu \in [1/2, 2]^2$, setting
    \begin{align*}
        K = \lceil \iota_K \widetilde c_1 s^* \log(ed/s^*) \rceil , \quad
        \mu = \iota_\mu \widetilde c_2 \sqrt{\frac{1}{n} \log\left(\frac{e d}{s^*} \right)},
    \end{align*}
    the estimator $(\widehat\bbeta_{K,\mu,\sigma_+}, \widehat\sigma_{K,\mu,\sigma_+})$ satisfies
    \begin{align*}
        \inf_{ \scaleto{\begin{array}{c}
        P_{\bX,\zeta} \in \mP_{[0, \, \sigma_+]} \\
        \bbeta^* \in \mF_{s^*}
        \end{array} }{25pt} } \hspace{-0.3cm}
        \P_{\mD \sim P_{\beta^*,P_{\bX,\zeta}}^{\otimes n}}
        \Bigg( \hspace{0cm}
        &\sup_{
        \mD' \in \mD(\widetilde c_3 \frakr_\mO) }
        \hspace{0cm}
        \bigg\{ \frakr_2^{-1}
        \big| \widehat\sigma (\mD') - \sigma^* \big| \\
        &\hspace{2cm} \vee \hspace{-0.1cm}
        \sup_{p \in [1,2]} \frakr_p^{-1}
        \big| \widehat\bbeta (\mD') - \bbeta^* \big|_p \bigg\}
        \leq \widetilde c_4 \sigma_+ \Bigg)
        \geq 1 - 4 \Big( \frac{s^*}{ed} \Big)^{\widetilde c_5 s^*},
    \end{align*}
\end{thm}

This theorem is proved in Section~\ref{proof:thm.rates_AMOM_sigmaplus}.
Theorem~\ref{thm.rates_AMOM_sigmaplus} ensures that, with high probability, the estimator $(\widehat\bbeta_{K,\mu,\sigma_+}, \widehat\sigma_{K,\mu,\sigma_+})$ achieves the rates 
$| \widehat\bbeta - \bbeta^* |_p \lesssim \sigma_+ {s^*}^{1/p} \sqrt{(1/n) \log(ed/s^*)}$ and
$| \widehat\sigma - \sigma^* | \lesssim \sigma_+ \sqrt{(s^*/n) \log(ed/s^*)}$, uniformly over the class of distributions $\mP_{[0, \, \sigma_+]}$ with bounded variance while being robust to up to $\widetilde c_3 s^* \log(ed/s^*)$ arbitrary outliers. 
However, the uniform constants appearing in the statement might be difficult to compute in practice, to obtain precise values, one would need to quantify the constants in Theorem~1.6 in~\cite{mendelson2017multiplier} and Lemma~5.3 in~\cite{lecue2018regularization}.
As usual for MOM estimators, the maximum number of outliers is of the same order as the number of blocks.
Note that the estimator needs the knowledge of an upper bound on the noise level $\sigma_+$ and the sparsity level $s$.

In \cite{bellec2018}, it has been proved that the optimal minimax rate of estimation of $\bbeta^*$ in the $|\cdot|_p$ norm is $\sigma^* \sqrt{(s^*/n) \log(ed/s^*)}$ when $\sigma^*$ is fixed and the noise is sub-Gaussian. Our theorem shows that the rate of estimation of $\bbeta$ over $\mP_{[0, \, \sigma_+]}$ is the optimal minimax rate of estimation for the worst-case noise level $\sigma_+$.
In particular, this means that in the noiseless case when $\sigma^* = 0$, the estimator $\widehat\bbeta_{K,\mu,\sigma_+}$ does not achieve perfect reconstruction of the signal $\bbeta^*$. This is worse than the square-root Lasso \cite{derumigny2018improved} which achieves the minimax optimal rate $| \widehat\bbeta^{SR\text{-}Lasso} - \bbeta^* |_p \lesssim \sigma^* {s^*}^{1/p} \sqrt{(1/n) \log(ed/s^*)}$ adaptively over $\sigma^* \in \R_+$.
However, the square-root Lasso is not robust to even one outlier in the dataset. Furthermore, this optimal rate for the square-root Lasso has only been proved for sub-Gaussian noise $\zeta$ whereas in Theorem~\ref{thm.rates_AMOM_sigmaplus}, we allow for any distribution of $\zeta$ with finite fourth moment.
The MOM-Lasso \cite{lecue2020robustML} achieves the optimal rate $| \widehat\bbeta^{MOM-Lasso} - \bbeta^* |_p \lesssim \sigma^* {s^*}^{1/p} \sqrt{(1/n) \log(ed/s^*)}$, but needs the knowledge of $\sigma^*$. Therefore, this bound can uniformly hold only on a class of the form $\mP_{[C_1 \sigma^*, C_2 \sigma^*]}$ for some fixed $0 < C_1 \leq C_2$.

To our knowledge, the estimator $\widehat \sigma$ is the first estimator of $\sigma^*$ that achieves robustness. Its rate of estimation $\sqrt{(s^*/n) \log(ed/s^*)}$ is slower than the parametric rate $1/\sqrt{n}$ that one would get if $\beta^*$ was known.
Theorem 5 in~\cite{comminges2018adaptive} suggests that this rate $\frakr_2$ might be minimax as well: the authors show that, albeit in a Gaussian sequence model, the factor $\sqrt{s^* \log(ed/s^*)}$ arises naturally in the estimation of $\sigma^*$ by means of any adaptive procedure in a setting where the distribution of the noise $\zeta$ is unknown. Even in the case where no outliers are present, we improve on the best known bound on the estimation of $\sigma^*$,~\cite[Corollary 2]{belloni2014pivotal} which was 
$\big| (\hat \sigma^{SR\text{-}Lasso})^2 - \sigma^2 \big|
\lesssim \sigma^2 \Big( \frac{s^* \log(n \vee d \log n)}{n}
+ \sqrt{\frac{s^* \log(d \, \vee \, n)}{n}} + \frac{1}{\sqrt{n}} \Big)$.

\begin{rem}
    When $\bbeta^*$ is not sparse but very close to a sparse vector (i.e. $|\bbeta^*-\bbeta^{**}|_1 \lesssim \sigma^* \sqrt{s^* \log(ed/s^*) /n }$ for a sparse vector $\bbeta^{**} \in \mF_{s^*}$, the complexity parameter $r(\rho)$ is in fact unchanged compared to the sparse case and the upper bounds on the rates of estimation $| \widehat\bbeta - \bbeta^* |_p \lesssim \sigma_+ {s^*}{1/p} \sqrt{(1/n) \log(ed/s^*)}$ and $| \widehat\sigma - \sigma^* | \lesssim \sigma_+ \sqrt{(s^*/n) \log(ed/s^*)}$ still hold, extending Theorem~\ref{thm.rates_AMOM_sigmaplus}.
\end{rem}

In practice, it may not be obvious to choose what a good value for $\sigma_+$ could be. This means that the (unknown) distribution belongs in fact to the class $\mP_{[0, + \infty]} = \bigcup_{\sigma_+ > 0} \mP_{[0, \sigma_+]}$. A natural idea is to cut the data into two parts. On the first half of the data, we estimate the variance $\Var[Y]$ by the MOM estimator
$\widehat\sigma_{K,+}^2 := Q_{1/2,K}\left[ Y^2 \right] - \left(Q_{1/2,K}\left[ Y \right]\right)^2$.
On the second half of the data, we use $\widehat\sigma_{K,+}$ as the ``known'' upper bound $\sigma_+$ and apply our algorithm as defined in Equation~\eqref{def.MOM_est}.
The following corollary, proved in Section~\ref{proof:cor.AMOM_estiSigma}, gives a bound on the performance of this estimator on the larger class $\mP_{[0, \, +\infty]}$.

\begin{cor}[Performance of the estimator with estimated $\sigma_+$ on $\mP_{[0, \, +\infty]}$]
    \label{cor.AMOM_estiSigma}
    Let $s^*>0$. 
    Then, for every $P_{\bX,\zeta} \in \mP_{[0, \, +\infty]}$ and $\bbeta^* \in\mF_{s^*}$, there exists a constant $C > 0$ such that, for any $n > C s^* \log(p/s^*)$
    the estimator $(\widehat\bbeta_{K,\mu,\widehat\sigma_{K,+}}, \widehat\sigma_{K,\mu,\widehat\sigma_{K,+}})$ satisfies
    \begin{align*}
        \P_{\mD \sim P_{\beta^*,P_{\bX,\zeta}}^{\otimes n}}
        \Bigg( 
        & \sup_{ \mD' \in \mD(\widetilde c_3 \frakr_\mO) }
        \bigg\{ \frakr_2^{-1}
        \big| \widehat\sigma (\mD') - \sigma^* \big|
        \vee
        \sup_{p \in [1,2]} \frakr_p^{-1}
        \big| \widehat\bbeta (\mD') - \bbeta^* \big|_p \bigg\} \\
        & \vspace{-0.2cm} \hspace{4cm}
        \leq 4 \, \widetilde c_4 \, \sqrt{1 + SNR} \, \sigma^* \Bigg)
        \geq 1 - 4 \Big( \frac{s^*}{ed} \Big)^{\widetilde c_5 s^*}
        - 2 \Big( \frac{s^*}{ed} \Big)^{\widetilde c_6 s^*},
    \end{align*}
    where $\widetilde c_6$ is a universal constant and $SNR$ denotes the signal-to-noise ratio, defined by
    $SNR := \Var[\bX^\top \bbeta^*] / \sigma^{*2}
    = \bbeta^*{}^\top \Var[X] \bbeta^* / \sigma^{*2}$.
\end{cor}

This corollary ensures that, with high probability, the estimator $(\widehat\bbeta_{K,\mu,\widehat\sigma_{K,+}}, \widehat\sigma_{K,\mu,\widehat\sigma_{K,+}})$ achieves the rates of estimation
$| \widehat\bbeta - \bbeta^* |_p \lesssim \sqrt{1 + SNR} \, \sigma^* {s^*}^{1/p} \sqrt{(1/n) \log(ed/s^*)}$ and
$| \widehat\sigma - \sigma^* | \lesssim \sqrt{1 + SNR} \, \sigma^* \sqrt{(s^*/n) \log(ed/s^*)}$. The factor $\sqrt{1 + SNR}$ describes how the estimation rates of $\bbeta^*$ and $\sigma^*$ are degraded as a function of the signal-to-noise ratio. Indeed, when the noise level is of the same order or higher than the standard deviation of $f^*(\bX)$, the rates are optimal. On the contrary, when the noise level is very small ($SNR \ll 1$), the rates of estimation are dominated by $\sqrt{\Var\big[\bX^\top \beta \big]} \frakr_p$.

\subsection{Adaptation to the unknown sparsity}

We now provide an adaptive to $s$ version of Theorem~\ref{thm.MOM_sqrt_lasso} by introducing an estimator $(\widetilde \beta, \widetilde \sigma, \widetilde s)$ that simultaneously estimates the vector of coefficients, the noise standard deviation and the sparsity level. This procedure is inspired by \cite[Section 4]{derumigny2018improved} that proposes a general Lepski-type method for constructing an adaptive to $s$ estimator from a sequence of estimators that attains the same rate for each value of $s$. 
This method is different from the one proposed in~\cite{lecue2020robustML} for making the MOM-LASSO estimator adaptive to the sparsity level $s$, which seems difficult to adapt for the case of unknown noise level. 

The main idea of this procedure is to compute different estimators for several possible sparsity levels. Starting from a sparsity of 2, we try different estimators by increasing each time the sparsity by a factor of 2 unless the difference between an estimator and the next one is too small. We choose this stopping value as the estimated sparsity level, and it gives directly an estimated number of blocks to use, since there exists an optimal number of blocks for each sparsity level.
More precisely, given a sparsity estimator $\widetilde s,$ we take $\widetilde K = \lceil \widetilde c_2 \widetilde s \log(ed/\widetilde s) \rceil.$ 

Given a known upper bound $s_+ \leq d$ on the sparsity, we define the sequence of MOM$-K$ estimators $(\widehat\bbeta_{(s),\sigma_+}, \widehat\sigma_{(s),\sigma_+})_{s=1,\ldots,s_{+}}$ by $\widehat\bbeta_{(s)} := \widehat\bbeta_{K_s,\mu_s,\sigma_+},$ $\widehat\sigma_{(s),\sigma_+} := \widehat\sigma_{K_s,\mu_s,\sigma_+}$ and 
\begin{align}\label{def.K_s_mu_s}
    K_s := \left\lceil \widetilde c_2 s \log\left(\frac{ed}{s}\right) \right\rceil,\quad 
    \mu_s := \widetilde c_\mu \sqrt{\frac{1}{n} \log\left(\frac{ed}{s} \right)}.
\end{align}
The adaptive procedure yields an estimator of the form $\widetilde s = 2^{\widetilde m}$ for some integer $\widetilde m \in \{1,\ldots,\lceil\log_2(s_{+})\rceil+1\},$ from which we get the simultaneous adaptive (to $s$ and $\sigma^*$) MOM estimator $(\widetilde \bbeta_{\sigma_+}, \widetilde\sigma_{\sigma_+}, \widetilde s_{\sigma_+}) = (\widehat\bbeta_{(\widetilde s),\sigma_+}, \widehat\sigma_{(\widetilde s),\sigma_+}, \widetilde s_{\sigma_+}).$ 

\textbf{Algorithm for adaptation to sparsity.} The steps of the adaptive procedure are as follows.
\begin{itemize} 
    \item Set $M := \lceil \log_2(s_+) \rceil.$
    
    \item For every $m\in\{1,\ldots,M+1\},$ compute
    $\displaystyle (\widehat\bbeta_{(2^m),\sigma_+}, \widehat\sigma_{(2^m)},\sigma_+) = \left(\widehat\bbeta_{K_{2^m},\mu_{2^m},\sigma_+}, \widehat\sigma_{K_{2^m},\mu_{2^m},\sigma_+} \right),$
    with $K_{2^m}$ and $\mu_{2^m}$ as defined in Equation~\eqref{def.K_s_mu_s}.
    
    \item For $u \in \{1, \dots, 2s_+\}$, let $\frakr_p(u) = u^{1/p} \sqrt{(1/n) \log(ed/u)}$ and
    \begin{align*}
        \mM := \bigg\{ & m\in\{1,\ldots,M\} : \, \text{for all $k\geq m$, }
        |\widehat\bbeta_{(2^{k-1})} - \widehat\bbeta_{(2^k)}|_1 \leq C_1 \widehat\sigma_{(2^{M+1})} \frakr_1(2^k)
        , \\
        &|\widehat\bbeta_{(2^{k-1})} - \widehat\bbeta_{(2^k)}|_2 \leq C_2 \widehat\sigma_{(2^{M+1})} \frakr_2(2^k)
        \text{ and }
        |\widehat\sigma_{(2^{k-1})} - \widehat\sigma_{(2^k)}| \leq C_3 \widehat\sigma_{(2^{M+1})} \frakr_2(2^k)
        \bigg\}.
    \end{align*}
    
    \item Set $\widetilde m := \min \mM,$ with the convention that $\widetilde m := M+1$ if $\mM = \emptyset.$ 
    
    \item Define $\widetilde s_{\sigma_+} := 2^{\widetilde m}$ and 
    $(\widetilde\bbeta_{\sigma_+}, \widetilde\sigma_{\sigma_+})
    := (\widehat\bbeta_{(\widetilde s),\sigma_+}, \widehat\sigma_{(\widetilde s),\sigma_+}).$
\end{itemize}

The following theorem is proved in Section~\ref{sec.proofs_thm_lasso_adaptive} and gives uniform bounds for the performance of the aggregated estimator $(\widetilde\bbeta_{\sigma_+}, \widetilde\sigma_{\sigma_+}, \widetilde s_{\sigma_+})$.

\begin{thm} \label{thm.MOM_sqrt_lasso_adaptive}
    Let $\theta_0, \theta_1, c_0, L, \kappa_+ > 1$.
    Let $s_+ \in \{1, \dots, d/(2e)\}$ and assume that ${\frakr_2(2s^+) < 1}$. Then the aggregated estimator $(\widetilde\bbeta_{\sigma_+}, \widetilde\sigma_{\sigma_+}, \widetilde s_{\sigma_+})$ satisfies
    \begin{align*}
        \inf_{\vphantom{2^{2^2}} s^* = 1, \dots, s_+} 
        \inf_{ \scaleto{\begin{array}{c}
        P_{\bX,\zeta} \in \mP_{[0, \, \sigma_+]} \\
        \bbeta^* \in \mF_{s^*}
        \end{array} }{25pt} } \hspace{-0.3cm}
        P_{\beta^*,P_{\bX,\zeta}}^{\otimes n}
        \Bigg( 
        &\sup_{ \mD' \in \mD(\widetilde c_3 \frakr_\mO) }
        \bigg\{ \frakr_2(s^*)^{-1}
        \big| \widehat\sigma (\mD') - \sigma^* \big|
        \vee \hspace{-0.1cm}
        \sup_{p \in [1,2]} \frakr_p(s^*)^{-1}
        \big| \widehat\bbeta (\mD') - \bbeta^* \big|_p \bigg\} \\
        & \hspace{2cm} \leq 4 \widetilde c_4 \sigma_+ \Bigg)
        \geq 1 - 4 (\log_2(s_{+})+1)^2 \bigg( \frac{2s_+}{ed} \bigg)^{2\widetilde c_5 s_+}
    \end{align*}
    and for all $\mD' \in \mD(\widetilde c_3 \frakr_\mO),$
    $\widetilde s_{\sigma_+}(\mD') \leq s^*$ on the same event.
\end{thm}

This theorem guarantees that for every $s^* \in \{1, \dots, s_+\}$, both estimators $\widetilde\bbeta$ and $\widetilde\sigma$ converge to their true values at the rate $\sigma_+ {s^*}^{1/p} \sqrt{(1/n)\log(ed/s^*)}$ as if the true sparsity level $s^*$ was known. However, the probability bounds are slightly deteriorated due to the knowledge of an upper bound $s_{+}$ only.

Note that the estimator presented above uses the knowledge of the upper bound on the standard deviation $\sigma_+$. If $\sigma_+$ is not available, the estimator presented in Corollary~\ref{cor.AMOM_estiSigma} can be aggregated in the same way. It will satisfy the same bounds up to some small degradation in the probability of the event.

\section{From the choice of the functional \texorpdfstring{$R_c$}{Rc} to empirical process bounds} \label{sec.discussion}

Our construction in Section~\ref{sec.estimator} produces a family of MOM estimators
\begin{align*}
    (\widehat f_{K,\mu\sigma_+}, \widehat\sigma_{K,\mu,\sigma_+})
    &= \argmin_{f \in \mF,\ \sigma \leq \sigma_{+}}
    \maxnew_{g \in \mF,\ \chi \leq \sigma_{+}}
    \left\{ MOM_K\Big( R_c(\ell_g,\chi,\ell_f,\sigma) \Big) + \mu\big(\|f\|-\|g\| \big) \right\},
\end{align*}
where $R_c$ is a carefully chosen functional in~\eqref{def.R_functional}. As mentioned in Section~\ref{sec.conv_conc}, this extends the scope of the MOM estimator in~\cite{lecue2020robustML}
\begin{align*}
    \widehat f_{K,\lambda}
    &= \argmin_{f \in \mF} \maxnew_{g \in \mF} \left\{MOM_K\big( R(\ell_g,\ell_f) \big) + \lambda\big(\|f\|-\|g\| \big) \right\},
\end{align*} 
where $R(\ell_g,\ell_f) = \ell_f - \ell_g,$ which was constructed in the setting of known $\sigma^*.$ In this section we discuss in detail the role of the functional $R_c.$ In Section~\ref{sec.disc_estimator} we motivate our choice by showing that, in the sparse linear setting, we recover a robust version of the square-root LASSO. In Section~\ref{sec.disc_conv_rate_risk} we lay down our proving strategy and highlight the contribution of $R_c$ in recovering convergence rates and excess risk bounds in terms of complexity parameters. In Section~\ref{sec.disc_complexity_subGauss} and Section~\ref{sec.disc_complexity_sparse} we reproduce the main results on complexity parameters in the sub-Gaussian and sparse linear case respectively.

\subsection{Adaptivity to \texorpdfstring{$\sigma^*$}{sigma*}: choice of the functional \texorpdfstring{$R_c$}{Rc} and corresponding conditions} \label{sec.disc_estimator}

Since we implement the same proving strategy as in~\cite{lecue2020robustML}, we introduce the following properties as natural assumptions that the functional $R_c$ should satisfy.
\begin{enumerate}[label=\textbf{P\arabic*.}, ref=P\arabic*]
    \item \textbf{Anti-symmetry.} For all $f,g\in\mF,$ $\chi,\sigma \in R_+$ and $(\bx,y) \in \mX\times\R,$ we have 
    \begin{align*}
        R_c\big(\ell_g(\bx,y),\chi,\ell_f(\bx,y),\sigma \big) = - R_c\big(\ell_f(\bx,y),\sigma,\ell_g(\bx,y),\chi \big),
    \end{align*}
    \label{prop.antysymmetry}
    in short, we write $R_c(\ell_g,\chi,\ell_f,\sigma) = - R_c(\ell_f,\sigma,\ell_g,\chi).$
\end{enumerate}
\unskip
The latter is a crucial requirement for the whole convex-concave procedure to work, as we show in the next section. It is automatically satisfied when $\sigma^*$ is known, since $R(\ell_g,\ell_f) = \ell_f - \ell_g = - R(\ell_f,\ell_g).$

\begin{enumerate}[label=\textbf{P\arabic*.}, ref=P\arabic*]
    \setcounter{enumi}{1}
    \item \textbf{Concavity in \bm{$\chi$}, given \bm{$f=g.$}} For any fixed $f=g\in\mF,\ \sigma \in \R_+$ and $(\bx,y)\in\mX\times\R,$ the function $\chi \mapsto R_c(\ell_f(\bx,y),\chi,\ell_f(\bx,y),\sigma)$ is concave and has a unique maximum for $\chi\in\R_+.$ \label{prop.conc_max_chi}
\end{enumerate}
\unskip
This is an additional requirement that has no counterpart when $\sigma^*$ is known. In fact, for $f=g,$ we have $R(\ell_g,\ell_f) = \ell_f - \ell_g \equiv 0.$ 

\begin{enumerate}[label=\textbf{P\arabic*.}, ref=P\arabic*]
    \setcounter{enumi}{2}
    \item \textbf{Maximization over \bm{$g.$}} For any fixed $f\in\mF$ and $\chi,\sigma \in \R_+,$ the problems of maximizing the functionals
    \begin{align*}
        g \mapsto MOM_K\Big(R_c(\ell_g,\chi,\ell_f,\sigma) \Big),\quad g \mapsto MOM_K\Big( \ell_f - \ell_g \Big),
    \end{align*}
    over $g\in\mF$ are equivalent. \label{prop.conc_max_g}
\end{enumerate}
\unskip
The latter condition requires that our functional $R_c(\ell_g,\chi,\ell_f,\sigma)$ behaves similarly to $R(\ell_g,\ell_f) = \ell_f-\ell_g$ when viewed as a functional on $g\in\mF.$

As a consequence of anti-symmetry, the following properties are equivalent to~\ref{prop.antysymmetry}--\ref{prop.conc_max_g} above:
\begin{enumerate}[label=\textbf{P\arabic*'.}, ref=P\arabic*']
    \item \textbf{Anti-symmetry.} For all $f,g\in\mF$ and $\chi,\sigma \in \R_+,$ we have $R_c(\ell_g,\chi,\ell_f,\sigma) = - R_c(\ell_f,\sigma,\ell_g,\chi).$
    \item \textbf{Convexity in \bm{$\sigma$}, given \bm{$f=g.$}} For any fixed $f=g\in\mF,\ \chi \in \R_+$ and $(\bx,y)\in\mX\times\R,$ the function $\sigma \mapsto R_c(\ell_f(\bx,y),\chi,\ell_f(\bx,y),\sigma)$ is convex and has a unique minimum for $\sigma\in\R_+.$
    \item \textbf{Minimization over \bm{$f.$}} For any fixed $g\in\mF$ and $\chi,\sigma \in \R_+,$ the problems of minimizing the functionals
    \begin{align*}
        f \mapsto MOM_K\Big(R_c(\ell_g,\chi,\ell_f,\sigma) \Big),\quad f \mapsto MOM_K\Big( \ell_f - \ell_g \Big),
    \end{align*}
    over $f\in\mF$ are equivalent.
\end{enumerate}

Consider the sparse linear setting, where we want to recover oracle solutions 
\begin{align*}
    \bbeta^* \in \argmin_{\bbeta\in\R^d} \E\left[ (Y-\bX^\top\bbeta )^2 \right],\quad \sigma^* = \E\left[ (Y-\bX^\top\bbeta^* )^2 \right]^{\frac{1}{2}}.
\end{align*}
Any linear function $f:\mX \to \R$ can be identified with some $\bbeta_f\in\R^d$ such that $f(\bx) = \bx^\top\bbeta_f$ and $\ell_f(\bx,y) = \ell_{\bbeta_f}(\bx,y) = (y-\bx^\top\bbeta_f)^2.$ The MOM method in~\cite{lecue2020robustML} yields a robust version of the LASSO estimator
\begin{align*}
    \widehat\bbeta^{L} \in \argmin_{\bbeta\in\R^d} \left\{ \frac{1}{n} \sum_{i=1}^n (Y_i-\bX_i^\top\bbeta)^2 + \lambda|\bbeta|_1\right\},
\end{align*}
which has been shown to be minimax optimal in~\cite{bellec2016,bellec2017,bellec2018}, but its optimal tuning parameter $\lambda$ is proportional to $\sigma^*.$ An adaptive version of the LASSO is the square-root LASSO introduced in~\cite{belloni2010}, which is also minimax optimal, as shown in~\cite{derumigny2018improved}. This adaptive method uses
\begin{align*}
    \widehat\bbeta^{SR\text{-}Lasso} \in \argmin_{\bbeta\in\R^d} \left\{ \left( \frac{1}{n}\sum_{i=1}^n (Y_i-\bX_i^\top\bbeta)^2 \right)^{\frac{1}{2}} + \mu|\bbeta|_1\right\},
\end{align*}
and its optimal tuning parameter $\mu$ does not require the knowledge of $\sigma^*.$ The key insight behind the square-root LASSO, see for example Section 5 in~\cite{giraud2014introduction}, is that when $\bbeta$ is close to $\bbeta^*$ one can approximate $\sigma^{*2}$ by $\E[(Y -\bX^\top \bbeta)^2].$ Thus, with $\lambda = \sigma^*\mu,$ one finds
\begin{align*}
    \frac{\E[(Y -\bX^\top \bbeta)^2]}{\sigma^*} + \frac{\lambda}{\sigma^*} |\bbeta|_1 \simeq \E[(Y -\bX^\top \bbeta)^2]^{\frac{1}{2}} + \mu|\bbeta|_1,
\end{align*}
and the minimization problem is independent of $\sigma^*.$ 

In view of the discussion above, a candidate natural implementation of the robust square-root LASSO is given by
\begin{align*}
    \widetilde R_c(\ell_g,\chi,\ell_f,\sigma) &= \frac{\ell_f}{\sigma} + \sigma - \frac{\ell_g}{\chi} - \chi, \\
    &= (\sigma - \chi) \bigg( 1 - \frac{\ell_f}{\sigma \chi} \bigg) + \frac{\ell_f - \ell_g}{\chi}, \\
    \widetilde T_{K,\mu}(g,\chi,f,\sigma) &= MOM_K\Big( \widetilde R_c(\ell_g,\chi,\ell_f,\sigma) \Big) + \mu\big(\|f\|-\|g\| \big),
\end{align*}
since $\widetilde R_c$ implements the idea that, in the linear setting, dividing $\ell_f$ by $\sigma$ should lead to the square-root of $\ell_f.$ Also, this choice satisfies the properties~\ref{prop.antysymmetry}--\ref{prop.conc_max_g}:
\begin{itemize}
    \item Anti-symmetry holds by construction.
    \item When $f=g,$ replace $\ell_f(\bx,y) = \ell_g(\bx,y)$ by some positive real number $a^2>0,$ then the function
    \begin{align*}
        \chi \mapsto \widetilde R_c(a^2,\chi,a^2,\sigma) = (\sigma - \chi)\left(1 - \frac{a^2}{\sigma\chi}\right),
    \end{align*}
    is concave and has a unique maximum for $\chi\in\R_+.$
    \item By definition, maximizing $g\mapsto MOM_K(\ell_f-\ell_g)$ with fixed $f\in\mF$ is equivalent to maximizing the empirical average
    \begin{align*}
        g \mapsto -\frac{1}{|B_k|}\sum_{i\in B_k} \ell_g(\bX_i,Y_i),
    \end{align*}
    where the block $B_k$ realizes the median. For the same reason, maximising $g\mapsto MOM_K(\widetilde R_c(\ell_g,\chi,\ell_f,\sigma))$ is equivalent to maximizing the empirical average
    \begin{align*}
        g \mapsto \frac{1}{|B_k|}\sum_{i\in B_k} \left( \frac{\ell_f(\bX_i,Y_i)}{\sigma} + \sigma - \frac{\ell_g(\bX_i,Y_i)}{\chi} - \chi\right),
    \end{align*}
    where the block $B_k$ realizes the median. Since the quantities $f,\sigma,\chi$ are fixed, this coincides with the above.
\end{itemize}

However, this choice comes with a drawback. The proof of our main result is based on the argument proposed in~\cite{lecue2020robustML}, which requires sharp bounds for the functional $\widetilde T_{K,\mu}(\ell_g,\chi,\ell_{f^*},\sigma^*)$ over the possible values of $(g,\chi).$ This is done by carefully slicing the domain and assessing the contribution of each term appearing in $\widetilde T_{K,\mu}.$ In particular, one finds a slice in which $\chi < \sigma^* - c_\alpha r(2\rho_K)$ and the leading term of $\widetilde T_{K,\mu}$ is of the form $2\eps/\chi,$ with some small fixed $\eps>0.$ Since $2\eps/\chi \to +\infty,$ for $\chi\to0,$ we cannot control the supremum of $\widetilde T_{K,\mu}(\ell_g,\chi,\ell_{f^*},\sigma^*)$ over this slice. The only way around it would be to assume from the start that $\sigma^* > \sigma_{-},$ for some known lower bound $\sigma_{-}>0,$ but this would be a stronger assumption than the upper bound $\sigma_{+}$ we use in~\eqref{def.zeta_moments}. This issue is caused by the fact that the two terms of $\widetilde R_c(\ell_g,\chi,\ell_f,\sigma)$ are 
\begin{align*}
    (\ell_g,\chi,\ell_f,\sigma) \mapsto (\sigma - \chi) \bigg( 1 - \frac{\ell_f}{\sigma \chi} \bigg),\quad  (\ell_g,\chi,\ell_f,\sigma) \mapsto \frac{\ell_f - \ell_g}{\chi},
\end{align*}
and the second one cannot be controlled if $\chi\to 0.$ A way to introduce stability is to replace the denominator $\chi$ by the average $(\sigma+\chi)/2,$ which is always bounded away from zero when $\sigma$ is fixed. However, making this substitution alone breaks the anti-symmetry of the functional, so we have to take care of both terms simultaneously. To this end, we use
\begin{align*}
    R_c(\ell_g,\chi,\ell_f,\sigma) &= (\sigma-\chi)\bigg(1-2\frac{\ell_f+\ell_g}{(\sigma+\chi)^2} \bigg) + 2c\frac{\ell_f-\ell_g}{\sigma+\chi}, \\
    T_{K,\mu}(g,\chi,f,\sigma) &= MOM_K\Big( R_c(\ell_g,\chi,\ell_f,\sigma) \Big) + \mu\big(\|f\|-\|g\| \big),
\end{align*}
for all $(f,g)\in\mF\times\mF$ and $(\sigma,\chi)\in (0,\sigma_{+}]\times(0,\sigma_{+}],$ which guarantees that $R_c$ satisfies properties~\ref{prop.antysymmetry}--\ref{prop.conc_max_g}. In fact, anti-symmetry holds for both terms
\begin{align*}
    (\ell_g,\chi,\ell_f,\sigma) \mapsto (\sigma-\chi)\bigg(1-2\frac{\ell_f+\ell_g}{(\sigma+\chi)^2} \bigg) ,\quad  (\ell_g,\chi,\ell_f,\sigma) \mapsto 2c\frac{\ell_f-\ell_g}{\sigma+\chi},
\end{align*}
separately. Also, for any fixed $f=g\in\mF,\ \sigma\in\R_+,$ we have
\begin{align*}
    \chi \mapsto R_c(\ell_f,\chi,\ell_f,\sigma) &= (\sigma-\chi)\bigg(1-\frac{4\ell_f}{(\sigma+\chi)^2} \bigg),
\end{align*}
which satisfies property~\ref{prop.conc_max_chi}. Finally, for any fixed $f\in\mF,\ \sigma,\chi\in\R_+,$ we can rewrite
\begin{align*}
    g &\mapsto MOM_K\left( R_c(\ell_g,\chi,\ell_f,\sigma) \right) \\
    &\quad = MOM_K\left( (\sigma-\chi) + \frac{2\ell_f}{\sigma + \chi} \left(c - \frac{\sigma-\chi}{\sigma+\chi} \right) - \frac{2\ell_g}{\sigma + \chi} \left(c + \frac{\sigma-\chi}{\sigma+\chi} \right) \right).
\end{align*}
Since the quantity $c + (\sigma-\chi)/(\sigma+\chi)$ belongs to the interval $[c-1,c+1]$ and $c>1,$ property \ref{prop.conc_max_g} is satisfied.

\subsection{From \texorpdfstring{$R_c$}{Rc} to convergence rates and excess risk bounds} \label{sec.disc_conv_rate_risk}

The choice of $R_c$ induces a penalized functional $T_{K,\mu}$ which characterizes the MOM$-K$ estimator  
\begin{align*}
    (\widehat f_{K,\mu,\sigma_+}, \widehat\sigma_{K,\mu,\sigma_+}) = \argmin_{f \in \mF,\ \sigma \in I_{+}} \maxnew_{g \in \mF,\ \chi \in I_{+}} T_{K,\mu}(g,\chi,f,\sigma),\quad I_{+} = (0,\sigma_{+}].
\end{align*}
Our goal is to guarantee that, with as high probability as possible, the function estimator $\widehat f_{K,\mu,\sigma_+}$ recovers $f^*$ with as small as possible rates in $\|\cdot\|$ and $\|\cdot\|_{2,\bX},$ and that the standard deviation estimator $\widehat\sigma_{K,\mu,\sigma_+}$ recovers $\sigma^*$ with as small as possible rates in absolute value. With the same high probability, we also want that the excess risk $\Risk(\widehat f_{K,\mu})-\Risk(f^*)$ is as small as possible.

Starting with the convergence rates, they can be obtained by showing that the estimator $(\widehat f_{K,\mu,\sigma_+}, \widehat\sigma_{K,\mu,\sigma_+})$ belongs to a bounded ball of the form
\begin{align*}
    \B^*(2\rho) := \big\{(f,\sigma)\in\mF\times I_{+} : \|f-f^*\| \leq 2\rho,\ \|f-f^*\|_{2,\bX} \leq r(2\rho),\ |\sigma-\sigma^*| \leq c_{\alpha} r(2\rho) \big\},
\end{align*}
with appropriate radius $\rho$ and complexity measure $r(2\rho).$ In the proof of Theorem~\ref{thm.main_theorem}, we show that this can be achieved with $\rho = \rho_K$ and any $r(\rho) \geq \max\{r_P(\rho,\gamma_P),\ r_M(\rho,\gamma_M) \},$ which only requires the complexities $r_P,r_M.$ The convergence rates $2\rho_K, r(2\rho_K)$ are perfectly in line with those obtained with the MOM tournaments procedure in \cite{lugosi2017regularization} and the robust MOM method in \cite{lecue2020robustML}. The key idea behind this result is to essentially show that the evaluation of $T_{K,\mu}$ at the point $(\widehat f_{K,\mu,\sigma_+}, \widehat\sigma_{K,\mu,\sigma_+},f^*,\sigma^*)$ is too big for $(\widehat f_{K,\mu,\sigma_+}, \widehat\sigma_{K,\mu,\sigma_+})$ to be outside of the bounded ball $\B^*(2\rho_K).$ Precisely, we show that, for some $B_{1,1} > 0,$
\begin{align*}
    T_{K,\mu}(\widehat f_{K,\mu,\sigma_+}, \widehat\sigma_{K,\mu,\sigma_+},f^*,\sigma^*) \geq -B_{1,1},\quad \sup_{(g,\chi)\notin \B^*(2\rho_K, r(2\rho_K))} T_{K,\mu}(g,\chi,f^*,\sigma^*) < - B_{1,1},
\end{align*}
which guarantees that $(\widehat f_{K,\mu,\sigma_+}, \widehat\sigma_{K,\mu,\sigma_+},f^*,\sigma^*) \in \B^*(2\rho_K).$ The problem of finding a suitable bound $B_{1,1}$ is solved as follows.
\begin{itemize}
    \item The problem is equivalent to $- T_{K,\mu}(\widehat f_{K,\mu}, \widehat\sigma_{K,\mu},f^*,\sigma^*) \leq B_{1,1}.$ 
    \item By the anti-symmetry property~\ref{prop.antysymmetry} of $R_c,$ together with the quantile properties in Lemma~\ref{lemma.quantile_prop}, we have $- T_{K,\mu}(f, \sigma,f^*,\sigma^*) \leq T_{K,\mu}(f^*,\sigma^*,f, \sigma)$ and it is sufficient to find $T_{K,\mu}(f^*,\sigma^*,\widehat f_{K,\mu,\sigma_+}, \widehat\sigma_{K,\mu,\sigma_+}) \leq B_{1,1}.$
    \item The evaluation at $(f^*,\sigma^*)$ can be bounded with the supremum over the domain, that is, we look for $\sup_{(g,\chi)\in \mF\times I_{+}} T_{K,\mu}(g,\chi,\widehat f_{K,\mu,\sigma_+}, \widehat\sigma_{K,\mu,\sigma_+}) \leq B_{1,1}.$
    \item By definition, the MOM$-K$ estimator $(\widehat f_{K,\mu,\sigma_+}, \widehat\sigma_{K,\mu,\sigma_+})$ minimizes the latter supremum if we allow for other pairs $(f,\sigma).$ In particular, with $(f,\sigma)=(f^*,\sigma^*),$ it is enough to find $\sup_{(g,\chi)\in \mF\times I_{+}} T_{K,\mu}(g,\chi,f^*,\sigma^*) \leq B_{1,1}.$
    \item Finally, in Lemma~\ref{lemma.comparison_bounds_B_i} we show that the supremum is achieved on the bounded ball $\B^*(\rho_K),$ that is, the solution to the problem is the sharpest bound such that
    \begin{align*}
        \sup_{(g,\chi)\in \B^*(\rho_K)} T_{K,\mu}(g,\chi,f^*,\sigma^*) \leq B_{1,1}.
    \end{align*}
\end{itemize}
The argument we just sketched can be found in the proof of the main result in~\cite{lecue2020robustML}, it is a clever exploitation of the convex-concave formulation of the problem. One key element of the argument is that the computations only require lower bounds on the quantiles of the quadratic and multiplier empirical processes, which in turn can be obtained by means of the complexities $r_P$ and $r_M$ alone. These facts has been established in~\cite{lecue2018regularization, lugosi2019regularization} and we provide them in Lemma~\ref{lemma.3}, Lemma~\ref{lemma.4}.

The fact that the estimator $(\widehat f_{K,\mu,\sigma_+}, \widehat\sigma_{K,\mu,\sigma_+})$ belongs to the ball $\B^*(2\rho_K)$ is instrumental in obtaining excess risk bounds. First, one writes
\begin{align*}
    \Risk(\widehat f_{K,\mu,\sigma_+})-\Risk(f^*) &= \|\widehat f_{K,\mu,\sigma_+} - f^*\|_{2,\bX}^2 + \E[-2\zeta(\widehat f_{K,\mu,\sigma_+}-f^*)(\bX)],
\end{align*}
and then bounds $\|\widehat f_{K,\mu,\sigma_+} - f^*\|_{2,\bX}^2 \leq r^2(2\rho_K).$ By applying a quantile inequality, see Lemma~\ref{lemma.useful_bounds_lecue}, and adding the quadratic term $(\widehat f_{K,\mu,\sigma_+}-f^*)^2,$ the expectation term becomes
\begin{align*}
    \E[-2\zeta(\widehat f_{K,\mu,\sigma_+}-f^*)(\bX)] &\leq Q_{1/4,K}\left[-2\zeta(\widehat f_{K,\mu,\sigma_+}-f^*)\right] + \alpha_M^2 \\
    &\leq Q_{1/4,K}\left[\ell_{\widehat f_{K,\mu,\sigma_+}} - \ell_{f^*}\right] + \alpha_M^2,
\end{align*}
since $\ell_f - \ell_{f^*} = (f-f^*)^2 - 2\zeta(f-f^*).$ Since the $1/4-$quantile is always smaller than the $1/2-$quantile, which is the median, some algebraic manipulations allow to rewrite the difference $\ell_{\widehat f_{K,\mu,\sigma_+}} - \ell_{f^*}$ in terms of our functional $R_c(\ell_{f^*},\sigma^*,\ell_{\widehat f_{K,\mu,\sigma_+}},\widehat\sigma_{K,\mu,\sigma_+})$ and to recover the penalized $T_{K,\mu}(f^*,\sigma^*,\widehat f_{K,\mu,\sigma_+},\widehat\sigma_{K,\mu,\sigma_+}).$ Specifically, in Lemma~\ref{lemma.bound_P2zeta_f_fstar} we find
\begin{align*}
    \E[-2\zeta(\widehat f_{K,\mu,\sigma_+}-f^*)(\bX)] &\leq \frac{\widehat\sigma_{K,\mu,\sigma_+} + \sigma^*}{2c} T_{K,\mu}(f^*,\sigma^*,\widehat f_{K,\mu,\sigma_+},\widehat\sigma_{K,\mu,\sigma_+}) + \text{remainder}, \\
    &\leq \frac{\widehat\sigma_{K,\mu,\sigma_+} + \sigma^*}{2c} B_{1,1} + \text{remainder},
\end{align*}
where $B_{1,1}$ is the upper bound we found when dealing with the convergence rates. It is easy to show that $B_{1,1} \lesssim r^2(2\rho_K),$ the majority of the work is spent on bounding the remainder terms. In the same lemma, we show that they are: the quantity $\mu\rho_K \lesssim r^2(\rho_K)$ where $\mu \simeq r^2(\rho_K)/\rho_K$ is the penalization parameter, the quantity $\alpha_M^2 \lesssim r^2(2\rho_K)$ related to the quantiles of the multiplier process, the mixed terms
\begin{itemize}
    \item $|\widehat\sigma_{K,\mu,\sigma_+} - \sigma^*| \cdot Q_{15/16,K}\left[(\widehat f_{K,\mu,\sigma_+}-f^*)^2\right],$
    \item $|\widehat\sigma_{K,\mu,\sigma_+} - \sigma^*|\cdot Q_{15/16,K}\left[-2\zeta(\widehat f_{K,\mu,\sigma_+}-f^*) \right],$
\end{itemize}
involving the quantiles of the quadratic and multiplier processes. The standard deviation estimator satisfies $|\widehat\sigma_{K,\mu,\sigma_+} - \sigma^*| \lesssim r(2\rho_K).$ In Lemma~\ref{lemma.useful_bounds_lecue} we show that $Q_{15/16,K}[-2\zeta(\widehat f_{K,\mu}-f^*)] \leq \E[-2\zeta(\widehat f_{K,\mu,\sigma_+}-f^*)] + \alpha_M^2,$ so that the Cauchy-Schwarz inequality is sufficient for $\E[-2\zeta(\widehat f_{K,\mu,\sigma_+}-f^*)] \leq 4\sigma^* \|\widehat f_{K,\mu,\sigma_+} - f^*\|_{2,\bX} \lesssim r(2\rho_K).$ Finally, in Lemma~\ref{lemma.l2_quantile} we find $Q_{15/16,K}[(\widehat f_{K,\mu,\sigma_+}-f^*)^2 ] \leq r^2(2\rho_K) + \alpha_Q^2 \lesssim r^2(2\rho_K) \vee r_Q^2(2\rho_K,\gamma_Q).$

\subsection{Complexity parameters in the sub-Gaussian setting} \label{sec.disc_complexity_subGauss}

We follow the construction presented in~\cite{lecue2013learning}. Let $G = (G(f) : f\in L^2(\P_{\bX}))$ the Gaussian process indexed on $L^2(\P_{\bX})$ and such that $\E[G(f)] = 0$ and $\E[G(f)G(h)] = \E[f(\bX)h(\bX)].$ For any $\mF'\subseteq\mF,$ we set
\begin{align*}
    \E\left[\|G\|_{\mF'}\right] := \sup \left\{\E\left[\sup_{h\in\mH} G(h) \right] : \mH\subseteq\mF' \text{ is finite} \right\}.
\end{align*}
As an example, if $\mF' = \{\bx \mapsto \bx^\top\bbeta: \bbeta\in T\subset \R^d \}$ and $\bX$ is a random vector in $\R^d$ with covariance matrix $\Sigma,$ then $G \sim \mN(0,\Sigma)$ and 
\begin{align*}
    \E\left[\|G\|_{\mF'}\right] = \E\left[\sup_{\bbeta\in T} G^\top\bbeta\right].
\end{align*}

\textbf{Sub-Gaussian class.} We say that $\mF$ is sub-Gaussian if there exists a constant $L$ such that, for all $f,h\in\mF$ and $p\geq2,$ one has $\|f-h\|_{p,\bX} \leq L \sqrt{p} \|f-h\|_{2,\bX}.$

\textbf{Gaussian complexities.} For any $r\geq0,$ set $\B_2(r) = \{f\in L^2(\P_{\bX}):\|f\|_{2,\bX}\leq r\}$ and $\mF-\mF= \{f-h:f,h\in\mF\}.$ For any $\gamma,\gamma'>0,$ take
\begin{align}
\begin{split}\label{def.s*r*}
    s_n^*(\gamma) &:= \inf\{r>0:\E\big[\|G\|_{\B_2(r)\cap(\mF-\mF)}\big] \leq \gamma r^2 \sqrt{n}\}, \\
    r_n^*(\gamma') &:= \inf\{r>0:\E\big[\|G\|_{\B_2(r)\cap(\mF-\mF)}\big] \leq \gamma' r \sqrt{n}\}.
\end{split}
\end{align}

The goal of this section is to provide the following bounds.
\begin{lem}\label{lemma.complexity_bound_subgauss}
    Under the sub-Gaussian assumption, there exist absolute constants $c_2,c_3$ such that the complexity parameters $r_P,r_Q,r_M$ defined in \eqref{def.complexity} satisfy
    \begin{align}\label{eq.compl_params}
        r_P(\rho, \gamma_P) \leq r_n^*\left(\frac{\gamma_P}{c_2L^2} \right),\quad r_Q(\rho, \gamma_Q) \leq r_n^*\left(\frac{\gamma_Q}{c_2L^2} \right),\quad r_M(\rho, \gamma_M) \leq s_n^*\left(\frac{\gamma_M}{c_3 L \frakm^*} \right).
    \end{align}
    In particular, any continuous non-decreasing function $\rho \mapsto r(\rho)$ with 
    \begin{align*}
        r(\rho) \geq \max\left\{r_n^*\left(\frac{\gamma_P}{c_2L^2} \right), s_n^*\left(\frac{\gamma_M}{c_3 L \frakm^*}\right) \right\},
    \end{align*}
    is a valid choice in~\eqref{def.r_complexity}.
\end{lem}
\begin{proof}[Proof of Lemma \ref{lemma.complexity_bound_subgauss}]
    We invoke Lemma~\ref{cor.1.8}, Lemma~\ref{lemma.2.6} and Lemma~\ref{lemma.2.7} below. They are all based on a symmetrization argument in \cite{mendelson2014upper}, which controls the processes 
    \begin{align*}
        \sup_{f\in\mF:\|f-f^*\|_{2,\bX} \leq r} &\bigg|\frac{1}{n}\sum_{i=1}^n (f-f^*)(\bX_i) - \E[(f-f^*)(\bX)] \bigg|, \\
        \sup_{f\in\mF:\|f-f^*\|_{2,\bX} \leq r} &\bigg|\frac{1}{n}\sum_{i=1}^n (f-f^*)^2(\bX_i) - \E[(f-f^*)^2(\bX)] \bigg|, \\
        \sup_{f\in\mF:\|f-f^*\|_{2,\bX} \leq r} &\bigg|\frac{1}{n}\sum_{i=1}^n \zeta_i(f-f^*)(\bX_i) - \E[\zeta(f-f^*)(\bX)] \bigg|,
    \end{align*}
    in terms of the processes
    \begin{align*}
        \sup_{f\in\mF:\|f-f^*\|_{2,\bX} \leq r} &\bigg|\frac{1}{n}\sum_{i=1}^n \xi_i(f-f^*)(\bX_i)\bigg|, \\
        \sup_{f\in\mF:\|f-f^*\|_{2,\bX} \leq r} &\bigg|\frac{1}{n}\sum_{i=1}^n \xi_i(f-f^*)^2(\bX_i)\bigg|, \\
        \sup_{f\in\mF:\|f-f^*\|_{2,\bX} \leq r} &\bigg|\frac{1}{n}\sum_{i=1}^n \xi_i\zeta_i(f-f^*)(\bX_i)\bigg|,
    \end{align*}
    with Rademacher variables $(\xi_i)_{i=1,\ldots,n}.$ These processes play a role in the definition of the complexities in~\eqref{def.complexity}.
    
    Lemma~\ref{cor.1.8} below shows that, for any $r>r_n^*(\gamma'),$
    \begin{align*}
        \sup_{f,h\in \mF: \|f-h\|_{2,\bX} \leq r} \bigg|\frac{1}{n}\sum_{i=1}^n (f-h)(\bX_i) - \E[(f-h)(\bX)] \bigg| \leq c_2\gamma' L r,
    \end{align*}
    with probability bigger than $1-2\exp(-c_1\gamma'{}^2n).$ Choosing $\gamma' = \gamma_P/(c_2L)$ and $h=f^*$ gives, for all $r>r_n^*(\gamma_Q/(c_2L)),$
    \begin{align*}
        \sup_{f\in\mF:\|f-f^*\|_{2,\bX} \leq r} \bigg|\frac{1}{n}\sum_{i=1}^n (f-f^*)(\bX_i) - \E[(f-f^*)(\bX)] \bigg| \leq \gamma_Q r.
    \end{align*}
    By definition, the complexity $r_P(\rho,\gamma_P)$ is the smallest level $r$ at which the latter display holds for all functions $f$ in the smaller set $\B(f^*,\rho,r).$ Thus $r_P(\rho,\gamma_P) \leq r_n^*(\gamma_P/(c_2 L)).$ 
    
    Lemma~\ref{lemma.2.6} below shows that, for any $r>r_n^*(\gamma'),$
    \begin{align*}
        \sup_{f,h\in \mF: \|f-h\|_{2,\bX} \leq r} \bigg|\frac{1}{n}\sum_{i=1}^n (f-h)^2(\bX_i) - \E[(f-h)^2(\bX)] \bigg| \leq c_2\gamma' L^2 r^2,
    \end{align*}
    with probability bigger than $1-2\exp(-c_1\gamma'{}^2n).$ Choosing $\gamma' = \gamma_Q/(c_2L^2)$ and $h=f^*$ gives, for all $r>r_n^*(\gamma_Q/(c_2L^2)),$
    \begin{align*}
        \sup_{f\in\mF:\|f-f^*\|_{2,\bX} \leq r} \bigg|\frac{1}{n}\sum_{i=1}^n (f-f^*)^2(\bX_i) - \E[(f-f^*)^2(\bX)] \bigg| \leq \gamma_Q r^2.
    \end{align*}
    By definition, the complexity $r_Q(\rho,\gamma_Q)$ is the smallest level $r$ at which the latter display holds for all functions $f$ in the smaller set $\B(f^*,\rho,r).$ Thus $r_Q(\rho,\gamma_Q) \leq r_n^*(\gamma_Q/(c_2L^2)).$
    
    With $\E[\zeta^4]^{1/4} = \frakm^*,$ Lemma~\ref{lemma.2.7} below shows that, for any $r > s_n^*(\gamma),$
    \begin{align*}
        \sup_{f,h\in\mF:\|f-h\|_{2,\bX}\leq r}  \bigg|\frac{1}{n}\sum_{i=1}^n \zeta_i(f-h)(\bX_i) - \E[\zeta(f-h)(\bX)] \bigg| \leq c_3 \gamma \frakm^* L r^2,
    \end{align*}
    with probability bigger than $1-4\exp(-c_1 n \min\{\gamma^2r^2,1\}).$ Choosing $\gamma = \gamma_M/(c_3 L \frakm^*)$ and $h=f^*$ gives, for all $r>s_n^*(\gamma_M/(c_3 L \frakm^*)),$ \begin{align*}
        \sup_{f\in\mF:\|f-f^*\|_{2,\bX}\leq r}  \bigg|\frac{1}{n}\sum_{i=1}^n \zeta_i(f-f^*)(\bX_i) - \E[\zeta(f-f^*)(\bX)] \bigg| \leq \gamma_M r^2.
    \end{align*}
    By definition, the complexity $r_M(\rho,\gamma_M)$ is the smallest display $r$ at which the latter display holds for all functions $f$ in the smaller set $\B(f^*,\rho,r).$ Thus $r_M(\rho,\gamma_M) \leq s_n^*(\gamma_M/(c_3 L \frakm^*)).$
\end{proof}

\begin{lem}[Corollary 1.8 in~\cite{mendelson2014upper}] \label{cor.1.8}
    There exist absolute constants $c_1,c_2$ for which the following holds. Let $\mF$ be an $L-$sub-Gaussian class, assume that $\mF-\mF$ is star-shaped around $0.$ If $\gamma'\in(0,1)$ and $r>r_n^*(\gamma'),$ then with probability at least $1-2\exp(-c_1\gamma'{}^2n),$ we have
    \begin{align*}
        \sup_{f,h\in \mF: \|f-h\|_{2,\bX} \leq r} \bigg|\frac{1}{n}\sum_{i=1}^n (f-h)(\bX_i) - \E[(f-h)(\bX)] \bigg| \leq c_2\gamma' L r.
    \end{align*}
\end{lem}

\begin{lem}[Lemma 2.6 in~\cite{lecue2013learning}] \label{lemma.2.6}
    There exist absolute constants $c_1,c_2$ for which the following holds. Let $\mF$ be an $L-$sub-Gaussian class, assume that $\mF-\mF$ is star-shaped around $0.$ If $\gamma'\in(0,1)$ and $r>r_n^*(\gamma'),$ then with probability at least $1-2\exp(-c_1\gamma'{}^2n),$ we have
    \begin{align*}
        \sup_{f,h\in \mF: \|f-h\|_{2,\bX} \leq r} \bigg|\frac{1}{n}\sum_{i=1}^n (f-h)^2(\bX_i) - \E[(f-h)^2(\bX)] \bigg| \leq c_2\gamma' L^2 r^2.
    \end{align*}
\end{lem}

\begin{lem}[Corollary of Theorem 2.7 in~\cite{lecue2013learning}] \label{lemma.2.7}
    Let $\mF$ be an $L-$sub-Gaussian class, assume that $\mF-\mF$ is star-shaped around $0.$ 
    Let $\E[|\zeta|^q]^{1/q} = \frakm^*$ for some $q>2,$ there exists an absolute constant $c_3(q),$ depending on $q$ only, for which the following holds. For some $\gamma>0$ and $r>s_n^*(\gamma),$ with probability at least $1-4\exp(-c_1 n \min\{\gamma^2r^2,1\}),$ we have
    \begin{align*}
        \sup_{f,h\in\mF:\|f-h\|_{2,\bX}\leq r}  \bigg|\frac{1}{n}\sum_{i=1}^n \zeta_i(f-h)(\bX_i) - \E[\zeta(f-h)(\bX)] \bigg| \leq c_3(q) \gamma \frakm^* L r^2.
    \end{align*}
\end{lem}

\subsection{Complexity parameters in the sparse linear setting} \label{sec.disc_complexity_sparse}

The next result shows that, in the linear setting, it is possible to weaken the sub-Gaussian assumption and still be able to control the complexity parameters $r_P, r_M$ as in \eqref{eq.compl_params}. 

\begin{thm}[Theorem~1.6 in~\cite{mendelson2017multiplier}] \label{thm.1.6}
    There exists an absolute constant $c_1$ and for $K\geq1,$ $L\geq1$ and $q_0>2$ there exists a constant $c_2$ that depends only on $K,L,q_0$ for which the following holds. Consider
    \begin{itemize}
        \item $V\subset\R^d$ for which the norm $\|\cdot\|_V = \sup_{\bv\in V} |\langle \bv,\cdot\rangle|$ is $K-$unconditional with respect to the basis $\{\be_1,\ldots,\be_d\};$
        \item $\frakm^* = \E\big[|\zeta|^{q_0}\big]^{1/q_0} < +\infty;$
        \item an isotropic random vector $\bX \in\R^d$ which satisfies the \textit{weak moment condition}: for some constants $c_0, L > 1,$ for all $\by \in \R^d,$ $1 \leq p\leq c_0\log(ed),$ $1 \leq j \leq d,$
        \begin{align*}
            \E\big[|\bX^\top \be_j|^p \big]^{\frac{1}{p}}
            \leq L \sqrt{p} \, \E\big[|\bX^\top \be_j|^2 \big]^{\frac{1}{2}}.
        \end{align*}
    \end{itemize}
    If $(\bX_i,\zeta_i)_{i=1}^n$ are i.i.d. copies of $(\bX,\zeta),$ then
    \begin{align*}
        \E\left[\sup_{v\in V} \left| \frac{1}{\sqrt{n}} \sum_{i=1}^n \left(\zeta_i \bX_i^\top \bv - \E[\zeta \bX^\top \bv] \right) \right| \right] \leq c_2 \frakm^* \E[\|G\|_V].
    \end{align*}
\end{thm}
Since this result deals with the multiplier empirical process and, when $\zeta\equiv1,$ with the standard empirical process, by arguing as in the proof of Lemma~\ref{lemma.complexity_bound_subgauss} we find that any function
\begin{align*}
    \rho \mapsto r(\rho) \geq \max\left\{r_n^*\left(\frac{\gamma_P}{c_2} \right), s_n^*\left(\frac{\gamma_M}{c_2 \frakm^*}\right) \right\},
\end{align*}
is a valid choice in \eqref{def.r_complexity}. Our Definition~\ref{def.distr_class} restricts our analysis to settings where the assumptions of the previous theorem are satisfied.

By following Section 4 in~\cite{lecue2013learning}, we provide bounds for the complexity parameters $r_n^*,s_n^*$ in~\eqref{eq.compl_params}. For any $\bbeta\in\R^d,$ set $f_{\bbeta}:\R^d\to\R$ the linear map $f_{\bbeta}(\bx) = \bx^\top\bbeta,$ consider $\mF = \big\{ f_{\bbeta} : \bbeta\in\R^d \big\}$ and, for any $\rho>0,$
\begin{align*}
    \B_1(\rho) = \big\{ f_{\bbeta}\in\mF : |\bbeta|_1\leq \rho \big\}. 
\end{align*}
Assume that $\bX$ is an isotropic random vector that satisfies the weak moment condition of Theorem~\ref{thm.1.6}, recall that $\frakm^* = \E[\zeta^4]^{1/4}.$ By symmetry, $\B_1(\rho)-\B_1(\rho) = \B_1(2\rho)$ and it is sufficient to control the function $r\mapsto \E\big[\|G\|_{\B_1(2\rho)\cap\B_2(r)}\big].$ One finds, for every $2\rho/\sqrt{d} \leq r,$
\begin{align*}
    \E\big[ \|G\|_{\B_1(2\rho)\cap\B_2(r)} \big] = \E\bigg[\sup_{\bbeta\in\R^d:|\bbeta|_1\leq 2\rho,|\bbeta|_2\leq r} \bigg|\sum_{i=0}^d g_i\beta_i \bigg| \bigg] \sim \rho\sqrt{\log(ed\min\{r^2/\rho^2,1\})},
\end{align*}
and if $r\leq 2\rho/\sqrt{d},$ then
\begin{align*}
    \E\big[ \|G\|_{\B_1(2\rho)\cap\B_2(r)} \big] = \E\bigg[\sup_{\bbeta\in\R^d:|\bbeta|_1\leq 2\rho,|\bbeta|_2\leq r} \bigg|\sum_{i=0}^d g_i\beta_i \bigg| \bigg] \sim \rho\sqrt{d}.
\end{align*}
With $C_{\gamma_P}$ some constants only depending on $L$ and $\gamma_P,$ one finds
\begin{align*}
    r_n^{*2}\left(\frac{\gamma_P}{c_2} \right)  \leq C_{\gamma_P}^2 \times \begin{cases}
    \frac{\rho^2}{n} \log\Big(\frac{e d}{n} \Big) & \text{if $n \leq c_3 d,$} \\
    \frac{\rho^2}{d} & \text{if $c_3 d \leq n \leq c_4 d,$} \\
    0 & \text{$n > c_4 d,$}
    \end{cases}
\end{align*}
the constants $c_3,c_4$ depend only on $L.$ Similarly, with $C_{\gamma_M}$ some constants only depending on $L$ and $\gamma_M,$
\begin{align*}
    s_n^{*2}\left(\frac{\gamma_M}{c_2\frakm^*} \right) \leq C_{\gamma_M}^2 \times
    \begin{cases}
    \rho \frakm^* \sqrt{\frac{\log d}{n}} & \text{if $\rho^2 n \leq \frakm^{*2}\log d,$} \\
    \rho \frakm^* \sqrt{\frac{1}{n} \log\Big(\frac{e d^2 \frakm^{*2}}{\rho^2 n} \Big)} & \text{if $\frakm^{*2}\log d \leq \rho^2 n \leq \frakm^{*2}d^2,$} \\
    \frakm^{*2} \frac{d}{n} & \text{$\rho^2 n \geq \frakm^{*2}d^2.$}
    \end{cases}
\end{align*}
The bounds given above are valid for any regime of $n$ and $d,$ but we continue the discussion for the more interesting high-dimensional case, that is $d\gg n.$ This simplifies the notation and allows to choose, for some constant $C_{\gamma_P,\gamma_M}$ only depending on $L,\gamma_P,\gamma_M,$
\begin{align}\label{def.r_complexity_linear}
    r^2(\rho) = C_{\gamma_P,\gamma_M}^2 
    \begin{cases}
        \max\Big\{\rho \frakm^* \sqrt{\frac{\log d}{n}},\ \frac{\rho^2}{n}\log\left(\frac{ed}{n}\right) \Big\}, & \text{if $\rho \leq \frac{\frakm^* \sqrt{\log d}}{\sqrt{n}},$} \\
        \max\Big\{\rho \frakm^* \sqrt{\frac{1}{n} \log\left(\frac{e d^2 \frakm^{*2}}{\rho^2 n} \right)},\ \frac{\rho^2}{n}\log\left(\frac{ed}{n}\right) \Big\}, & \text{if } \frac{\frakm^* \sqrt{\log d}}{\sqrt{n}} \leq \rho \leq \frac{\frakm^* d}{\sqrt{n}},
    \end{cases}
\end{align}
which coincides with the function obtained in Section 4.4 in~\cite{lecue2020robustML}.

\textbf{Solution of the sparsity equation.} We study the case $n\geq s\log(ed/s)$ and assume there exists a $s-$sparse vector in $\bbeta^*+ \B_1(\rho/20).$ In the proof of Theorem 1.4 in~\cite{lecue2018regularization}, it is shown that the smallest solution of the sparsity equation~\eqref{def.sparsity_eq} is
\begin{align*}
    \rho^* = C_{\gamma_P,\gamma_M}^* \frakm^* s^* \sqrt{\frac{1}{n} \log\left(\frac{ed}{s^*} \right)},
\end{align*}
for some constant $C_{\gamma_P,\gamma_M}^*$ only depending on $L,\gamma_P,\gamma_M.$ We now compute $r^2(\rho^*).$ Up to multiplying $\rho^*$ by a big constant, we have $\rho^* \gtrsim \frakm^* \sqrt{\log d} / \sqrt{n},$ since $s^* \sqrt{\log(ed/s^*)} > \sqrt{\log d}$ for all $1 < s^* \leq d.$ By definition, we have
\begin{align*}
    r^2(\rho^*)
    &= C_{\gamma_P,\gamma_M}^2 \max\Big\{\rho^* \frakm^* \sqrt{\frac{1}{n} \log\left(\frac{e d^2 \frakm^{*2}}{\rho^{*2} n} \right)},\ \frac{\rho^{*2}}{n}\log\left(\frac{ed}{n}\right) \Big\} \\
    &= C_{\gamma_P,\gamma_M}^2 \rho^* \frakm^* \sqrt{\frac{1}{n} \log\left(\frac{e d^2 \frakm^{*2}}{\rho^{*2} n} \right)} \\
    &= C_{\gamma_P,\gamma_M}^2 C_{\gamma_P,\gamma_M}^* \frac{\frakm^{*2} s^*}{n} \sqrt{\log\left(\frac{ed}{s^*} \right)} \sqrt{\log\left(\frac{e d^2 }{C_{\gamma_P,\gamma_M}^{*2} {s^*}^2 \log\left(\frac{ed}{s^*} \right)} \right)} \\
    &\leq \sqrt{2} C_{\gamma_P,\gamma_M}^2 C_{\gamma_P,\gamma_M}^* \frac{\frakm^{*2} s^*}{n} \log\left(\frac{ed}{s^*} \right),
\end{align*}
in the last inequality we have used that $\log(a^2) = 2\log(|a|)$ and $C_{\gamma_P,\gamma_M}^* > 1/\sqrt{\log(ed/s^*)}.$ The latter is true without loss of generality in the high-dimensional setting $d\gg n \geq s^*\log(ed/s^*).$ The quantity $r(\rho^*)$ is the convergence rate of the LASSO estimator with penalization parameter $\lambda \sim r^2(\rho^*)/\rho^* \sim \frakm^*\sqrt{\log(ed/s^*)/n}.$ This choice of $\lambda$ requires the knowledge of the true sparsity parameter $s^*.$

\section*{Acknowledgements}

The research of Gianluca Finocchio is part of the project \textit{Nonparametric Bayes for high-dimensional models: contraction, credible sets, computations} (with project number 613.001.604) which is (partly) financed by the Dutch Research Council (NWO). 
Part of this work was done while Alexis Derumigny was employed at the University of Twente. Alexis Derumigny would like to thank Alexandre Tsybakov, Guillaume Lecué and Matthieu Lerasle for useful discussions on a very early version of the manuscript, that appears in his thesis~\cite[Chapter 3]{derumigny2019thesis}. This earlier version considers a different estimator and derives theoretical guarantees in the framework that $\sigma^* \asymp \sigma_-$ for some known level $\sigma_- > 0$.

\newpage

\begin{appendices}

\section{Proof of Theorem~\ref{thm.main_theorem}} \label{sec.proofs_main}

The structure of the proof is as follows. First, we control the supremum of the functional $T_{K, \mu}(g, \chi, f^*, \sigma^*)$ over possible values of $(g,\chi)$ by partitioning the domain in slices. Each slice is treated separately by the results from Lemma~\ref{lemma.bound_supg_Fkappa_1} to Lemma~\ref{lemma.bound_supg_Fkappa_9}. Then, we compare the bounds over different slices in Lemma~\ref{lemma.comparison_bounds_B_i} and show that the leading contribution comes from a bounded ball of the form
\begin{align*}
    \B^*(\rho_K) = \big\{(g,\chi)\in\mF\times (0,\sigma_{+}]: \|g-f^*\|\leq \rho_K,\ \|g-f^*\|_{2,\bX} \leq r(\rho_K),\ |\chi-\sigma^*|\leq c_\alpha r(\rho_K) \big\}.
\end{align*}
In Lemma~\ref{lemma.conv_rates} we translate the supremum bounds into convergence rates by showing that the MOM$-K$ estimator belongs to a bounded ball $\B^*(2\rho_K).$ We finalize the proof by computing the excess risk bound in Lemma~\ref{lemma.excess_risk}.

\medskip

In the notation of Theorem~\ref{thm.main_theorem}, for any $c>2$ we have
\begin{align*}
    c_\mu &:= 200 (c+2) \kappa_{+}^{1/2}, \\
    \eps &:= \frac{c-2}{192 \theta_0^2 (c+2) \big(8 + 134\kappa_{+}^{1/2} ((1+\frac{\sigma_{+}}{\sigma^*})\vee \frac{6}{5}) \big)}, \\
    c_\alpha^2 &:= \frac{3 (c-2)}{5 \theta_0^2},
\end{align*}
furthermore, we use the auxiliary parameters
\begin{align}
\begin{split}\label{def.parameters}
    \gamma_P = \frac{1}{1488 \theta_0},\quad \gamma_Q = \frac{\eps}{360},\quad \gamma_M = \frac{\eps}{744},\quad \eta = \frac{1}{16},\quad \gamma = \frac{31}{32},\quad \alpha = x = \frac{1}{93}.
\end{split}
\end{align}
We denote by $r(\cdot)$ a function such that $r(\rho) \geq \max\{r_P(\rho,\gamma_P), r_M(\rho,\gamma_M)\}.$ By Assumption~\ref{ass.r2rho_rrho}, there exists an absolute constant such that $r(\rho) \leq r(2\rho) < c_r r(\rho).$ With $C^2 = 384 \theta_1^2 c_r^2 c_\alpha^2 \kappa_{+}^{1/2},$ we allow for $K \in \left[K^* \vee 32|\mO|,\ n\eps^2/C^2 \right]$. We denote by $\Omega(K)$ the intersection of the event $\Omega_1(K)$ in Lemma~\ref{lemma.Q1/8_zeta2}, the event $\Omega_2(K)$ in Lemma~\ref{lemma.useful_bounds_lecue} and the event $\Omega_3(K)$ in Lemma~\ref{lemma.l2_quantile}. The probability of $\Omega(K) = \Omega_1(K) \cap \Omega_2(K) \cap \Omega_3(K)$ is at least $1 - \P(\Omega_1(K)) - \P(\Omega_2(K)) - \P(\Omega_3(K)) \geq 1 - 4 \exp(-K/8920).$ For any $c_\rho\in\{1,2\},$ we denote
\begin{align}\label{def.alpha_2k_delta_Kn}
    \alpha_{K,c_\rho} := c_\alpha r(c_\rho \rho_K),\quad \delta_{K,n}^2 := \frac{25 \frakm^{*4} K}{n},\quad r^2(\rho_K) = \frac{384 \theta_1^2 \delta_{K,n}^2}{25\frakm^{*2} \eps^2},
\end{align}
the last equation rewrites the implicit definition of $\rho_K$ in~\eqref{eq.rhoK}.

The next lemma checks that the choices made in Theorem~\ref{thm.main_theorem} satisfy a set of sufficient conditions that are required by our proving strategy. In principle, our main result is valid for different choices as long as the relevant quantities satisfy the conditions below.
\begin{lem}\label{lemma.conditions}
    The assumptions of Theorem~\ref{thm.main_theorem} imply, with $c_K^2 = 384$ and any $\iota_\mu \in [1/4, 4],$
    \begin{align}
        n\eps^2 &> K c_K^2 \theta_1^2 c_r^2 c_\alpha^2 \kappa_{+}^{1/2} , \label{ass.n_big} \\
        \iota_\mu c_\mu &> \frac{1600\kappa_{+}^{3/4}\eps}{c_K^2 \theta_1^2} + 48 \kappa_{+}^{1/2}(c+2), \label{ass.c_mu_lowbound} \\
        \frac{c-2}{24\theta_0^2} &> \frac{800\kappa_{+}^{1/2}\eps^2}{c_K^2 \theta_1^2} + 16(c+2) \eps + \bigg( \frac{1+\frac{\sigma_{+}}{\sigma^*}}{3} \vee \frac{36}{10} \bigg) \iota_\mu c_\mu\eps, \label{ass.1/theta_0_lowbound} \\
        c_\alpha^2 &> \frac{1800\kappa_{+}^{1/2}\eps^2}{c_K^2 \theta_1^2} + 108(c+2) \eps + \frac{144 \iota_\mu c_\mu \eps}{10}. \label{ass.c_alpha_lowbound}
    \end{align}
    Conditions~\eqref{ass.n_big} and~\eqref{ass.c_alpha_lowbound} imply $4\delta_{K,n}/\sigma^* < \alpha_{K,c_\rho} < \sigma^*.$ Condition~\eqref{ass.1/theta_0_lowbound} implies both
    \begin{align}
        \frac{1}{16\theta_0^2} &> 4\eps + \frac{(\sigma^*+\sigma_{+}) \iota_\mu c_\mu\eps}{2(c-2)\frakm^*} , \label{ass.1/theta_0_lowbound_negative} \\
        \frac{c-2}{24\theta_0^2} &> \frac{800\kappa^{*1/2}\eps^2}{c_K^2 \theta_1^2} + 16(c+2) \eps + \frac{36 \iota_\mu c_\mu \eps}{10}. \label{ass.1/theta_0_lowbound_2}
    \end{align}
\end{lem}
\begin{proof}[Proof of Lemma~\ref{lemma.conditions}]
    Condition~\eqref{ass.n_big} is equivalent to the upper bound $K \leq n\eps^2/C^2$ on the number of blocks. We have $r^2(\rho_K) = c_K^2 \theta_1^2 \frakm^{*2} K / (\eps^2 n)$ and $r^2(2\rho_K) \leq c_r^2 r^2(\rho_K),$ by Assumption~\ref{ass.r2rho_rrho}. Since $\alpha_{K,2} = c_\alpha r(2\rho_K),$ then also $\alpha_{K,2} \leq c_r \alpha_{K,1},$ therefore
    \begin{align*}
        \frac{\alpha_{K,1}^2}{\sigma^{*2}} \leq \frac{\alpha_{K,2}^2}{\sigma^{*2}} \leq \frac{c_r^2 \alpha_{K,1}^2}{\sigma^{*2}} = c_r^2 c_\alpha^2 \frac{c_K^2 \theta_1^2 \frakm^{*2} K}{\sigma^{*2} n\eps^2} = c_r^2 c_\alpha^2 \frac{c_K^2 \theta_1^2 \kappa^{*1/2} K}{n\eps^2} < 1,
    \end{align*}
    where the last inequality is condition~\eqref{ass.n_big}, then $\alpha_{K,c_\rho} < \sigma^*.$ We show $4\delta_{K,n}/\sigma^* < \alpha_{K,c_\rho}$ using 
    \begin{align}\label{eq.alpha_K_2delta_K}
        \frac{16 \delta_{K,n}^2}{\sigma^{*2}} = \frac{400 \frakm^{*4} K}{\sigma^{*2} n} = \kappa^{*1/2} \frac{400 \frakm^{*2} K}{n} < c_\alpha^2 \frac{384\theta_1^2 \frakm^{*2} K}{n\eps^2} = \alpha_{K,1},
    \end{align}
    where the only inequality is implied by condition \eqref{ass.c_alpha_lowbound}.
    
    By definition of $c_\mu$ in~\eqref{def.constants}, we have 
    \begin{align*}
        \iota_\mu c_\mu \geq \frac{c_\mu}{4} = 50 (c+2) \kappa_{+}^{1/2} = 2(c+2)\kappa_{+}^{1/2} + 48(c+2) \kappa_{+}^{1/2},
    \end{align*}
    thus~\eqref{ass.c_mu_lowbound} is satisfied since, as we show below,
    \begin{align*}
        \eps < \frac{ c_K^2 \theta_1^2 (c+2)}{800\kappa_{+}^{1/4}} = \frac{12 \theta_1^2 (c+2)}{25 \kappa_{+}^{1/4}}.
    \end{align*}
    
    With $c_K^2 = 384,$ we rewrite condition~\eqref{ass.1/theta_0_lowbound} as
    \begin{align*}
        \frac{50 \kappa_{+}^{1/2} \theta_0^2 \eps^2}{(c-2) \theta_1^2} + \frac{384 \theta_0^2 (c+2) \eps}{c-2} + \bigg( \frac{1+\frac{\sigma_{+}}{\sigma^*}}{3} \vee \frac{36}{10} \bigg) \frac{24 \theta_0^2 \iota_\mu c_\mu \eps}{c-2}  < 1.
    \end{align*}
    With the definition of $c_\mu$ in~\eqref{def.constants} and $\iota_\mu = 4,$ this becomes
    \begin{align*}
        \frac{50 \kappa_{+}^{1/2} \theta_0^2 }{(c-2) \theta_1^2}\eps^2 + \frac{48 \theta_0^2 (c+2) }{c-2} \left( 8 + \frac{400 \kappa_{+}^{1/2}}{3} \bigg( \left(1+\frac{\sigma_{+}}{\sigma^*}\right) \vee \frac{12}{10} \bigg) \right)\eps < 1.
    \end{align*}
    The inequality above has the form $A\eps^2 + B\eps < 1,$ which is satisfied by any $\eps$ smaller than $\min\{1/\sqrt{2A},\ 1/2B \}.$ The definition of $\eps$ in~\eqref{def.constants} coincides with imposing $\eps= c_\eps\cdot \min\{1/\sqrt{2A},\ 1/2B \} = c_\eps/2B,$ with $c_\eps = 1/2$ and
    \begin{align*}
        \frac{1}{\sqrt{2A}} &= \sqrt{c-2} \frac{\theta_1}{10 \theta_0 \kappa_{+}^{1/4}}, \\
        \frac{1}{2B} &= \frac{c-2}{96 \theta_0^2 (c+2) \big(8 + 134\kappa_{+}^{1/2} ((1+\frac{\sigma_{+}}{\sigma^*})\vee \frac{6}{5}) \big)},
    \end{align*}
    we have used that $400/3 < 134.$ Thus, condition~\eqref{ass.1/theta_0_lowbound} is satisfied. It is immediate to verify that this implies both~\eqref{ass.1/theta_0_lowbound_negative} and~\eqref{ass.1/theta_0_lowbound_2}.

    With $c_K^2 =384,$ the definition of $c_\mu$ in~\eqref{def.constants} and $\iota_\mu = 4,$ we rewrite~\eqref{ass.c_alpha_lowbound} as
    \begin{align*}
        c_\alpha^2 &> \frac{75\kappa_{+}^{1/2}}{16 \theta_1^2}\eps^2 + 108(c+2)\left( 1 + \frac{320}{3}\kappa_{+}^{1/2} \right) \eps.
    \end{align*}
    By the discussion on $\eps$ above, it is sufficient that, with $c_\eps = 1/2$ and $320/3 < 107,$
    \begin{align*}
        c_\alpha^2 &> \frac{75\kappa_{+}^{1/2}}{16 \theta_1^2} \cdot \frac{c_\eps^2 (c-2) \theta_1^2}{100 \theta_0^2 \kappa_{+}^{1/2}} + 108(c+2)\left( 1 + 107 \kappa_{+}^{1/2} \right) \frac{c_\eps (c-2)}{96 \theta_0^2 (c+2) \big(8 + 134\kappa_{+}^{1/2} ((1+\frac{\sigma_{+}}{\sigma^*})\vee \frac{6}{5}) \big)}.
    \end{align*}
    This is equivalent to
    \begin{align*}
        c_\alpha^2 &> \frac{15 (c-2)}{320 \theta_0^2} c_\eps^2 + \frac{27 (c-2) ( 1 + 107\kappa_{+}^{1/2})}{24 \theta_0^2 \big(8 + 134\kappa_{+}^{1/2} ((1+\frac{\sigma_{+}}{\sigma^*})\vee \frac{6}{5}) \big)} c_\eps,
    \end{align*}
    and, with 
    \begin{align*}
        \frac{ 1 + 107\kappa_{+}^{1/2}}{8 + 134\kappa_{+}^{1/2} ((1+\frac{\sigma_{+}}{\sigma^*})\vee \frac{6}{5})} < 1,
    \end{align*}
    and $c_\eps = 1/2,$ condition~\eqref{ass.c_alpha_lowbound} holds if
    \begin{align*}
        c_\alpha^2 &\geq \frac{15 (c-2)}{320 \theta_0^2} c_\eps^2 + \frac{27 (c-2)}{24 \theta_0^2} c_\eps = \frac{(c-2)}{16 \theta_0^2} \left( \frac{15}{80} + \frac{27}{3}\right) = \frac{441 (c-2)}{768 \theta_0^2}.
    \end{align*}
    This is exactly the case from the definition of $c_\alpha$ in~\eqref{def.constants}, since $3/5 > 441/768.$ The proof is complete.
\end{proof}

\subsection{Control of the supremum of \texorpdfstring{$T_{K, \mu}(g, \chi, f^*, \sigma^*)$}{T(K,mu)}} \label{sec.proofs_bound_supg_Fkappa}

With $\sigma_{+}$ the known upper bound on $\sigma^*,$ set $I_{+} = (0,\sigma_{+}]$ and, with $r(\cdot)$ any function such that $r(\rho) \geq \{r_P(\rho,\gamma_P), r_M(\rho,\gamma_M)\},$ any $c_\rho \in \{1,2\}$ and $\alpha_{K,c_\rho} = c_\alpha r(c_\rho\rho_K),$ let us define
\begin{align*}
    \mF^{(c_\rho)}_1 &:= \{ (g, \chi) \in \mF \times I_{+} :
    \|g - f^*\| \leq c_\rho \rho_K,
    \|g - f^*\|_{2,\bX} \leq r(c_\rho \rho_K),\ 
    |\sigma^* - \chi| \leq \alpha_{K,c_\rho}
    \} \\
    \mF^{(c_\rho)}_2 &:= \{ (g, \chi) \in \mF \times I_{+} :
    \|g - f^*\| \leq c_\rho \rho_K,
    \|g - f^*\|_{2,\bX} > r(c_\rho \rho_K),\ 
    |\sigma^* - \chi| \leq \alpha_{K,c_\rho}
    \} \\
    \mF^{(c_\rho)}_3 &:= \{ (g, \chi) \in \mF \times I_{+} :
    \|g - f^*\| > c_\rho \rho_K,\ 
    |\sigma^* - \chi| \leq \alpha_{K,c_\rho}
    \} \\
    \mF^{(c_\rho)}_4 &:= \{ (g, \chi) \in \mF \times I_{+} :
    \|g - f^*\| \leq c_\rho \rho_K,
    \|g - f^*\|_{2,\bX} \leq r(c_\rho \rho_K),\ 
    \chi > \sigma^* + \alpha_{K,c_\rho}
    \} \\
    \mF^{(c_\rho)}_5 &:= \{ (g, \chi) \in \mF \times I_{+} :
    \|g - f^*\| \leq c_\rho \rho_K,
    \|g - f^*\|_{2,\bX} > r(c_\rho \rho_K),\ 
    \chi > \sigma^* + \alpha_{K,c_\rho}
    \} \\
    \mF^{(c_\rho)}_6 &:= \{ (g, \chi) \in \mF \times I_{+} :
    \|g - f^*\| > c_\rho \rho_K,\ 
    \chi > \sigma^* + \alpha_{K,c_\rho}
    \} \\
    \mF^{(c_\rho)}_7 &:= \{ (g, \chi) \in \mF \times I_{+} :
    \|g - f^*\| \leq c_\rho \rho_K,
    \|g - f^*\|_{2,\bX} \leq r(c_\rho \rho_K),\ 
    \chi < \sigma^* - \alpha_{K,c_\rho}
    \} \\
    \mF^{(c_\rho)}_8 &:= \{ (g, \chi) \in \mF \times I_{+} :
    \|g - f^*\| \leq c_\rho \rho_K,
    \|g - f^*\|_{2,\bX} > r(c_\rho \rho_K),\ 
    \chi < \sigma^* - \alpha_{K,c_\rho}
    \} \\
    \mF^{(c_\rho)}_9 &:= \{ (g, \chi) \in \mF \times I_{+} :
    \|g - f^*\| > c_\rho \rho_K,\ 
    \chi < \sigma^* - \alpha_{K,c_\rho}
    \}.
\end{align*}
The sets above are a partition of the domain $\mF\times I_+$ where the functional
\begin{align*}
    T_{K, \mu}(g, \chi, f^*, \sigma^*) = MOM_K\Big( R_c(\ell_g,\chi,\ell_{f^*},\sigma^*) \Big) + \mu (\|f^*\| - \|g\|)
\end{align*}
takes inputs. For $c_\rho\in\{1,2\}$ and $i=1,\ldots,9,$ we set $B_{i,c_\rho}$ some upper bound for the supremum of $T_{K, \mu}(g, \chi, f^*, \sigma^*)$ over $(g,\chi)\in\mF^{(c_\rho)}_i.$ That is,
\begin{align}\label{def.B_ik}
    \sup_{(g,\chi)\in\mF^{(c_\rho)}_i} T_{K, \mu}(g, \chi, f^*, \sigma^*) \leq B_{i,c_\rho},
\end{align}
and the goal of this section is to give sharp bounds for each slice separately. Using the definition of $R_c(\ell_g,\chi,\ell_{f^*},\sigma^*)$ in~\eqref{def.R_functional}, and $\ell_g = \ell_{f^*} + \ell_g - \ell_{f^*},$ we find
\begin{align*}
    R_c(\ell_g,\chi,\ell_{f^*},\sigma^*) &= (\sigma^*-\chi)\bigg(1-2\frac{\ell_{f^*}+\ell_g}{(\sigma^*+\chi)^2} \bigg) + 2c\frac{\ell_{f^*}-\ell_g}{\sigma^*+\chi} \\
    &= (\sigma^*-\chi) \bigg( 1 - \frac{4 \ell_{f^*}}{(\sigma^*+\chi)^2} \bigg) + 2\frac{\ell_{f^*} - \ell_g}{\sigma^*+\chi}\bigg(c + \frac{\sigma^*-\chi}{\sigma^*+\chi}\bigg) \\
    &= R_c(\ell_{f^*},\chi,\ell_{f^*},\sigma^*) + 2\Delta_c(\chi,\sigma^*) \frac{\ell_{f^*} - \ell_g}{\sigma^*+\chi}.
\end{align*}
with 
\begin{align}\label{def.Delta_c}
    \Delta_c(\chi,\sigma) := \bigg(c + \frac{\sigma-\chi}{\sigma+\chi}\bigg) \in [c-1,c+1],\quad \forall \sigma,\chi\in(0,+\infty),
\end{align}
and $c>2$ by construction.
We plug this into the functional $T_{K, \mu}(g, \chi, f^*, \sigma^*),$ so that 
\begin{align*}
    T_{K, \mu}&(g, \chi, f^*, \sigma^*) = MOM_K \bigg( R_c(\ell_{f^*},\chi,\ell_{f^*},\sigma^*) + 2\Delta_c(\chi,\sigma^*) \frac{\ell_{f^*} - \ell_g}{\sigma^*+\chi} \bigg)
    + \mu (\|f^*\| - \|g\|).
\end{align*}
For all $(\bx,y)\in\mX\times\R,$ we have the decomposition
\begin{align*}
    \ell_f(\bx,y) - \ell_g(\bx,y) = 2\big(y-f(\bx)\big) \big(g(\bx)-f(\bx) \big) - \big(g(\bx) - f(\bx) \big)^2,
\end{align*}
and this gives $\ell_{f^*}-\ell_g = 2\zeta(g-f^*)-(g-f^*)^2.$ By the triangular quantile property in Lemma~\ref{lemma.quantile_prop}, we can write
\begin{align}\label{eq.T_bound_main}
\begin{split}
    T_{K, \mu}&(g, \chi, f^*, \sigma^*) \\
    &= Q_{3/4,K}\bigg[ R_c(\ell_{f^*},\chi,\ell_{f^*},\sigma^*) + 2\Delta_c(\chi,\sigma^*) \frac{\ell_{f^*} - \ell_g}{\sigma^*+\chi} \bigg]
    + \mu (\|f^*\| - \|g\|) \\
    &\leq Q_{3/4,K}\big[ R_c(\ell_{f^*},\chi,\ell_{f^*},\sigma^*) \big] + \frac{2 \Delta_c(\chi,\sigma^*)}{(\sigma^*+\chi)} Q_{3/4,K}\Big[ 2\zeta(g-f^*)-(g-f^*)^2 \Big] \\
    &\quad + \mu (\|f^*\| - \|g\|).
\end{split}
\end{align}
By arguing as in the proof of Lemma~\ref{lemma.bound_P2zeta_f_fstar}, see bound for \eqref{eq.Q7/8_control}, the quantity
\begin{align*}
    Q_{3/4,K}\big[ R_c(\ell_{f^*},\chi,\ell_{f^*},\sigma^*) \big] &= Q_{3/4,K}\Bigg[ (\sigma^*-\chi) \bigg( 1 - \frac{4 \ell_{f^*}}{(\sigma^*+\chi)^2} \bigg) \Bigg]
\end{align*}
is bounded above, when $\chi\geq\sigma^*,$ by
\begin{align}
\begin{split}
    \label{eq.bound_Q_3_4_R>}
    Q_{3/4,K}\big[ R_c(\ell_{f^*},\chi,\ell_{f^*},\sigma^*) \big] &\leq (\chi-\sigma^*) \Big( \frac{4\sigma^{*2}+4\delta_{K,n}}{(\sigma^*+\chi)^2} - 1 \Big),
\end{split}
\end{align}
or, when $\chi\leq\sigma^*,$ by 
\begin{align}
\begin{split}
    \label{eq.bound_Q_3_4_R<}
    Q_{3/4,K}\big[ R_c(\ell_{f^*},\chi,\ell_{f^*},\sigma^*) \big] &\leq (\sigma^*-\chi) \Big( 1 - \frac{4\sigma^{*2}-4\delta_{K,n}}{(\sigma^*+\chi)^2} \Big).
\end{split}
\end{align}

The following lemmas show that, on the event $\Omega(K),$ one can choose bounds $B_{i,c_\rho} $ in \eqref{def.B_ik}, for $i=1,\ldots,9$ and $c_\rho \in\{1,2\},$ as 
\begin{align*}
    B_{1,c_\rho} &= \frac{16}{\sigma^*(2\sigma^*-\alpha_{K,c_\rho})^2} \delta_{K,n}^2 + \frac{8(c+2)\eps }{2\sigma^*-\alpha_{K,c_\rho}} r^2(c_\rho \rho_K) + \frac{c_\mu\eps c_\rho}{\frakm^*}r^2(\rho_K), \\
    B_{2,c_\rho} &= \frac{16}{\sigma^*(2\sigma^*-\alpha_{K,c_\rho})^2} \delta_{K,n}^2
        + 2(c-2) \frac{4\eps - (4\theta_0)^{-2}}{2\sigma^*+\alpha_{K,c_\rho}} r^2(c_\rho \rho_K)  + \frac{c_\mu\eps c_\rho}{\frakm^*}r^2(\rho_K), \\
    B_{3,c_\rho} &= \max \Bigg\{ \frac{16}{\sigma^*(2\sigma^*-\alpha_{K,c_\rho})^2} \delta_{K,n}^2 + c_\rho \bigg( \frac{ 8(c+2)\eps }{2\sigma^*-\alpha_{K,c_\rho}} - \frac{4 c_\mu \eps}{5 \frakm^*}\bigg) r^2(\rho_K) + \frac{c_\mu\eps}{10\frakm^*} r^2(\rho_K), 
         \\
    &\hspace{1.5cm}
    \frac{16}{\sigma^*(2\sigma^*-\alpha_{K,c_\rho})^2} \delta_{K,n}^2 + c_\rho \bigg(2(c-2)\frac{4\eps - (4\theta_0)^{-2}}{2\sigma^*+\alpha_{K,c_\rho}}  + \frac{c_\mu\eps}{\frakm^*} \bigg)r^2(\rho_K) + \frac{c_\mu\eps}{10\frakm^*}r^2(\rho_K) \Bigg\}, \\
    B_{4,c_\rho} &= -\frac{2\sigma^*}{(2\sigma^*+\alpha_{K,c_\rho})^2} \alpha_{K,c_\rho}^2 + \frac{8(c+2) \eps}{2\sigma^*+\alpha_{K,c_\rho}} r^2(c_\rho \rho_K) + \frac{c_\mu\eps c_\rho}{\frakm^*} r^2(\rho_K), \\
    B_{5,c_\rho} &= -\frac{2\sigma^*}{(2\sigma^*+\alpha_{K,c_\rho})^2} \alpha_{K,c_\rho}^2 + 2(c-2) \frac{4\eps - (4\theta_0)^{-2}}{\sigma^*+\sigma_{+}} r^2(c_\rho \rho_K)  + \frac{c_\mu\eps c_\rho}{\frakm^*}r^2(\rho_K), \\
    B_{6,c_\rho} &= \max \Bigg\{ -\frac{2\sigma^*}{(2\sigma^*+\alpha_{K,c_\rho})^2} \alpha_{K,c_\rho}^2 + c_\rho \bigg( \frac{ 8(c+2)\eps }{2\sigma^*+\alpha_{K,c_\rho}} - \frac{4 c_\mu \eps}{5 \frakm^*}\bigg) r^2(\rho_K)
    + \frac{c_\mu\eps}{10\frakm^*} r^2(\rho_K), \\
    &\hspace{1.5cm}
    -\frac{2\sigma^*}{(2\sigma^*+\alpha_{K,c_\rho})^2} \alpha_{K,c_\rho}^2 + c_\rho \bigg(2(c-2)\frac{4\eps - (4\theta_0)^{-2}}{\sigma^*+\sigma_{+}}  + \frac{c_\mu\eps}{\frakm^*} \bigg)r^2(\rho_K) + \frac{c_\mu\eps}{10\frakm^*}r^2(\rho_K) \Bigg\}, \\
    B_{7,c_\rho} &= -\frac{2\sigma^*}{(2\sigma^*-\alpha_{K,c_\rho})^2} \alpha_{K,c_\rho}^2 + \frac{8(c+2)\eps}{\sigma^*} r^2(c_\rho \rho_K) + \frac{c_\mu\eps c_\rho}{\frakm^*} r^2(\rho_K), \\
    B_{8,c_\rho} &= -\frac{2\sigma^*}{(2\sigma^*-\alpha_{K,c_\rho})^2} \alpha_{K,c_\rho}^2 + 2(c-2) \frac{4\eps - (4\theta_0)^{-2}}{2\sigma^*-\alpha_{K,c_\rho}} r^2(c_\rho \rho_K)  + \frac{c_\mu\eps c_\rho}{\frakm^*}r^2(\rho_K), \\
    B_{9,c_\rho} &= \max \Bigg\{ -\frac{2\sigma^*}{(2\sigma^*-\alpha_{K,c_\rho})^2} \alpha_{K,c_\rho}^2 + \bigg( \frac{ 8(c+2)\eps c_\rho }{\sigma^*} - \frac{4 c_\mu \eps c_\rho}{5 \frakm^*} + \frac{c_\mu\eps}{10\frakm^*} \bigg) r^2(\rho_K), \\
    &\hspace{1.5cm}
    -\frac{2\sigma^*}{(2\sigma^*-\alpha_{K,c_\rho})^2} \alpha_{K,c_\rho}^2 + c_\rho \bigg(2(c-2)\frac{4\eps - (4\theta_0)^{-2}}{2\sigma^*-\alpha_{K,c_\rho}}  + \frac{c_\mu\eps}{\frakm^*} \bigg)r^2(\rho_K) + \frac{c_\mu\eps}{10\frakm^*}r^2(\rho_K) \Bigg\}.
\end{align*}

\begin{lem}\label{lemma.bound_supg_Fkappa_1}
    On the event $\Omega(K),$ for all $c_\rho \in \{1, 2\},$ the supremum of $T_{K, \mu}(g, \chi, f^*, \sigma^*)$ over the set
    \begin{align*}
        \mF^{(c_\rho)}_1 &:= \{ (g, \chi) \in \mF \times I_{+} :
        \|g - f^*\| \leq c_\rho \rho_K,
        \|g - f^*\|_{2,\bX} \leq r(c_\rho \rho_K),\ 
        |\sigma^* - \chi| \leq \alpha_{K,c_\rho}
        \},
    \end{align*}
    is bounded above by 
    \begin{align*}
        B_{1,c_\rho} 
        &
        := \frac{16}{\sigma^*(2\sigma^*-\alpha_{K,c_\rho})^2} \delta_{K,n}^2 + \frac{8(c+2)\eps }{2\sigma^*-\alpha_{K,c_\rho}} r^2(c_\rho \rho_K) + \frac{c_\mu\eps c_\rho}{\frakm^*}r^2(\rho_K).
    \end{align*}
    
\end{lem}
\begin{proof}[Proof of Lemma~\ref{lemma.bound_supg_Fkappa_1}]
    Let $(g,\chi) \in \mF^{(c_\rho)}_1.$ Using the bound obtained in~\eqref{eq.T_bound_main}, the inequality $(g - f^*)^2 \geq 0$ and the triangular inequality, the quantity $T_{K, \mu} (g, \chi, f^*, \sigma^*)$ is bounded above by
    \begin{align*}
        Q_{3/4,K}& \big[ R_c(\ell_{f^*},\chi,\ell_{f^*},\sigma^*) \big]
        + \frac{2 \Delta_c(\chi,\sigma^*)}{\sigma^*+\chi} Q_{3/4,K}\big[2\zeta(g-f^*)-(g-f^*)^2 \big]
        + \mu (\|f^*\| - \|g\|) \\
        &\leq Q_{3/4,K}\big[ R_c(\ell_{f^*},\chi,\ell_{f^*},\sigma^*) \big] 
        + \frac{2 \Delta_c(\chi,\sigma^*) }{\sigma^*+\chi} Q_{3/4,K}\big[ 2\zeta(g-f^*)\big] + \mu \|f^*-g\|.
    \end{align*}
    By Lemma~\ref{lemma.useful_bounds_lecue}, $Q_{3/4,K}[ 2\zeta(g-f^*)] \leq \alpha_M^2 \leq 4\eps r^2(c_\rho \rho_K)$ and, with $\Delta_c(\chi,\sigma^*)\leq c+2,$ we find 
    \begin{align*}
        T_{K, \mu} (g, \chi, f^*, \sigma^*) &\leq Q_{3/4,K}\big[ R_c(\ell_{f^*},\chi,\ell_{f^*},\sigma^*) \big] + \frac{8(c+2)\eps }{2\sigma^*-\alpha_{K,c_\rho}} r^2(c_\rho \rho_K) + \mu c_\rho \rho_K \\
        &= Q_{3/4,K}\big[ R_c(\ell_{f^*},\chi,\ell_{f^*},\sigma^*) \big] + \frac{8(c+2)\eps }{2\sigma^*-\alpha_{K,c_\rho}} r^2(c_\rho \rho_K) + \frac{c_\mu\eps c_\rho}{\frakm^*} r^2(\rho_K),
    \end{align*} 
    where in the last step we put our choice $\mu = (c_\mu\eps/\frakm^*) r^2(\rho_K)/\rho_K.$ We now bound the quantile term appearing in the latter display. Directly from~\eqref{eq.bound_Q_3_4_R>} and~\eqref{eq.bound_Q_3_4_R<}, we get
    \begin{align*}
        &Q_{3/4,K}\big[ R_c(\ell_{f^*},\chi,\ell_{f^*},\sigma^*) \big]\\
        &\quad \leq \max\bigg\{ \sup_{\chi\in[\sigma^*,\sigma^*+\alpha_{K,c_\rho}]} |\sigma^*-\chi| \Big( \frac{4\sigma^{*2}+4\delta_{K,n}}{(\sigma^*+\chi)^2} - 1 \Big),\ \sup_{\chi\in[\sigma^*-\alpha_{K,c_\rho},\sigma^*]} |\sigma^*-\chi| \Big( 1 - \frac{4\sigma^{*2}-4\delta_{K,n}}{(\sigma^*+\chi)^2} \Big)  \bigg\}.
    \end{align*}
    By arguing as in the proof of Lemma~\ref{lemma.bound_P2zeta_f_fstar}, see bounds on \eqref{eq.Q7/8_control}, with $\alpha_{K,c_\rho} > 2\delta_{K,n}/\sigma^*$ we obtain
    \begin{align*}
        T_{K, \mu} (g, \chi, f^*, \sigma^*) 
        &\leq \frac{16}{\sigma^*(2\sigma^*-\alpha_{K,c_\rho})^2} \delta_{K,n}^2 + \frac{8(c+2)\eps }{2\sigma^*-\alpha_{K,c_\rho}} r^2(c_\rho \rho_K) + \frac{c_\mu\eps c_\rho}{\frakm^*}r^2(\rho_K),
    \end{align*}
    which is what we wanted.
\end{proof}

\begin{lem}\label{lemma.bound_supg_Fkappa_2}
    On the event $\Omega(K),$ for all $c_\rho \in \{1, 2\},$ the supremum of $T_{K, \mu}(g, \chi, f^*, \sigma^*)$ over the set
    \begin{align*}
        \mF^{(c_\rho)}_2 &:= \{ (g, \chi) \in \mF \times I_{+} :
        \|g - f^*\| \leq c_\rho \rho_K,
        \|g - f^*\|_{2,\bX} > r(c_\rho \rho_K),\ 
        |\sigma^* - \chi| \leq \alpha_{K,c_\rho}
        \},
    \end{align*}
    is bounded above by 
    \begin{align*}
        B_{2,c_\rho} 
        &
        := \frac{16}{\sigma^*(2\sigma^*-\alpha_{K,c_\rho})^2} \delta_{K,n}^2 
        + 2(c-2) \frac{4\eps - (4\theta_0)^{-2}}{2\sigma^*+\alpha_{K,c_\rho}} r^2(c_\rho \rho_K)  + \frac{c_\mu\eps c_\rho}{\frakm^*}r^2(\rho_K).
    \end{align*}
\end{lem}
\begin{proof}[Proof of Lemma~\ref{lemma.bound_supg_Fkappa_2}]
    Let $(g, \chi) \in \mF^{(c_\rho)}_2.$ The space $\mF^{(c_\rho)}_2$ shares with $\mF^{(c_\rho)}_1$ the conditions $\|g-f^*\|\leq c_\rho \rho_K$ and $|\chi-\sigma^*|\leq\alpha_{K,c_\rho}.$ By arguing as in the proof of Lemma~\ref{lemma.bound_supg_Fkappa_1}, we know already that
    \begin{align*}
        T_{K, \mu} &(g, \chi, f^*, \sigma^*) \\
        &\leq \frac{16}{\sigma^*(2\sigma^*-\alpha_{K,c_\rho})^2} \delta_{K,n}^2
        + \frac{2\Delta_c(\chi,\sigma^*)}{(\sigma^*+\chi)} Q_{3/4,K}\big[ 2\zeta(g-f^*)-(g-f^*)^2 \big]  + \frac{c_\mu\eps c_\rho}{\frakm^*}r^2(\rho_K).
    \end{align*}
    An application of Lemma~\ref{lemma.useful_bounds_lecue} bounds from above the quantiles of $2\zeta(g-f^*)$ and from below the quantiles of $(g-f^*)^2,$ together with $\Delta_c(\chi,\sigma^*) \geq c-2$ this leads to
    \begin{align*}
        T_{K, \mu} &(g, \chi, f^*, \sigma^*) \\
        &\leq \frac{16}{\sigma^*(2\sigma^*-\alpha_{K,c_\rho})^2} \delta_{K,n}^2
        + \frac{2\Delta_c(\chi,\sigma^*)}{(\sigma^*+\chi)} Q_{3/4,K}\big[ 2\zeta(g-f^*)-(g-f^*)^2 \big]  + \frac{c_\mu\eps c_\rho}{\frakm^*}r^2(\rho_K) \\
        &\leq \frac{16}{\sigma^*(2\sigma^*-\alpha_{K,c_\rho})^2} \delta_{K,n}^2
        + \frac{2\Delta_c(\chi,\sigma^*)}{(\sigma^*+\chi)} \big(\alpha_M^2 - (4\theta_0)^{-2}\|g-f^*\|_{2,\bX}^2 \big)  + \frac{c_\mu\eps c_\rho}{\frakm^*}r^2(\rho_K) \\
        &\leq \frac{16}{\sigma^*(2\sigma^*-\alpha_{K,c_\rho})^2} \delta_{K,n}^2 
        + 2(c-2) \frac{4\eps - (4\theta_0)^{-2}}{2\sigma^*+\alpha_{K,c_\rho}} r^2(c_\rho \rho_K)  + \frac{c_\mu\eps c_\rho}{\frakm^*}r^2(\rho_K),
    \end{align*}
    since $4 \eps < 1/(4 \theta_0)^2$ by condition~\eqref{ass.1/theta_0_lowbound_negative}, so
    $\alpha_M^2 - \|g - f^*\|_{2,\bX}^2 (4 \theta_0)^{-2}
    \leq (4 \eps - (4 \theta_0)^{-2}) r^2(c_\rho \rho_K).$
\end{proof}

\begin{lem}\label{lemma.bound_supg_Fkappa_3}
    On the event $\Omega(K),$ for all $c_\rho \in \{1, 2\},$ the supremum of $T_{K, \mu}(g, \chi, f^*, \sigma^*)$ over the set
    \begin{align*}
        \mF^{(c_\rho)}_3 &:= \{ (g, \chi) \in \mF \times I_{+} :
        \|g - f^*\| > c_\rho \rho_K,\ 
        |\sigma^* - \chi| \leq \alpha_{K,c_\rho}
        \},
    \end{align*}
    is bounded above by 
    \begin{align*}
        B_{3,c_\rho} 
        &
        := \max \Bigg\{ \frac{16}{\sigma^*(2\sigma^*-\alpha_{K,c_\rho})^2} \delta_{K,n}^2 + c_\rho \bigg( \frac{ 8(c+2)\eps }{2\sigma^*-\alpha_{K,c_\rho}} - \frac{4 c_\mu \eps}{5 \frakm^*}\bigg) r^2(\rho_K)
        + \frac{c_\mu\eps}{10\frakm^*} r^2(\rho_K), 
         \\
        &\hspace{1.5cm}
         \frac{16}{\sigma^*(2\sigma^*-\alpha_{K,c_\rho})^2} \delta_{K,n}^2 + c_\rho \bigg(2(c-2)\frac{4\eps - (4\theta_0)^{-2}}{2\sigma^*+\alpha_{K,c_\rho}}  + \frac{c_\mu\eps}{\frakm^*} \bigg)r^2(\rho_K) + \frac{c_\mu\eps}{10\frakm^*}r^2(\rho_K) \Bigg\}.
    \end{align*}
\end{lem}
\begin{proof}[Proof of Lemma~\ref{lemma.bound_supg_Fkappa_3}]
    Let $(g, \chi) \in \mF^{(c_\rho)}_3$. The space $\mF^{(c_\rho)}_3$ shares with $ \mF^{(c_\rho)}_1, \mF^{(c_\rho)}_2$ the constraint $|\chi-\sigma^*|\leq \alpha_{K,c_\rho}.$ By arguing as in the proofs of Lemma~\ref{lemma.bound_supg_Fkappa_1} and Lemma~\ref{lemma.bound_supg_Fkappa_2}, the bound in~\eqref{eq.T_bound_main} becomes
    \begin{align*}
        T_{K, \mu} &(g, \chi, f^*, \sigma^*) \\
        &\leq \frac{16}{\sigma^*(2\sigma^*-\alpha_{K,c_\rho})^2} \delta_{K,n}^2 + \frac{2\Delta_c(\chi,\sigma^*)}{(\sigma^*+\chi)} Q_{3/4,K}\big[ 2\zeta(g-f^*)-(g-f^*)^2 \big]  + \mu (\|f^*\| - \|g\|) \\
        &\leq \frac{16}{\sigma^*(2\sigma^*-\alpha_{K,c_\rho})^2} \delta_{K,n}^2 + \frac{2\Delta_c(\chi,\sigma^*)}{(\sigma^*+\chi)} Q_{3/4,K}\big[ 2\zeta(g-f^*)-(g-f^*)^2 \big]  \\
        &\quad - \mu \sup_{z^* \in \Gamma_{f^*}(\rho_K)} z^* (g - f^*)
        + \frac{\mu \rho_K}{10},
    \end{align*}
    where the last inequality follows from the application of Lemma~\ref{lemma.6} with $\rho = \rho_K$. We follow now the proof of Lemma 5 in~\cite{lecue2020robustML}. Let us define $f := f^* + \rho_K(g - f^*) / \|g - f^*\|,$ this function belongs to the function class $\mF$ by convexity. Let $\Upsilon := \|g - f^*\| / \rho_K.$ By construction, $\|f - f^*\| = \rho_K$ and $g - f^* = \Upsilon (f - f^*).$ Then,
    \begin{align*}
        T_{K, \mu} &(g, \chi, f^*, \sigma^*) \\
        &\leq \frac{16}{\sigma^*(2\sigma^*-\alpha_{K,c_\rho})^2} \delta_{K,n}^2 + \frac{2\Upsilon \Delta_c(\chi,\sigma^*)}{(\sigma^*+\chi)} Q_{3/4,K}\big[ 2\zeta(f-f^*)-(f-f^*)^2 \big]  \\
        &\quad - \mu \Upsilon \sup_{z^* \in \Gamma_{f^*}(\rho_K)} z^* (f - f^*) + \frac{\mu \rho_K}{10}.
    \end{align*}
    From here, we separate the cases $\|f - f^*\|_{2,\bX} \leq r(\rho_K)$ and $\|f - f^*\|_{2,\bX} > r(\rho_K).$
    
    We start with $\|f - f^*\|_{2,\bX} \leq r(\rho_K).$ Since $\|f-f^*\|= \rho_K,$ we have $f \in H_{\rho_K}$ with $H_{\rho_K} = \{f\in\mF:\|f-f^*\|\leq\rho_K,\ \|f-f^*\|_{2,\bX}\leq r(\rho_K)\}$ defined in Section~\ref{sec.main_sparsity}. Recall that $K^*$ is defined as the smallest integer satisfying $K^*\geq n\eps r^2(\rho^*)/c_K^2 \theta_m^2,$ with $\rho^*$ the smallest value $\rho>0$ satisfying the sparsity inequality
    \begin{align*}
        \inf_{f\in H_\rho} \sup_{z^*\in\Gamma_{f^*}(\rho_K)} z^*(f-f^*) \geq \frac{4}{5}\rho.
    \end{align*}
    Since $K \geq K^*,$ we get $\rho_K \geq \rho^*$ and $\rho_K$ satisfies the sparsity inequality
    \begin{align*}
        \sup_{z^* \in \Gamma_{f^*}(\rho_K)} z^* (f - f^*) \geq \frac{4}{5}\rho_K.
    \end{align*}
    Using our choice of $\mu = (c_\mu\eps/\frakm^*)r^2(\rho_K)/\rho_K,$ we get
    \begin{align*}
        - \mu \sup_{z^* \in \Gamma_{f^*}(\rho_K)} z^* (f - f^*)
        \leq - \frac{4 c_\mu \eps}{5 \frakm^*} r^2(\rho_K).
    \end{align*}
    The latter display, the fact that $(f - f^*)^2 \geq 0,$ the bound $\Delta_c(\chi,\sigma^*)\leq c+2,$ and the quantile bound $Q_{3/4,K}[ 2\zeta(f-f^*)] \leq \alpha_M^2 \leq 4\eps r^2(\rho_K)$ in Lemma~\ref{lemma.useful_bounds_lecue}, all together yield
    \begin{align*}
        T_{K, \mu} (g, \chi, f^*, \sigma^*) 
        &\leq \frac{16}{\sigma^*(2\sigma^*-\alpha_{K,c_\rho})^2} \delta_{K,n}^2 + \Upsilon \bigg( \frac{ 8(c+2)\eps }{2\sigma^*-\alpha_{K,c_\rho}} - \frac{4 c_\mu \eps}{5 \frakm^*}\bigg) r^2(\rho_K)
        + \frac{\mu \rho_K}{10}.
    \end{align*}
    By condition~\eqref{ass.c_mu_lowbound}, the term multiplied by $\Upsilon$ is negative. This is true because $\kappa_{+}^{1/4} \geq \kappa^{*1/4} = \frakm^*/\sigma^* > 1$ and
    \begin{align*}
        c_\mu > \frac{5\frakm^*(c+2)}{\sigma^*} \implies \frac{4c_\mu\eps}{5\frakm^*} > \frac{4(c+2)\eps}{\sigma^*} > \frac{ 4(c+2)\eps }{2\sigma^*-\alpha_{K,c_\rho}},
    \end{align*}
    the last inequality follows from $\alpha_{K,c_\rho} < \sigma^*,$ which is guaranteed by Lemma~\ref{lemma.conditions}. Since $\Upsilon > c_\rho,$ we have
    \begin{align*}
        T_{K, \mu} (g, \chi, f^*, \sigma^*) 
        &\leq \frac{16}{\sigma^*(2\sigma^*-\alpha_{K,c_\rho})^2} \delta_{K,n}^2 + c_\rho \bigg( \frac{ 8(c+2)\eps }{2\sigma^*-\alpha_{K,c_\rho}} - \frac{4 c_\mu \eps}{5 \frakm^*}\bigg) r^2(\rho_K)
        + \frac{c_\mu\eps}{10\frakm^*} r^2(\rho_K).
    \end{align*}
    This concludes the first part of the proof.
    
    We now consider the case $\|f-f^*\|_{2,\bX} > r(\rho_K).$ Since $\|f - f^*\| = \rho_K$ and $\Delta_c(\chi,\sigma^*)\geq c-2,$ an application of Lemma~\ref{lemma.useful_bounds_lecue} bounds from above the quantiles of $2\zeta(g-f^*)$ and from below the quantiles of $(g-f^*)^2,$ this gives 
    \begin{align*}
        T_{K, \mu} &(g, \chi, f^*, \sigma^*) \\
        &\leq \frac{16}{\sigma^*(2\sigma^*-\alpha_{K,c_\rho})^2} \delta_{K,n}^2 + \Upsilon \bigg(2(c-2)\frac{4\eps - (4\theta_0)^{-2}}{2\sigma^*+\alpha_{K,c_\rho}} r^2(\rho_K)  + \mu\rho_K \bigg) + \frac{\mu \rho_K}{10} \\
        &\leq \frac{16}{\sigma^*(2\sigma^*-\alpha_{K,c_\rho})^2} \delta_{K,n}^2 + c_\rho \bigg(2(c-2)\frac{4\eps - (4\theta_0)^{-2}}{2\sigma^*+\alpha_{K,c_\rho}}  + \frac{c_\mu\eps}{\frakm^*} \bigg)r^2(\rho_K) + \frac{c_\mu\eps}{10\frakm^*}r^2(\rho_K),
    \end{align*}
    using that $\Upsilon>c_\rho$ and the term multiplied by $\Upsilon$ is negative, by condition~\eqref{ass.1/theta_0_lowbound_negative}. This can be seen by
    \begin{align*}
        \frac{1}{16\theta_0^2} &> 4\eps + \frac{(\sigma^*+\sigma_{+}) c_\mu\eps}{2(c-2)\frakm^*} \\
        &\quad \iff 0 > 2(c-2)\frac{4\eps - (4\theta_0)^{-2}}{\sigma^*+\sigma_{+}} + \frac{c_\mu\eps}{\frakm^*} > 2(c-2)\frac{4\eps - (4\theta_0)^{-2}}{2\sigma^*+\alpha_{K,c_\rho}}  + \frac{c_\mu\eps}{\frakm^*}.
    \end{align*}
    This concludes the second part of the proof.
\end{proof}

\begin{lem}\label{lemma.bound_supg_Fkappa_4}
    On the event $\Omega(K),$ for all $c_\rho \in \{1, 2\},$ the supremum of $T_{K, \mu}(g, \chi, f^*, \sigma^*)$ over the set
    \begin{align*}
        \mF^{(c_\rho)}_4 &:= \{ (g, \chi) \in \mF \times I_{+} :
        \|g - f^*\| \leq c_\rho \rho_K,
        \|g - f^*\|_{2,\bX} \leq r(c_\rho \rho_K),\ 
        \chi > \sigma^* + \alpha_{K,c_\rho}
        \},
    \end{align*}
    is bounded above by 
    \begin{align*}
        B_{4,c_\rho} 
        &
        := -\frac{2\sigma^*}{(2\sigma^*+\alpha_{K,c_\rho})^2} \alpha_{K,c_\rho}^2 + \frac{8(c+2) \eps}{2\sigma^*+\alpha_{K,c_\rho}} r^2(c_\rho \rho_K) + \frac{c_\mu\eps c_\rho}{\frakm^*} r^2(\rho_K).
    \end{align*}
\end{lem}
\begin{proof}[Proof of Lemma~\ref{lemma.bound_supg_Fkappa_4}]
    Let $(g, \chi) \in \mF^{(c_\rho)}_4$. The space $\mF^{(c_\rho)}_4$ shares with $\mF^{(c_\rho)}_1$ the conditions $\|g-f^*\|\leq c_\rho \rho_K$ and $\|g-f^*\|_{2,\bX}\leq r(c_\rho \rho_K).$ By arguing as in the proof of Lemma~\ref{lemma.bound_supg_Fkappa_1} and using that $\chi > \sigma^* + \alpha_{K,c_\rho},$ from~\eqref{eq.bound_Q_3_4_R>} we get
    \begin{align*}
        T_{K, \mu} &(g, \chi, f^*, \sigma^*) \\
        &\leq \sup_{\chi>\sigma^*+\alpha_{K,c_\rho}} (\chi-\sigma^*) \Big( \frac{4(\sigma^{*2}+\delta_{K,n})}{(\sigma^*+\chi)^2} - 1 \Big) + \frac{8(c+2) \eps}{2\sigma^*+\alpha_{K,c_\rho}} r^2(c_\rho \rho_K) + \frac{c_\mu\eps c_\rho}{\frakm^*} r^2(\rho_K) \\
        &= -\alpha_{K,c_\rho} \Big(1 - \frac{8(\sigma^{*2}+\delta_{K,n})}{(2\sigma^*+\alpha_{K,c_\rho})^2} \Big) + \frac{8(c+2) \eps}{2\sigma^*+\alpha_{K,c_\rho}} r^2(c_\rho \rho_K) + \frac{c_\mu\eps c_\rho}{\frakm^*} r^2(\rho_K).
    \end{align*}
    Since $\alpha_{K,c_\rho} > 2\delta_{K,n}/\sigma^*,$ one has
    \begin{align*}
        1 - \frac{4(\sigma^{*2}+\delta_{K,n})}{(2\sigma^*+\alpha_{K,c_\rho})^2} = \frac{4(\sigma^*\alpha_{K,c_\rho}-\delta_{K,n})}{(2\sigma^*+\alpha_{K,c_\rho})^2} + \frac{\alpha_{K,c_\rho}^2}{(2\sigma^*+\alpha_{K,c_\rho})^2} > \frac{4(\sigma^*\alpha_{K,c_\rho}-\delta_{K,n})}{(2\sigma^*+\alpha_{K,c_\rho})^2} > \frac{2\sigma^*\alpha_{K,c_\rho}}{(2\sigma^*+\alpha_{K,c_\rho})^2},
    \end{align*}
    and
    \begin{align*}
        T_{K, \mu} (g, \chi, f^*, \sigma^*) 
        &\leq -\frac{2\sigma^*}{(2\sigma^*+\alpha_{K,c_\rho})^2} \alpha_{K,c_\rho}^2 + \frac{8(c+2) \eps}{2\sigma^*+\alpha_{K,c_\rho}} r^2(c_\rho \rho_K) + \frac{c_\mu\eps c_\rho}{\frakm^*} r^2(\rho_K).
    \end{align*}
    This is enough to conclude.
\end{proof}

\begin{lem}\label{lemma.bound_supg_Fkappa_5}
    On the event $\Omega(K),$ for all $c_\rho \in \{1, 2\},$ the supremum of $T_{K, \mu}(g, \chi, f^*, \sigma^*)$ over the set
    \begin{align*}
        \mF^{(c_\rho)}_5 &:= \{ (g, \chi) \in \mF \times I_{+} :
        \|g - f^*\| \leq c_\rho \rho_K,
        \|g - f^*\|_{2,\bX} > r(c_\rho \rho_K),\ 
        \chi > \sigma^* + \alpha_{K,c_\rho}
        \},
    \end{align*}
    is bounded above by 
    \begin{align*}
        B_{5,c_\rho} 
        &
        := -\frac{2\sigma^*}{(2\sigma^*+\alpha_{K,c_\rho})^2} \alpha_{K,c_\rho}^2 + 2(c-2) \frac{4\eps - (4\theta_0)^{-2}}{\sigma^*+\sigma_{+}} r^2(c_\rho \rho_K)  + \frac{c_\mu\eps c_\rho}{\frakm^*}r^2(\rho_K).
    \end{align*}
\end{lem}
\begin{proof}[Proof of Lemma~\ref{lemma.bound_supg_Fkappa_5}]
    Let $(g, \chi) \in \mF^{(c_\rho)}_5.$ The space $\mF^{(c_\rho)}_5$ shares with $\mF^{(c_\rho)}_1$ the condition $\|g-f^*\|\leq c_\rho \rho_K,$ with $\mF^{(c_\rho)}_2$ the condition $\|g-f^*\|_{2,\bX}>r(\rho_K),$ and with $\mF^{(c_\rho)}_4$ the condition $\chi > \sigma^* + \alpha_{K,c_\rho}.$ By arguing as in the proofs of Lemma~\ref{lemma.bound_supg_Fkappa_1}, Lemma~\ref{lemma.bound_supg_Fkappa_2} and Lemma~\ref{lemma.bound_supg_Fkappa_4}, one gets
    \begin{align*}
        T_{K, \mu} (g, \chi, f^*, \sigma^*) 
        \leq  -\frac{2\sigma^*}{(2\sigma^*+\alpha_{K,c_\rho})^2} \alpha_{K,c_\rho}^2 + 2(c-2) \frac{4\eps - (4\theta_0)^{-2}}{\sigma^*+\sigma_{+}} r^2(c_\rho \rho_K)  + \frac{c_\mu\eps c_\rho}{\frakm^*}r^2(\rho_K),
    \end{align*}
    where $\sigma_{+}$ is the upper bound on $\chi.$ 
\end{proof}

\begin{lem}\label{lemma.bound_supg_Fkappa_6}
    On the event $\Omega(K),$ for all $c_\rho \in \{1, 2\},$ the supremum of $T_{K, \mu}(g, \chi, f^*, \sigma^*)$ over the set
    \begin{align*}
        \mF^{(c_\rho)}_6 &:= \{ (g, \chi) \in \mF \times I_{+} :
        \|g - f^*\| > c_\rho \rho_K,\ 
        \chi > \sigma^* + \alpha_{K,c_\rho}
        \},
    \end{align*}
    is bounded above by 
    \begin{align*}
        B_{6,c_\rho} 
        &
        := \max \Bigg\{ -\frac{2\sigma^*}{(2\sigma^*+\alpha_{K,c_\rho})^2} \alpha_{K,c_\rho}^2 + c_\rho \bigg( \frac{ 8(c+2)\eps }{2\sigma^*+\alpha_{K,c_\rho}} - \frac{4 c_\mu \eps}{5 \frakm^*}\bigg) r^2(\rho_K)
        + \frac{c_\mu\eps}{10\frakm^*} r^2(\rho_K), 
         \\
        &\hspace{1.5cm}
        -\frac{2\sigma^*}{(2\sigma^*+\alpha_{K,c_\rho})^2} \alpha_{K,c_\rho}^2 + c_\rho \bigg(2(c-2)\frac{4\eps - (4\theta_0)^{-2}}{\sigma^*+\sigma_{+}}  + \frac{c_\mu\eps}{\frakm^*} \bigg)r^2(\rho_K) + \frac{c_\mu\eps}{10\frakm^*}r^2(\rho_K) \Bigg\}.
    \end{align*}
\end{lem}
\begin{proof}[Proof of Lemma~\ref{lemma.bound_supg_Fkappa_6}]
    Let $(g, \chi) \in \mF^{(c_\rho)}_6.$ The space $\mF^{(c_\rho)}_6$ shares with $\mF^{(c_\rho)}_3$ the condition $\|g-f^*\|>c_\rho \rho_K,$ and with $\mF^{(c_\rho)}_5$ the condition $\chi > \sigma^* + \alpha_{K,c_\rho}.$ By arguing as in the proofs of Lemma~\ref{lemma.bound_supg_Fkappa_3} and Lemma~\ref{lemma.bound_supg_Fkappa_5}, we find
    \begin{align*}
        T_{K, \mu} (g, \chi, f^*, \sigma^*)
        &\leq -\frac{2\sigma^*}{(2\sigma^*+\alpha_{K,c_\rho})^2} \alpha_{K,c_\rho}^2 + \frac{2\Upsilon \Delta_c(\chi,\sigma^*)}{\sigma^*+\chi} Q_{3/4,K}\big[ 2\zeta(f-f^*)-(f-f^*)^2 \big]  \\
        &\quad - \mu \Upsilon \sup_{z^* \in \Gamma_{f^*}(\rho_K)} z^* (f - f^*) 
        + \frac{\mu \rho_K}{10},
    \end{align*}
    with the function $f=f^* + \rho_K(g-f^*)/\|g-f^*\|$ and the quantity $\Upsilon = \|g-f^*\|/\rho_K,$ as in the proof of Lemma~\ref{lemma.bound_supg_Fkappa_3}. By following the same argument, we split the cases $\|f - f^*\|_{2,\bX} \leq r(\rho_K)$ and $\|f - f^*\|_{2,\bX} > r(\rho_K).$
    
    We start with $\|f - f^*\|_{2,\bX} \leq r(\rho_K).$ We find,
    \begin{align*}
        - \mu \sup_{z^* \in \Gamma_{f^*}(\rho_K)} z^* (f - f^*)
        \leq - \frac{4 c_\mu \eps}{5 \frakm^*} r^2(\rho_K).
    \end{align*}
    Combining this the fact that $(f - f^*)^2 \geq 0,$ we get
    \begin{align*}
        T_{K, \mu} (g, \chi, f^*, \sigma^*) 
        &\leq -\frac{2\sigma^*}{(2\sigma^*+\alpha_{K,c_\rho})^2} \alpha_{K,c_\rho}^2 + \Upsilon \bigg( \frac{ 8(c+2)\eps }{2\sigma^*+\alpha_{K,c_\rho}} - \frac{4 c_\mu \eps}{5 \frakm^*}\bigg) r^2(\rho_K)
        + \frac{\mu \rho_K}{10} \\
        &\leq -\frac{2\sigma^*}{(2\sigma^*+\alpha_{K,c_\rho})^2} \alpha_{K,c_\rho}^2 + c_\rho \bigg( \frac{ 8(c+2)\eps }{2\sigma^*+\alpha_{K,c_\rho}} - \frac{4 c_\mu \eps}{5 \frakm^*}\bigg) r^2(\rho_K)
        + \frac{c_\mu\eps}{10\frakm^*} r^2(\rho_K),
    \end{align*}
    using that the quantity multiplied by $\Upsilon$ is negative by condition~\eqref{ass.c_mu_lowbound}, and $\Upsilon > c_\rho.$ This concludes the first part of the proof.
    
    We now consider $\|f - f^*\|_{2,\bX} > r(\rho_K).$ We have,
    \begin{align*}
        T_{K, \mu} &(g, \chi, f^*, \sigma^*) \\
        &\leq -\frac{2\sigma^*}{(2\sigma^*+\alpha_{K,c_\rho})^2} \alpha_{K,c_\rho}^2 + \Upsilon \bigg(2(c-2)\frac{4\eps - (4\theta_0)^{-2}}{\sigma^*+\sigma_{+}} r^2(\rho_K)  + \mu\rho_K \bigg) + \frac{\mu \rho_K}{10} \\
        &\leq -\frac{2\sigma^*}{(2\sigma^*+\alpha_{K,c_\rho})^2} \alpha_{K,c_\rho}^2 + c_\rho \bigg(2(c-2)\frac{4\eps - (4\theta_0)^{-2}}{\sigma^*+\sigma_{+}}  + \frac{c_\mu\eps}{\frakm^*} \bigg)r^2(\rho_K) + \frac{c_\mu\eps}{10\frakm^*}r^2(\rho_K),
    \end{align*}
    using that the quantity multiplied by $\Upsilon$ is negative by condition~\eqref{ass.1/theta_0_lowbound_negative}, and $\Upsilon > c_\rho.$ This concludes the proof.
\end{proof}

\begin{lem}\label{lemma.bound_supg_Fkappa_7}
    On the event $\Omega(K),$ for all $c_\rho \in \{1, 2\},$ the supremum of $T_{K, \mu}(g, \chi, f^*, \sigma^*)$ over the set
    \begin{align*}
        \mF^{(c_\rho)}_7 &:= \{ (g, \chi) \in \mF \times I_{+} :
        \|g - f^*\| \leq c_\rho \rho_K,
        \|g - f^*\|_{2,\bX} \leq r(c_\rho \rho_K),\ 
        \chi < \sigma^* - \alpha_{K,c_\rho}
        \},
    \end{align*}
    is bounded above by 
    \begin{align*}
        B_{7,c_\rho} 
        &
        := -\frac{2\sigma^*}{(2\sigma^*-\alpha_{K,c_\rho})^2} \alpha_{K,c_\rho}^2 + \frac{8(c+2)\eps}{\sigma^*} r^2(c_\rho \rho_K) + \frac{c_\mu\eps c_\rho}{\frakm^*} r^2(\rho_K).
    \end{align*}
\end{lem}
\begin{proof}[Proof of Lemma~\ref{lemma.bound_supg_Fkappa_7}]
    Let $(g, \chi) \in \mF^{(c_\rho)}_7.$ The space $\mF^{(c_\rho)}_7$  shares with $\mF^{(c_\rho)}_1$ the conditions $\|g-f^*\|\leq c_\rho \rho_K$ and $\|g-f^*\|_{2,\bX}\leq r(c_\rho \rho_K).$ By arguing as in the proof of Lemma~\ref{lemma.bound_supg_Fkappa_1} and using $\chi < \sigma^* - \alpha_{K,c_\rho},$ from~\eqref{eq.bound_Q_3_4_R<} we get
    \begin{align*}
        T_{K, \mu} &(g, \chi, f^*, \sigma^*) \\
        &\leq \sup_{\chi<\sigma^*-\alpha_{K,c_\rho}} (\sigma^*-\chi) \Big( 1 - \frac{4(\sigma^{*2}-\delta_{K,n})}{(\sigma^*+\chi)^2} \Big) + \frac{8(c+2)\eps}{\sigma^*} r^2(c_\rho \rho_K) + \frac{c_\mu\eps c_\rho}{\frakm^*} r^2(\rho_K) \\
        &= -\alpha_{K,c_\rho} \Big(\frac{4(\sigma^{*2}-\delta_{K,n})}{(2\sigma^*-\alpha_{K,c_\rho})^2} - 1 \Big) + \frac{8(c+2)\eps}{\sigma^*} r^2(c_\rho \rho_K) + \frac{c_\mu\eps c_\rho}{\frakm^*} r^2(\rho_K).
    \end{align*}
    Since $4\delta_{K,n}/\sigma^* < \alpha_{K,c_\rho} < \sigma^*$ by Lemma~\ref{lemma.conditions}, we find
    \begin{align*}
        \frac{4(\sigma^{*2}-\delta_{K,n})}{(2\sigma^*-\alpha_{K,c_\rho})^2} - 1 = \frac{4\sigma^*\alpha_{K,c_\rho} - 4\delta_{K,n} - \alpha_{K,c_\rho}^2}{(2\sigma^*-\alpha_{K,c_\rho})^2} > \frac{2\sigma^*\alpha_{K,c_\rho}}{(2\sigma^*-\alpha_{K,c_\rho})^2},
    \end{align*}
    and
    \begin{align*}
        T_{K, \mu} (g, \chi, f^*, \sigma^*) \leq -\frac{2\sigma^*}{(2\sigma^*-\alpha_{K,c_\rho})^2} \alpha_{K,c_\rho}^2 + \frac{8(c+2)\eps}{\sigma^*} r^2(c_\rho \rho_K) + \frac{c_\mu\eps c_\rho}{\frakm^*} r^2(\rho_K),
    \end{align*}
    which is sufficient to conclude.
\end{proof}

\begin{lem}\label{lemma.bound_supg_Fkappa_8}
    On the event $\Omega(K),$ for all $c_\rho \in \{1, 2\},$ the supremum of $T_{K, \mu}(g, \chi, f^*, \sigma^*)$ over the set
    \begin{align*}
        \mF^{(c_\rho)}_8 &:= \{ (g, \chi) \in \mF \times I_{+} :
        \|g - f^*\| \leq c_\rho \rho_K,
        \|g - f^*\|_{2,\bX} > r(c_\rho \rho_K),\ 
        \chi < \sigma^* - \alpha_{K,c_\rho}
        \},
    \end{align*}
    is bounded above by 
    \begin{align*}
        B_{8,c_\rho} 
        &
        := -\frac{2\sigma^*}{(2\sigma^*-\alpha_{K,c_\rho})^2} \alpha_{K,c_\rho}^2 + 2(c-2) \frac{4\eps - (4\theta_0)^{-2}}{2\sigma^*-\alpha_{K,c_\rho}} r^2(c_\rho \rho_K)  + \frac{c_\mu\eps c_\rho}{\frakm^*}r^2(\rho_K).
    \end{align*}
\end{lem}
\begin{proof}[Proof of Lemma~\ref{lemma.bound_supg_Fkappa_8}]
    Let $(g, \chi) \in \mF^{(c_\rho)}_8.$ The space $\mF^{(c_\rho)}_8$ shares with $\mF^{(c_\rho)}_1$ the condition $\|g-f^*\|\leq c_\rho \rho_K,$ with $\mF^{(c_\rho)}_2$ the condition $\|g-f^*\|>r(c_\rho \rho_K),$ and with $\mF^{(c_\rho)}_7$ the condition $\chi < \sigma^* - \alpha_{K,c_\rho}.$ By arguing as in the proofs of Lemma~\ref{lemma.bound_supg_Fkappa_1}, Lemma~\ref{lemma.bound_supg_Fkappa_2} and Lemma~\ref{lemma.bound_supg_Fkappa_7}, one finds
    \begin{align*}
        T_{K, \mu} (g, \chi, f^*, \sigma^*) 
        \leq -\frac{2\sigma^*}{(2\sigma^*-\alpha_{K,c_\rho})^2} \alpha_{K,c_\rho}^2 + 2(c-2) \frac{4\eps - (4\theta_0)^{-2}}{2\sigma^*-\alpha_{K,c_\rho}} r^2(c_\rho \rho_K)  + \frac{c_\mu\eps c_\rho}{\frakm^*}r^2(\rho_K),
    \end{align*}
    which concludes the proof.
\end{proof}

\begin{lem}\label{lemma.bound_supg_Fkappa_9}
    On the event $\Omega(K),$ for all $c_\rho \in \{1, 2\},$ the supremum of $T_{K, \mu}(g, \chi, f^*, \sigma^*)$ over the set
    \begin{align*}
        \mF^{(c_\rho)}_9 &:= \{ (g, \chi) \in \mF \times I_{+} :
        \|g - f^*\| > c_\rho \rho_K,\ 
        \chi < \sigma^* - \alpha_{K,c_\rho}
        \},
    \end{align*}
    is bounded above by 
    \begin{align*}
        B_{9,c_\rho} 
        &
        := \max \Bigg\{ -\frac{2\sigma^*}{(2\sigma^*-\alpha_{K,c_\rho})^2} \alpha_{K,c_\rho}^2 + \bigg( \frac{ 8(c+2)\eps c_\rho }{\sigma^*} - \frac{4 c_\mu \eps c_\rho}{5 \frakm^*} + \frac{c_\mu\eps}{10\frakm^*} \bigg) r^2(\rho_K), 
         \\
        &\hspace{1.5cm}
         -\frac{2\sigma^*}{(2\sigma^*-\alpha_{K,c_\rho})^2} \alpha_{K,c_\rho}^2 + c_\rho \bigg(2(c-2)\frac{4\eps - (4\theta_0)^{-2}}{2\sigma^*-\alpha_{K,c_\rho}}  + \frac{c_\mu\eps}{\frakm^*} \bigg)r^2(\rho_K) + \frac{c_\mu\eps}{10\frakm^*}r^2(\rho_K) \Bigg\}.
    \end{align*}
\end{lem}
\begin{proof}[Proof of Lemma~\ref{lemma.bound_supg_Fkappa_9}]
    Let $(g, \chi) \in \mF^{(c_\rho)}_9.$ The space $\mF^{(c_\rho)}_9$ shares with $\mF^{(c_\rho)}_6$ the condition $\|g-f^*\|>c_\rho \rho_K,$ and with $\mF^{(c_\rho)}_7$ the condition $\chi < \sigma^* - \alpha_{K,c_\rho}.$ By arguing as in the proofs of Lemma~\ref{lemma.bound_supg_Fkappa_6} and Lemma~\ref{lemma.bound_supg_Fkappa_7}, we get
    \begin{align*}
        T_{K, \mu} (g, \chi, f^*, \sigma^*)
        &\leq -\frac{2\sigma^*}{(2\sigma^*-\alpha_{K,c_\rho})^2} \alpha_{K,c_\rho}^2 + \frac{2\Upsilon \Delta_c(\chi,\sigma^*)}{\sigma^*+\chi} Q_{3/4,K}\big[ 2\zeta(f-f^*)-(f-f^*)^2 \big]  \\
        &\quad - \mu \Upsilon \sup_{z^* \in \Gamma_{f^*}(\rho_K)} z^* (f - f^*)
        + \frac{\mu \rho_K}{10},
    \end{align*}
    with the function $f=f^* + \rho_K(g-f^*)/\|g-f^*\|$ and the quantity $\Upsilon = \|g-f^*\|/\rho_K.$ We now split the cases $\|f - f^*\|_{2,\bX} \leq r(\rho_K)$ and $\|f - f^*\|_{2,\bX} > r(\rho_K).$
    
    For $\|f - f^*\|_{2,\bX} \leq r(\rho_K),$ we find
    \begin{align*}
        - \mu \sup_{z^* \in \Gamma_{f^*}(\rho_K)} z^* (f - f^*)
        \leq - \frac{4 c_\mu \eps}{5 \frakm^*} r^2(\rho_K),
    \end{align*}
    which we combine with the fact that $(f - f^*)^2 \geq 0,$ this gives
    \begin{align*}
        T_{K, \mu} (g, \chi, f^*, \sigma^*) 
        &\leq -\frac{2\sigma^*}{(2\sigma^*-\alpha_{K,c_\rho}^2)^2} \alpha_{K,c_\rho}^2 + \Upsilon \bigg( \frac{ 8(c+2)\eps }{\sigma^*} - \frac{4 c_\mu \eps}{5 \frakm^*}\bigg) r^2(\rho_K)
        + \frac{\mu \rho_K}{10} \\
        &\leq -\frac{2\sigma^*}{(2\sigma^*-\alpha_{K,c_\rho}^2)^2} \alpha_{K,c_\rho}^2 + c_\rho \bigg( \frac{ 8(c+2)\eps }{\sigma^*} - \frac{4 c_\mu \eps}{5 \frakm^*}\bigg) r^2(\rho_K)
        + \frac{c_\mu\eps}{10\frakm^*} r^2(\rho_K),
    \end{align*}
    using that the quantity multiplied by $\Upsilon$ is negative by condition~\eqref{ass.c_mu_lowbound}, and $\Upsilon > c_\rho.$ This concludes the first part of the proof.
    
    We now consider the case $\|f - f^*\|_{2,\bX} > r(\rho_K).$ We find
    \begin{align*}
        T_{K, \mu} &(g, \chi, f^*, \sigma^*) \\
        &\leq -\frac{2\sigma^*}{(2\sigma^*-\alpha_{K,c_\rho})^2} \alpha_{K,c_\rho}^2 + \Upsilon \bigg(2(c-2)\frac{4\eps - (4\theta_0)^{-2}}{2\sigma^*-\alpha_{K,c_\rho}} r^2(\rho_K)  + \mu\rho_K \bigg) + \frac{\mu \rho_K}{10} \\
        &\leq -\frac{2\sigma^*}{(2\sigma^*-\alpha_{K,c_\rho})^2} \alpha_{K,c_\rho}^2 + c_\rho \bigg(2(c-2)\frac{4\eps - (4\theta_0)^{-2}}{2\sigma^*-\alpha_{K,c_\rho}}  + \frac{c_\mu\eps}{\frakm^*} \bigg)r^2(\rho_K) + \frac{c_\mu\eps}{10\frakm^*}r^2(\rho_K),
    \end{align*}
    using that the quantity multiplied by $\Upsilon$ is negative by condition~\eqref{ass.1/theta_0_lowbound_negative}, and $\Upsilon > c_\rho.$ This concludes the proof.
\end{proof}

\subsection{Comparison between the bounds}

This section compares the bounds $B_{1,c_\rho},\ldots,B_{9,c_\rho}$ found above. We show that, for $c_\rho=1,$ the quantity $B_{1,1}$ dominates the bounds $B_{i,1}$ on the slices $i=2,\ldots,9.$ Furthermore, for $c_\rho=2,$ the negative quantity $-B_{1,1}$ is also bigger than any other bound $B_{i,2}$ on the slices $i=2,\ldots,9.$ This implicitly shows that the bounds $B_{i,2}$ are negative and bounded away from zero, if $i\neq1.$

\begin{lem}\label{lemma.comparison_bounds_B_i}
    We have $B_{1,1} = \max_{i=1, \dots, 9} B_{i,1}$
    and $- B_{1,1} > \max_{i=2, \dots, 9} B_{i,2}.$
\end{lem}
\begin{proof}[Proof of Lemma~\ref{lemma.comparison_bounds_B_i}]
     We start by showing that $B_{1,1}$ is bigger than the other $B_{i,1}$, $i=2, \dots, 9$. By Lemma~\ref{lemma.bound_supg_Fkappa_1}, we have
     \begin{align*}
         B_{1,1} = \frac{16}{\sigma^*(2\sigma^*-\alpha_{K,1})^2} \delta_{K,n}^2 + \frac{8(c+2)\eps }{2\sigma^*-\alpha_{K,1}} r^2(\rho_K) + \frac{c_\mu\eps}{\frakm^*}r^2(\rho_K).
     \end{align*}
     
    Take $i=2.$ By Lemma~\ref{lemma.bound_supg_Fkappa_2}, we have
    \begin{align*}
        B_{2,1} = \frac{16}{\sigma^*(2\sigma^*-\alpha_{K,1})^2} \delta_{K,n}^2 
        + 2(c-2) \frac{4\eps - (4\theta_0)^{-2}}{2\sigma^*+\alpha_{K,1}} r^2(c_\rho \rho_K)  + \frac{c_\mu\eps}{\frakm^*}r^2(\rho_K),
    \end{align*}
    so that imposing $B_{2,1} \leq B_{1,1}$ is equivalent to
    \begin{align*}
        2(c-2) \frac{4\eps - (4\theta_0)^{-2}}{2\sigma^*+\alpha_{K,1}} \leq \frac{8(c+2) \eps}{2\sigma^*-\alpha_{K,1}},
    \end{align*}
    which is always true since $4\eps - (4\theta_0)^{-2} < 0,$ by condition~\eqref{ass.1/theta_0_lowbound_negative}.
    
    Take $i=3.$ By Lemma~\ref{lemma.bound_supg_Fkappa_3}, we have
    \begin{align*}
        B_{3,1} &= \max \Bigg\{ \frac{16}{\sigma^*(2\sigma^*-\alpha_{K,1})^2} \delta_{K,n}^2 + \bigg( \frac{ 8(c+2)\eps }{2\sigma^*-\alpha_{K,1}} - \frac{4 c_\mu \eps}{5 \frakm^*}\bigg) r^2(\rho_K)
        + \frac{c_\mu\eps}{10\frakm^*} r^2(\rho_K), 
         \\
        &\hspace{1.5cm}
         \frac{16}{\sigma^*(2\sigma^*-\alpha_{K,1})^2} \delta_{K,n}^2 + \bigg(2(c-2)\frac{4\eps - (4\theta_0)^{-2}}{2\sigma^*+\alpha_{K,1}}  + \frac{c_\mu\eps}{\frakm^*} \bigg)r^2(\rho_K) + \frac{c_\mu\eps}{10\frakm^*}r^2(\rho_K) \Bigg\},
    \end{align*}
    so that imposing $B_{3,1} \leq B_{1,1}$ requires both 
    \begin{align*}
        \frac{ 8(c+2)\eps }{2\sigma^*-\alpha_{K,1}} - \frac{17 c_\mu \eps}{10 \frakm^*} &\leq \frac{8(c+2)\eps }{2\sigma^*-\alpha_{K,1}}, \\
        2(c-2)\frac{4\eps - (4\theta_0)^{-2}}{2\sigma^*+\alpha_{K,1}}  + \frac{c_\mu\eps}{10\frakm^*} &\leq \frac{8(c+2)\eps }{2\sigma^*-\alpha_{K,1}}.
    \end{align*}
    The first inequality is always true, whereas the second is equivalent to
    \begin{align*}
        \frac{8(c-2)\eps}{2\sigma^*+\alpha_{K,1}} - \frac{8(c+2)\eps }{2\sigma^*-\alpha_{K,1}} + \frac{c_\mu\eps}{10\frakm^*} &\leq 2(c-2)\frac{(4\theta_0)^{-2}}{2\sigma^*+\alpha_{K,1}}.
    \end{align*}
    Since $2\sigma^* + \alpha_{K,1} > 2\sigma^* - \alpha_{K,1},$ the latter condition is implied by
    \begin{align*}
        \frac{c_\mu\eps}{10\frakm^*} -\frac{32\eps}{2\sigma^*-\alpha_{K,1}} &\leq \frac{ c-2}{8\theta_0^2 (2\sigma^*+\alpha_{K,1})}.
    \end{align*}
    By Lemma~\ref{lemma.conditions}, we have $0<\alpha_{K,1}<\sigma^*$ and the above display is satisfied if
    \begin{align*}
        \frac{c_\mu\eps}{10\frakm^*} \leq \frac{16\eps}{\sigma^*} + \frac{ c-2}{24\theta_0^2 \sigma^*}.
    \end{align*}
    We multiply by $\sigma^*$ and use that $\kappa^{*1/4} = \frakm^*/\sigma* \geq 1,$ so it is sufficient that
    \begin{align*}
        \frac{c_\mu\eps}{10} \leq 16\eps + \frac{c-2}{24\theta_0^2},
    \end{align*}
    which holds by condition~\eqref{ass.1/theta_0_lowbound_2}.
    
    Take $i=4.$ By Lemma~\ref{lemma.bound_supg_Fkappa_4}, we have
    \begin{align*}
        B_{4,1} = -\frac{2\sigma^*}{(2\sigma^*+\alpha_{K,1})^2} \alpha_{K,1}^2 + \frac{8(c+2) \eps}{2\sigma^*+\alpha_{K,1}} r^2(\rho_K) + \frac{c_\mu\eps}{\frakm^*} r^2(\rho_K),
    \end{align*}
    so that imposing $B_{4,1} \leq B_{1,1}$ is equivalent to 
    \begin{align*}
        -\frac{2\sigma^*}{(2\sigma^*+\alpha_{K,1})^2} \frac{\alpha_{K,1}^2}{r^2(\rho_K) } + \frac{8(c+2) \eps}{2\sigma^*+\alpha_{K,1}} \leq \frac{16}{\sigma^*(2\sigma^*-\alpha_{K,1})^2} \frac{\delta_{K,n}^2}{r^2(\rho_K) } + \frac{8(c+2) \eps}{2\sigma^*-\alpha_{K,1}} ,
    \end{align*}
    which is always satisfied.
    
    Take $i=5.$ By Lemma~\ref{lemma.bound_supg_Fkappa_5}, we have
    \begin{align*}
        B_{5,1} = -\frac{2\sigma^*}{(2\sigma^*+\alpha_{K,1})^2} \alpha_{K,1}^2 + 2(c-2) \frac{4\eps - (4\theta_0)^{-2}}{\sigma^*+\sigma_{+}} r^2(\rho_K)  + \frac{c_\mu\eps}{\frakm^*}r^2(\rho_K),
    \end{align*}
    so that imposing $B_{5,1} \leq B_{1,1}$ is equivalent to 
    \begin{align*}
        -\frac{2\sigma^*}{(2\sigma^*+\alpha_{K,1})^2} \frac{\alpha_{K,1}^2}{r^2(\rho_K)} + 2(c-2) \frac{4\eps - (4\theta_0)^{-2}}{\sigma^*+\sigma_{+}} \leq \frac{16}{\sigma^*(2\sigma^*-\alpha_{K,1})^2} \frac{\delta_{K,n}^2}{r^2(\rho_K)} + \frac{8(c+2) \eps}{2\sigma^*-\alpha_{K,1}},
    \end{align*}
    which is always satisfied, since the term on the left is negative by condition~\eqref{ass.1/theta_0_lowbound_negative}.
    
    Take $i=6.$ By Lemma~\ref{lemma.bound_supg_Fkappa_6}, we have
    \begin{align*}
        B_{6,1} &= \max \Bigg\{ -\frac{2\sigma^*}{(2\sigma^*+\alpha_{K,1})^2} \alpha_{K,1}^2 + \bigg( \frac{ 8(c+2)\eps }{2\sigma^*+\alpha_{K,1}} - \frac{4 c_\mu \eps}{5 \frakm^*}\bigg) r^2(\rho_K)
        + \frac{c_\mu\eps}{10\frakm^*} r^2(\rho_K), 
         \\
        &\hspace{1.5cm}
        -\frac{2\sigma^*}{(2\sigma^*+\alpha_{K,1})^2} \alpha_{K,1}^2 + \bigg(2(c-2)\frac{4\eps - (4\theta_0)^{-2}}{\sigma^*+\sigma_{+}}  + \frac{c_\mu\eps}{\frakm^*} \bigg)r^2(\rho_K) + \frac{c_\mu\eps}{10\frakm^*}r^2(\rho_K) \Bigg\},
    \end{align*}
    so that imposing $B_{6,1} \leq B_{1,1}$ is equivalent to both
    \begin{align*}
        \frac{ 8(c+2)\eps }{2\sigma^*+\alpha_{K,1}} 
        - \frac{7 c_\mu \eps}{10 \frakm^*}
        &\leq \frac{2\sigma^*}{(2\sigma^*+\alpha_{K,1})^2} \frac{\alpha_{K,1}^2}{r^2(\rho_K) } 
        + \frac{16}{\sigma^*(2\sigma^*-\alpha_{K,1})^2} \frac{\delta_{K,n}^2}{r^2(\rho_K) } 
        + \frac{8(c+2)\eps }{2\sigma^*-\alpha_{K,1}} 
        + \frac{c_\mu\eps}{\frakm^*},
    \end{align*}
    which is always true, and 
    \begin{align*}
        2(c-2)\frac{4\eps - (4\theta_0)^{-2}}{\sigma^*+\sigma_{+}} 
        + \frac{11 c_\mu\eps}{10\frakm^*} 
        &\leq \frac{2\sigma^*}{(2\sigma^*+\alpha_{K,1})^2} \frac{\alpha_{K,1}^2}{r^2(\rho_K) }    
        + \frac{16}{\sigma^*(2\sigma^*-\alpha_{K,1})^2} \frac{\delta_{K,n}^2}{r^2(\rho_K) }  \\
        &\quad + \frac{8(c+2)\eps }{2\sigma^*-\alpha_{K,1}} 
        + \frac{c_\mu\eps}{\frakm^*}.
    \end{align*}
    The first term on the left side is negative, by condition~\eqref{ass.1/theta_0_lowbound_negative}. With the ratio $\delta_{K,n}^2/r^2(\rho_K)$ in~\eqref{def.alpha_2k_delta_Kn}, it is sufficient that
    \begin{align*}
        \frac{c_\mu\eps}{10\frakm^*} 
        &\leq \frac{2\sigma^*}{(2\sigma^*+\alpha_{K,1})^2} c_\alpha^2   
        + \frac{16}{\sigma^*(2\sigma^*-\alpha_{K,1})^2} \cdot\frac{25\frakm^{*2}\eps^2}{c_K^2\theta_1^2} + \frac{8(c+2)\eps }{2\sigma^*-\alpha_{K,1}}.
    \end{align*}
    By Lemma~\ref{lemma.conditions}, we have $0<\alpha_{K,1}<\sigma^*,$ so it is enough that
    \begin{align*}
        \frac{c_\mu\eps}{10\frakm^*} 
        &\leq \frac{2 c_\alpha^2}{9\sigma^*}    
        + \frac{400\frakm^{*2}\eps^2}{4\sigma^{*3} c_K^2\theta_1^2}  + \frac{8(c+2)\eps }{2\sigma^*}.
    \end{align*}
    We now multiply by $\frakm^*$ and use that $\kappa^{*1/4} = \frakm^*/\sigma^* \geq 1,$ this gives the sufficient condition
    \begin{align*}
        \frac{9c_\mu\eps}{20} 
        &\leq c_\alpha^2 + \frac{450\eps^2}{c_K^2\theta_1^2}  + 18(c+2)\eps,
    \end{align*}
    which follows from condition~\eqref{ass.c_alpha_lowbound}. 
    
    Take $i=7.$ By Lemma~\ref{lemma.bound_supg_Fkappa_7}, we have
    \begin{align*}
        B_{7,1} = -\frac{2\sigma^*}{(2\sigma^*-\alpha_{K,1})^2} \alpha_{K,1}^2 + \frac{8(c+2)\eps}{\sigma^*} r^2(\rho_K) + \frac{c_\mu\eps}{\frakm^*} r^2(\rho_K),
    \end{align*}
    so that imposing $B_{7,1} \leq B_{1,1}$ is equivalent to
    \begin{align*}
         \frac{8(c+2)\eps}{\sigma^*} \leq \frac{2\sigma^*}{(2\sigma^*-\alpha_{K,1})^2} \frac{\alpha_{K,1}^2}{r^2(\rho_K)}
         + \frac{16}{\sigma^*(2\sigma^*-\alpha_{K,1})^2} \frac{\delta_{K,n}^2}{r^2(\rho_K) }  
         + \frac{8(c+2)\eps }{2\sigma^*-\alpha_{K,1}}.
    \end{align*}
    We argue as for $i=6,$ we plug in the ratio $\delta_{K,n}^2/r^2(\rho_K)$ from \eqref{def.alpha_2k_delta_Kn} and use $0<\alpha_{K,1}<\sigma^*$ and $2\sigma^*-\alpha_{K,1} < 2\sigma^*+\alpha_{K,1},$ it is enough that
    \begin{align*}
         \frac{8(c+2)\eps}{\sigma^*}
        &\leq \frac{2 c_\alpha^2}{9\sigma^*}    
        + \frac{400\frakm^{*2}\eps^2}{4\sigma^{*3} c_K^2\theta_1^2}  + \frac{8(c+2)\eps }{2\sigma^*}.
    \end{align*}
    We now multiply by $\sigma^*$ and use that  $\kappa^{*1/4} = \frakm^*/\sigma^* \geq 1,$ this gives the sufficient condition
    \begin{align*}
        8(c+2)\eps \leq \frac{2 c_\alpha^2}{9} + 100\eps^2 + 4(c+2)\eps,
    \end{align*}
    which is true if $18(c+2)\eps \leq c_\alpha^2 + 450\eps^2,$ which holds thanks to condition~\eqref{ass.c_alpha_lowbound}. 
    
    Take $i=8.$ By Lemma~\ref{lemma.bound_supg_Fkappa_8}, we have
    \begin{align*}
        B_{8,1} = -\frac{2\sigma^*}{(2\sigma^*-\alpha_{K,1})^2} \alpha_{K,1}^2 + 2(c-2) \frac{4\eps - (4\theta_0)^{-2}}{2\sigma^*-\alpha_{K,1}} r^2(\rho_K)  + \frac{c_\mu\eps}{\frakm^*}r^2(\rho_K),
    \end{align*}
    so that imposing $B_{8,1} \leq B_{1,1}$ is equivalent to
    \begin{align*}
         -\frac{2\sigma^*}{(2\sigma^*-\alpha_{K,1})^2} \frac{\alpha_{K,1}^2}{r^2(\rho_K)} + 2(c-2) \frac{4\eps - (4\theta_0)^{-2}}{2\sigma^*-\alpha_{K,1}}    \leq \frac{16}{\sigma^*(2\sigma^*-\alpha_{K,1})^2} \frac{\delta_{K,n}^2}{r^2(\rho_K) } + \frac{8(c+2) \eps}{2\sigma^*-\alpha_{K,1}},
    \end{align*}
    which holds since the left side is negative, thanks to condition~\eqref{ass.1/theta_0_lowbound_negative}.
    
    Take $i=9.$ By Lemma~\ref{lemma.bound_supg_Fkappa_9}, we have
    \begin{align*}
        B_{9,1} &= \max \Bigg\{ -\frac{2\sigma^*}{(2\sigma^*-\alpha_{K,1})^2} \alpha_{K,1}^2 + \bigg( \frac{ 8(c+2)\eps }{\sigma^*} - \frac{4 c_\mu \eps}{5 \frakm^*} + \frac{c_\mu\eps}{10\frakm^*} \bigg) r^2(\rho_K), 
         \\
        &\hspace{1.5cm}
         -\frac{2\sigma^*}{(2\sigma^*-\alpha_{K,1})^2} \alpha_{K,1}^2 + \bigg(2(c-2)\frac{4\eps - (4\theta_0)^{-2}}{2\sigma^*-\alpha_{K,1}}  + \frac{c_\mu\eps}{\frakm^*} \bigg)r^2(\rho_K) + \frac{c_\mu\eps}{10\frakm^*}r^2(\rho_K) \Bigg\},
    \end{align*}
    so that imposing $B_{9,1} \leq B_{1,1}$ is equivalent to both
    \begin{align*}
         \frac{ 8(c+2)\eps }{\sigma^*} 
         - \frac{7 c_\mu \eps}{10 \frakm^*}
         &\leq \frac{2\sigma^*}{(2\sigma^*-\alpha_{K,1})^2} \frac{\alpha_{K,1}^2}{r^2(\rho_K)}
         + \frac{16}{\sigma^*(2\sigma^*-\alpha_{K,1})^2} \frac{\delta_{K,n}^2}{r^2(\rho_K)} 
         + \frac{8(c+2)\eps }{2\sigma^*-\alpha_{K,1}} 
         + \frac{c_\mu\eps}{\frakm^*},
    \end{align*}
    which is always true, and
    \begin{align*}
         2(c-2)\frac{4\eps - (4\theta_0)^{-2}}{2\sigma^*-\alpha_{K,1}}
         + \frac{c_\mu\eps}{10\frakm^*}
         &\leq \frac{2\sigma^*}{(2\sigma^*-\alpha_{K,1})^2} \frac{\alpha_{K,1}^2}{r^2(\rho_K)}
         + \frac{16}{\sigma^*(2\sigma^*-\alpha_{K,1})^2} \frac{\delta_{K,n}^2}{r^2(\rho_K)} 
         + \frac{8(c+2)\eps }{2\sigma^*-\alpha_{K,1}}.
    \end{align*}
    Arguing as in $i=6,$ the first term on the left side is negative by condition~\eqref{ass.1/theta_0_lowbound_negative}, then it is sufficient that
    \begin{align*}
         \frac{c_\mu\eps}{10\frakm^*}
         &\leq \frac{2\sigma^*}{(2\sigma^*-\alpha_{K,1})^2} \frac{\alpha_{K,1}^2}{r^2(\rho_K)}
         + \frac{16}{\sigma^*(2\sigma^*-\alpha_{K,1})^2} \frac{\delta_{K,n}^2}{r^2(\rho_K)} 
         + \frac{8(c+2)\eps }{2\sigma^*-\alpha_{K,1}},
    \end{align*}
    which coincides with the bound obtained in $i=6.$ 
    
    The first part of the proof is complete. We now show that $- B_{1,1}$ is bigger than $B_{i,2},$ for all $i=2, \dots, 9.$ We recall that Lemma~\ref{lemma.bound_supg_Fkappa_1} gives
     \begin{align*}
         B_{1,1} = \frac{16}{\sigma^*(2\sigma^*-\alpha_{K,1})^2} \delta_{K,n}^2 + \frac{8(c+2)\eps }{2\sigma^*-\alpha_{K,1}} r^2(\rho_K) + \frac{c_\mu\eps}{\frakm^*}r^2(\rho_K).
     \end{align*}
    
    Take $i=2.$ By Lemma~\ref{lemma.bound_supg_Fkappa_2}, we have
    \begin{align*}
        B_{2,2} = \frac{16}{\sigma^*(2\sigma^*-\alpha_{K,2})^2} \delta_{K,n}^2 
        + 2(c-2) \frac{4\eps - (4\theta_0)^{-2}}{2\sigma^*+\alpha_{K,2}} r^2(2\rho_K)  + \frac{2 c_\mu\eps}{\frakm^*}r^2(\rho_K),
    \end{align*}
    so that imposing $B_{2,2} + B_{1,1} < 0$ gives
    \begin{align*}
        &\frac{16}{\sigma^*(2\sigma^*-\alpha_{K,2})^2} \delta_{K,n}^2
        +  \frac{8(c-2) \eps }{2\sigma^*+\alpha_{K,2}} r^2(2\rho_K)  + \frac{2c_\mu\eps}{\frakm^*}r^2(\rho_K) \\
        &\quad + \frac{16}{\sigma^*(2\sigma^*-\alpha_{K,1})^2} \delta_{K,n}^2 + \frac{8(c+2) \eps}{2\sigma^*-\alpha_{K,1}} r^2(\rho_K) + \frac{c_\mu\eps}{\frakm^*}r^2(\rho_K) < 2(c-2) \frac{ (4\theta_0)^{-2}}{2\sigma^*+\alpha_{K,2}} r^2(2\rho_K),
    \end{align*}
    Since $r^2(2\rho_K) \geq r^2(\rho_K),$ $\alpha_{K,2} \geq \alpha_{K,1},$ it is sufficient to show
    \begin{align*}
        \frac{32}{\sigma^*(2\sigma^*-\alpha_{K,2})^2} \frac{\delta_{K,n}^2}{r^2(2\rho_K)} + \frac{8(c-2) \eps }{2\sigma^*-\alpha_{K,2}} + \frac{8(c+2) \eps}{2\sigma^*-\alpha_{K,2}} + \frac{3c_\mu\eps}{\frakm^*} < \frac{c-2}{8\theta_0^2(2\sigma^*+\alpha_{K,2})}.
    \end{align*}
    By Lemma~\ref{lemma.conditions}, we have $0 < \alpha_{K,2} < \sigma^*$ and, with the ratio $\delta_{K,n^2}/r^2(\rho_K)$ in~\eqref{def.alpha_2k_delta_Kn}, it is enough that
    \begin{align*}
        \frac{800\frakm^{*2}\eps^2}{\sigma^{*2}c_K^2 \theta_1^2} + 8(c-2) \eps + 8(c+2) \eps + \frac{3c_\mu\eps\sigma^*}{\frakm^*} < \frac{c-2}{24\theta_0^2}.
    \end{align*}
    Since $\kappa^{*1/4} = \frakm^*/\sigma^* \geq 1,$ we find the sufficient condition
    \begin{align*}
        \frac{800\kappa^{*1/2}\eps^2}{c_K^2 \theta_1^2} + 16 (c+2) \eps + 3c_\mu\eps < \frac{c-2}{24\theta_0^2},
    \end{align*}
    which is true by condition~\eqref{ass.1/theta_0_lowbound_2}.
    
    Take $i=3.$ By Lemma~\ref{lemma.bound_supg_Fkappa_3}, we have
    \begin{align*}
        B_{3,2} &= \max \Bigg\{ \frac{16}{\sigma^*(2\sigma^*-\alpha_{K,2})^2} \delta_{K,n}^2 + 2 \bigg( \frac{ 8(c+2)\eps }{2\sigma^*-\alpha_{K,2}} - \frac{4 c_\mu \eps}{5 \frakm^*}\bigg) r^2(\rho_K)
        + \frac{c_\mu\eps}{10\frakm^*} r^2(\rho_K), 
         \\
        &\hspace{1.5cm}
         \frac{16}{\sigma^*(2\sigma^*-\alpha_{K,2})^2} \delta_{K,n}^2 + 2 \bigg(2(c-2)\frac{4\eps - (4\theta_0)^{-2}}{2\sigma^*+\alpha_{K,2}}  + \frac{c_\mu\eps}{\frakm^*} \bigg)r^2(\rho_K) + \frac{c_\mu\eps}{10\frakm^*}r^2(\rho_K) \Bigg\},
    \end{align*}
    so that imposing $B_{3,2} + B_{1,1} < 0$ requires both
    \begin{align*}
        \frac{32}{\sigma^*(2\sigma^*-\alpha_{K,2})^2} \frac{\delta_{K,n}^2}{r^2(2\rho_K)} + \frac{ 16(c+2)\eps }{2\sigma^*-\alpha_{K,2}}
        + \frac{8(c+2)\eps }{2\sigma^*-\alpha_{K,1}}
        &< \frac{ c_\mu \eps}{2 \frakm^*}, \\
        \frac{32}{\sigma^*(2\sigma^*-\alpha_{K,2})^2} \frac{\delta_{K,n}^2}{r^2(2\rho_K)} 
        + \frac{16(c-2)\eps}{2\sigma^*+\alpha_{K,2}}
        + \frac{8(c+2)\eps }{2\sigma^*-\alpha_{K,1}} 
        + \frac{31 c_\mu\eps}{10 \frakm^*}
        &< \frac{ c-2}{4\theta_0^2(2\sigma^*+\alpha_{K,2})}.
    \end{align*}
    
    By arguing as for $i=2,$ it is sufficient that both
    \begin{align*}
        \frac{800\kappa^{*1/2}\eps^2}{c_K^2\theta_1^2} + 16(c+2)\eps
        + 8(c+2)\eps
        &< \frac{ c_\mu \eps}{2\kappa^{*1/4}}, \\
        \frac{800\kappa^{*1/2}\eps^2}{c_K^2\theta_1^2}
        + 8(c-2)\eps
        + 8(c+2)\eps 
        + \frac{31 c_\mu\eps}{10 \kappa^{*1/4}}
        &< \frac{ c-2}{12\theta_0^2}.
    \end{align*}
    The first bound holds by condition~\eqref{ass.c_mu_lowbound}, so we plug it into the second line using $\kappa^*\geq1,$ we obtain the sufficient condition $36 c_\mu\eps/10 < (c-2) / (12\theta_0^2),$ which follows from condition~\eqref{ass.1/theta_0_lowbound_2}.
    
    Take $i=4.$ By Lemma~\ref{lemma.bound_supg_Fkappa_4}, we have
    \begin{align*}
        B_{4,2} = -\frac{2\sigma^*}{(2\sigma^*+\alpha_{K,2})^2} \alpha_{K,2}^2 + \frac{8(c+2) \eps}{2\sigma^*+\alpha_{K,2}} r^2(2 \rho_K) + \frac{2 c_\mu\eps}{\frakm^*} r^2(\rho_K),
    \end{align*}
    so that imposing $B_{4,2} + B_{1,1} < 0$ gives
    \begin{align*}
        \frac{16}{\sigma^*(2\sigma^*-\alpha_{K,1})^2} \frac{\delta_{K,n}^2}{r^2(2 \rho_K)} + \frac{8(c+2) \eps}{2\sigma^*+\alpha_{K,2}} + \frac{8(c+2) \eps}{2\sigma^*-\alpha_{K,1}} + \frac{3c_\mu\eps}{\frakm^*} < \frac{2\sigma^*}{(2\sigma^*+\alpha_{K,2})^2} \frac{\alpha_{K,2}^2}{r^2(2 \rho_K)}.
    \end{align*}
    By arguing as for $i=3,$ it is sufficient that
    \begin{align*}
        \frac{400\kappa^{*1/2}\eps^2}{c_K^2\theta_1^2} + 4(c+2) \eps + 8(c+2) \eps + \frac{3c_\mu\eps}{\kappa^{*1/4}} < \frac{2c_\alpha^2}{9}.
    \end{align*}
    With $\kappa^*\geq1,$ it is enough that
    \begin{align*}
        \frac{1800\kappa^{*1/2}\eps^2}{c_K^2\theta_1^2} + 54(c+2) \eps + \frac{27c_\mu\eps}{2} < c_\alpha^2,
    \end{align*}
    which follows from condition~\eqref{ass.c_alpha_lowbound}.
    
    Take $i=5.$ By Lemma~\ref{lemma.bound_supg_Fkappa_5}, we have
    \begin{align*}
        B_{5,2} = -\frac{2\sigma^*}{(2\sigma^*+\alpha_{K,2})^2} \alpha_{K,2}^2 + 2(c-2) \frac{4\eps - (4\theta_0)^{-2}}{\sigma^*+\sigma_{+}} r^2(2\rho_K)  + \frac{2 c_\mu\eps}{\frakm^*}r^2(\rho_K),
    \end{align*}
    so that imposing $B_{5,2} + B_{1,1} < 0$ gives
    \begin{align*}
        \frac{16}{\sigma^*(2\sigma^*-\alpha_{K,1})^2} \frac{\delta_{K,n}^2}{r^2(2\rho_K)}
        + 2(c-2) \frac{4\eps - (4\theta_0)^{-2}}{\sigma^*+\sigma_{+}}  
        + \frac{3c_\mu\eps}{\frakm^*} + \frac{8(c+2) \eps}{2\sigma^*-\alpha_{K,1}} 
        < \frac{2\sigma^*}{(2\sigma^*+\alpha_{K,2})^2} \frac{\alpha_{K,2}^2}{r^2(2\rho_K)}.
    \end{align*}
    The second term in the latter display is negative by condition~\eqref{ass.1/theta_0_lowbound_negative}. By arguing as for $i=4,$ it is sufficient that
    \begin{align*}
         \frac{400\kappa^{*1/2}\eps^2}{c_K^2\theta_1^2}
        + \frac{3c_\mu\eps}{\kappa^{*1/4}} + 8(c+2) \eps
        < \frac{2c_\alpha^2}{9}.
    \end{align*}
    With $\kappa^*\geq1,$ it is enough that
    \begin{align*}
         \frac{1800\kappa^{*1/2}\eps^2}{c_K^2\theta_1^2}
        + \frac{27c_\mu\eps}{2} + 36(c+2) \eps
        < c_\alpha^2,
    \end{align*}
    which is true thanks to condition~\eqref{ass.c_alpha_lowbound}.
    
    Take $i=6.$ By Lemma~\ref{lemma.bound_supg_Fkappa_6}, we have
    \begin{align*}
        B_{6,2} &= \max \Bigg\{ -\frac{2\sigma^*}{(2\sigma^*+\alpha_{K,2})^2} \alpha_{K,2}^2 + 2 \bigg( \frac{ 8(c+2)\eps }{2\sigma^*+\alpha_{K,2}} - \frac{4 c_\mu \eps}{5 \frakm^*}\bigg) r^2(\rho_K)
        + \frac{c_\mu\eps}{10\frakm^*} r^2(\rho_K), 
         \\
        &\hspace{1.5cm}
        -\frac{2\sigma^*}{(2\sigma^*+\alpha_{K,2})^2} \alpha_{K,2}^2 + 2 \bigg(2(c-2)\frac{4\eps - (4\theta_0)^{-2}}{\sigma^*+\sigma_{+}}  + \frac{c_\mu\eps}{\frakm^*} \bigg)r^2(\rho_K) + \frac{c_\mu\eps}{10\frakm^*}r^2(\rho_K) \Bigg\},
    \end{align*}
    so that imposing $B_{6,2} + B_{1,1} < 0$ requires both
    \begin{align*}
        \frac{16}{\sigma^*(2\sigma^*-\alpha_{K,1})^2} \frac{\delta_{K,n}^2}{r^2(\rho_K)} 
        + \frac{ 16(c+2)\eps }{2\sigma^*+\alpha_{K,2}} 
        + \frac{8(c+2)\eps }{2\sigma^*-\alpha_{K,1}}
        - \frac{ c_\mu \eps}{2 \frakm^*} 
        &< \frac{2\sigma^*}{(2\sigma^*+\alpha_{K,2})^2} \frac{\alpha_{K,2}^2}{r^2(\rho_K)} ,\\
        \frac{16}{\sigma^*(2\sigma^*-\alpha_{K,1})^2} \frac{\delta_{K,n}^2}{r^2(\rho_K)} 
        + \frac{8(c+2)\eps }{2\sigma^*-\alpha_{K,1}}
        + \frac{31c_\mu\eps}{10\frakm^*}
        + 4(c-2)\frac{4\eps - (4\theta_0)^{-2}}{\sigma^*+\sigma_{+}}
        &< \frac{2\sigma^*}{(2\sigma^*+\alpha_{K,2})^2} \frac{\alpha_{K,2}^2}{r^2(\rho_K)}.
    \end{align*}
    By condition~\eqref{ass.1/theta_0_lowbound_negative}, the last terms on the left side of both equations are negative. By arguing as in $i=5,$ we find the sufficient conditions
    \begin{align*}
        \frac{400\kappa^{*1/2}\eps^2}{c_K^2\theta_1^2}
        + 24(c+2)\eps
        &< \frac{2c_\alpha^2}{9},\\
        \frac{400\kappa^{*1/2}\eps^2}{c_K^2\theta_1^2}
        + 8(c+2)\eps
        + \frac{31c_\mu\eps}{10\kappa^{*1/4}}
        &< \frac{2c_\alpha^2}{9}.
    \end{align*}
    With $\kappa^*\geq1,$ it is enough that
    \begin{align*}
        \frac{1800\kappa^{*1/2}\eps^2}{c_K^2\theta_1^2}
        + 108(c+2)\eps
        &< c_\alpha^2,\\
        \frac{1800\kappa^{*1/2}\eps^2}{c_K^2\theta_1^2}
        + 36(c+2)\eps
        + 14c_\mu\eps
        &< c_\alpha^2,
    \end{align*}
    which follow from condition~\eqref{ass.c_alpha_lowbound}. 
    
    Take $i=7.$ By Lemma~\ref{lemma.bound_supg_Fkappa_7}, we have
    \begin{align*}
        B_{7,2} = -\frac{2\sigma^*}{(2\sigma^*-\alpha_{K,2})^2} \alpha_{K,2}^2 + \frac{8(c+2)\eps}{\sigma^*} r^2(2 \rho_K) + \frac{2 c_\mu\eps}{\frakm^*} r^2(\rho_K),
    \end{align*}
    so that imposing $B_{7,2} + B_{1,1} < 0$ gives
    \begin{align*}
        \frac{16}{\sigma^*(2\sigma^*-\alpha_{K,1})^2} \frac{\delta_{K,n}^2}{r^2(2 \rho_K)}
        + \frac{8(c+2)\eps}{\sigma^*} 
        + \frac{8(c+2)\eps }{2\sigma^*-\alpha_{K,1}} 
        + \frac{3c_\mu\eps}{\frakm^*}
        < \frac{2\sigma^*}{(2\sigma^*-\alpha_{K,2})^2} \frac{\alpha_{K,2}^2}{r^2(2 \rho_K)}.
    \end{align*}
    By arguing as in $i=6,$ it is sufficient that
    \begin{align*}
        \frac{400\kappa^{*1/2}\eps^2}{c_K^2\theta_1^2}
        + 8(c+2)\eps
        + 8(c+2)\eps
        + \frac{3c_\mu\eps}{\kappa^{*1/4}}
        < \frac{2c_\alpha^2}{9}.
    \end{align*}
    With $\kappa^*\geq 1,$ it is enough that 
    \begin{align*}
        \frac{1800\kappa^{*1/2}\eps^2}{c_K^2\theta_1^2}
        + 72(c+2)\eps
        + \frac{27c_\mu\eps}{2}
        < c_\alpha^2,
    \end{align*}
    which follows from condition~\eqref{ass.c_alpha_lowbound}. 
    
    Take $i=8.$ By Lemma~\ref{lemma.bound_supg_Fkappa_8}, we have
    \begin{align*}
        B_{8,2} = -\frac{2\sigma^*}{(2\sigma^*-\alpha_{K,2})^2} \alpha_{K,2}^2 + 2(c-2) \frac{4\eps - (4\theta_0)^{-2}}{2\sigma^*-\alpha_{K,2}} r^2(2\rho_K)  + \frac{2 c_\mu\eps}{\frakm^*}r^2(\rho_K),
    \end{align*}
    so that imposing $B_{8,2} + B_{1,1} < 0$ gives
    \begin{align*}
        \frac{16}{\sigma^*(2\sigma^*-\alpha_{K,1})^2} \frac{\delta_{K,n}^2}{r^2(2 \rho_K)}
        + \frac{8(c+2)\eps }{2\sigma^*-\alpha_{K,1}} 
        + \frac{3c_\mu\eps}{\frakm^*} 
        + 2(c-2) \frac{4\eps - (4\theta_0)^{-2}}{2\sigma^*-\alpha_{K,2}}
        < \frac{2\sigma^*}{(2\sigma^*-\alpha_{K,2})^2} \frac{\alpha_{K,2}^2}{r^2(2\rho_K)}.
    \end{align*}
    By condition~\eqref{ass.1/theta_0_lowbound_negative}, the last term on the left side is negative. By arguing as in $i=7,$ it is sufficient that
    \begin{align*}
        \frac{400\kappa^{*1/2}\eps^2}{c_K^2\theta_1^2}
        + 8(c+2)\eps
        + \frac{3c_\mu\eps}{\kappa^{*1/4}} 
        < \frac{2c_\alpha^2}{9}.
    \end{align*}
    With $\kappa^*\geq 1,$ it is enough that
    \begin{align*}
        \frac{1800\kappa^{*1/2}\eps^2}{c_K^2\theta_1^2}
        + 36(c+2)\eps
        + \frac{27c_\mu\eps}{2} 
        < c_\alpha^2,
    \end{align*}
    which holds thanks to condition~\eqref{ass.c_alpha_lowbound}.
    
    Take $i=9.$ By Lemma~\ref{lemma.bound_supg_Fkappa_9}, we have
    \begin{align*}
        B_{9,2} &= \max \Bigg\{ -\frac{2\sigma^*}{(2\sigma^*-\alpha_{K,2})^2} \alpha_{K,2}^2 + \bigg( \frac{ 16(c+2)\eps }{\sigma^*} - \frac{8 c_\mu \eps}{5 \frakm^*} + \frac{c_\mu\eps}{10\frakm^*} \bigg) r^2(\rho_K), 
         \\
        &\hspace{1.5cm}
         -\frac{2\sigma^*}{(2\sigma^*-\alpha_{K,2})^2} \alpha_{K,2}^2 + 2 \bigg(2(c-2)\frac{4\eps - (4\theta_0)^{-2}}{2\sigma^*-\alpha_{K,2}}  + \frac{c_\mu\eps}{\frakm^*} \bigg)r^2(\rho_K) + \frac{c_\mu\eps}{10\frakm^*}r^2(\rho_K) \Bigg\},
    \end{align*}
    so that imposing $B_{9,2} + B_{1,1} < 0$ gives both
    \begin{align*}
        \frac{16}{\sigma^*(2\sigma^*-\alpha_{K,1})^2} \frac{\delta_{K,n}^2}{r^2(\rho_K)}
        + \frac{ 16(c+2)\eps}{\sigma^*} 
        + \frac{8(c+2)\eps }{2\sigma^*-\alpha_{K,1}} 
        - \frac{5 c_\mu \eps}{10 \frakm^*} 
        &< \frac{2\sigma^*}{(2\sigma^*-\alpha_{K,2})^2} \frac{\alpha_{K,2}^2}{r^2(\rho_K)},  \\
        \frac{16}{\sigma^*(2\sigma^*-\alpha_{K,1})^2} \frac{\delta_{K,n}^2}{r^2(\rho_K)} 
        + \frac{8(c+2)\eps }{2\sigma^*-\alpha_{K,1}} 
        + \frac{32c_\mu\eps}{10\frakm^*}
        + 4(c-2)\frac{4\eps - (4\theta_0)^{-2}}{2\sigma^*-\alpha_{K,2}}
        &< \frac{2\sigma^*}{(2\sigma^*-\alpha_{K,2})^2} \frac{\alpha_{K,2}^2}{r^2(\rho_K)}.
    \end{align*}
    By condition~\eqref{ass.1/theta_0_lowbound_negative}, the last terms on the left side in the latter display are negative. By arguing as in $i=8,$ it is sufficient that
    \begin{align*}
        \frac{400\kappa^{*1/2}\eps^2}{c_K^2\theta_1^2}
        + 16(c+2)\eps
        + 8(c+2)\eps
        &< \frac{2c_\alpha^2}{9},  \\
        \frac{400\kappa^{*1/2}\eps^2}{c_K^2\theta_1^2}
        + 8(c+2)\eps
        + \frac{32c_\mu\eps}{10\kappa^{*1/4}}
        &< \frac{2c_\alpha^2}{9}.
    \end{align*}
    With $\kappa^*\geq 1,$ it is enough that
    \begin{align*}
        \frac{1800\kappa^{*1/2}\eps^2}{c_K^2\theta_1^2}
        + 108(c+2)\eps
        &< c_\alpha^2,  \\
        \frac{1800\kappa^{*1/2}\eps^2}{c_K^2\theta_1^2}
        + 36(c+2)\eps
        + \frac{144c_\mu\eps}{10}
        &< c_\alpha^2,
    \end{align*}
    which both follow from condition~\eqref{ass.c_alpha_lowbound}.
\end{proof}

\subsection{Contraction rates and risk bound}

In this section we obtain convergence rates and risk bounds by exploiting the results of the previous section. We recall that we are using a function $r(\cdot)$ such that $r(\rho) \geq \max\{r_P(\rho,\gamma_P), r_M(\rho,\gamma_M)\}.$ By Assumption~\ref{ass.r2rho_rrho}, there exists an absolute constant $c_r$ such that $r(\rho) \leq r(2\rho) < c_r r(\rho).$ With $C^2 = 384 \theta_1^2 c_r^2 c_\alpha^2 \kappa_{+}^{1/2},$ we allow for $K \in \left[K^* \vee 32|\mO|,\ n\eps^2/C^2 \right]$. We denote by $\Omega(K)$ the intersection of the event $\Omega_1(K)$ in Lemma~\ref{lemma.Q1/8_zeta2}, the event $\Omega_2(K)$ in Lemma~\ref{lemma.useful_bounds_lecue} and the event $\Omega_3(K)$ in Lemma~\ref{lemma.l2_quantile}. The probability of $\Omega(K) = \Omega_1(K) \cap \Omega_2(K) \cap \Omega_3(K)$ is at least $1 - \P(\Omega_1(K)) - \P(\Omega_2(K)) - \P(\Omega_3(K)) \geq 1 - 4 \exp(-K/8920).$ 

\begin{lem} \label{lemma.conv_rates}
    On the event $\Omega(K)$ defined above, the MOM$-K$ estimator $(\widehat f_{K,\mu,\sigma_+}, \widehat\sigma_{K,\mu,\sigma_+})$ belongs to the slice
    \begin{align*}
        \mF^{(2)}_1 &:= \{ (g, \chi) \in \mF \times I_{+} :
    \|g - f^*\| \leq 2 \rho_K,
    \|g - f^*\|_{2,\bX} \leq r(2 \rho_K),\ 
    |\sigma^* - \chi| \leq c_\alpha r(2 \rho_K)
    \},
    \end{align*}
    thus recovering the convergence rates in~\eqref{eq.bound_thm}.
\end{lem}
\begin{proof}[Proof of Lemma \ref{lemma.conv_rates}]
    By definition~\eqref{def.C_functional}, we have
    \begin{align*}
        \mC_{K,\mu} (\widehat f_{K,\mu}, \widehat\sigma_{K,\mu})
        &\leq \mC_{K,\mu} (f^*, \sigma^*)
        = \sup_{g \in \mF,\ \chi < \sigma_{+}} T_{K, \mu}(g, \chi, f^*, \sigma^*) \leq B_{1,1},
    \end{align*}
    where the last inequality follows from Lemma~\ref{lemma.comparison_bounds_B_i}. Then,
    \begin{align*}
        B_{1,1}
        \geq \mC_{K,\mu} (\widehat f_{K,\mu}, \widehat\sigma_{K,\mu})
        &= \sup_{g \in \mF, \, \chi < \sigma_{+}} T_{K, \mu}(g, \chi, \widehat f_{K,\mu}, \widehat\sigma_{K,\mu}) \\
        &\geq T_{K, \mu}(f^*, \sigma^*, \widehat f_{K,\mu}, \widehat\sigma_{K,\mu})
        \geq - T_{K, \mu}(\widehat f_{K,\mu}, \widehat\sigma_{K,\mu}, f^*, \sigma^*),
    \end{align*}
    in the last step we have used $Q_{1/2}[\bx] \geq -Q_{1/2}[-\bx]$ from Lemma~\ref{lemma.quantile_prop}. We deduce that, on the event $\Omega(K),$ $T_{K, \mu}(\widehat f_{K,\mu}, \widehat\sigma_{K,\mu}, f^*, \sigma^*) \geq - B_{1,1}.$ Applying Lemma~\ref{lemma.comparison_bounds_B_i} again, we have $- B_{1,1} > \sup_{i=2, \dots 9} B_{i,2}$ and
    \begin{align*}
        \max_{i=2,\ldots,9} \sup_{(g,\chi)\in\mF_i^{(2)}} T_{K,\mu}(g,\chi,f^*,\sigma^*) \leq \max_{i=2,\ldots,9} B_{i,2} < -B_{1,1}.
    \end{align*}
    Thus, the estimator $(\widehat f_{K,\mu,\sigma_+}, \widehat\sigma_{K,\mu,\sigma_+})$ is outside $\cup_{i=2}^9 \mF_i^{(2)},$ which means that $(\widehat f_{K,\mu,\sigma_+}, \widehat\sigma_{K,\mu,\sigma_+})$ belongs to $\mF_1^{(2)}.$ By definition of $\mF_1^{(2)},$ we have
    $\|\widehat f_{K,\mu,\sigma_+} - f^*\| \leq 2 \rho_K,$ $\|\widehat f_{K,\mu,\sigma_+} - f^*\|_{2,\bX} \leq r(2 \rho_K),$ and $|\widehat\sigma_{K,\mu,\sigma_+} - \sigma^*| \leq \alpha_{K, 2} = c_\alpha r(2\rho_K).$ The proof is complete.
\end{proof}

\begin{lem} \label{lemma.excess_risk}
    On the event $\Omega(K)$ defined above, the MOM$-K$ estimator $(\widehat f_{K,\mu,\sigma_+}, \widehat\sigma_{K,\mu,\sigma_+})$ satisfies
    \begin{align*}
        R(\widehat f_{K,\mu,\sigma_+}) - R(f^*) &\leq \left(2 + 2 c_\alpha +  \left(44 + 5 c_\mu \right)\eps + \frac{25 \kappa^{*1/2}}{8 \theta_1^2}\eps^2\right) r^2(2\rho_K) \\
        &\quad + 4\theta_1^2\eps \left(r^2(2\rho_K) \vee r_Q^2(2\rho_K,\gamma_Q)\right),
    \end{align*}
    thus recovering the excess risk bound in~\eqref{eq.bound_excess_risk}.
\end{lem}
\begin{proof}[Proof of Lemma \ref{lemma.excess_risk}]
    We apply Lemma~\ref{lemma.bound_P2zeta_f_fstar} with $\rho=2\rho_K$ and $\alpha_{K,c_\rho} = \alpha_{K,2},$ which gives
    \begin{align*}
        R(\widehat f_{K,\mu}) &- R(f^*)
        = \|\widehat f_{K,\mu} - f^*\|_{2,\bX}^2 + \E[-2 \zeta (\widehat f_{K,\mu} - f^*)(\bX)] \\
        &\leq r^2(2 \rho_K) 
        + \frac{2\sigma^*+\alpha_{K,2}}{2c} T_{K,\mu}(f^*,\sigma^*,\widehat f_{K,\mu},\widehat\sigma_{K,\mu}) 
        + \frac{2\sigma^*+\alpha_{K,2}}{c} \mu\rho_K + \alpha_M^2 \\
        &\quad + \frac{8(2\sigma^*+\alpha_{K,2})}{c\sigma^*(2\sigma^*-\alpha_{K,2})^2} \delta_{K,n}^2 
        + \frac{\alpha_{K,2}}{c(2\sigma^*-\alpha_{K,2})} \left( 2 \sigma^* r(2\rho_K) + r^2(2\rho_K) + \alpha_Q^2 + \alpha_M^2 \right).
    \end{align*}
    In the proof of Lemma~\ref{lemma.conv_rates} we have shown that $T_{K, \lambda}(f^*, \sigma^*, \widehat f_{K,\mu}, \widehat\sigma_{K,\mu}) \leq \mC_{K, \lambda}(\widehat f_{K,\mu}, \widehat\sigma_{K,\mu})\leq B_{1,1}.$ By Lemma~\ref{lemma.bound_supg_Fkappa_1} and the ratio $\delta_{K,n}^2/r^2(2\rho_K)$ in \eqref{def.alpha_2k_delta_Kn}, we have
    \begin{align*}
        B_{1,1} &= \frac{16}{\sigma^*(2\sigma^*-\alpha_{K,1})^2} \delta_{K,n}^2 + \frac{8(c+2)\eps }{2\sigma^*-\alpha_{K,1}} r^2(\rho_K) + \frac{c_\mu\eps}{\frakm^*}r^2(\rho_K) \\
        &= \left( \frac{25\frakm^{*2} \eps^2}{24 \theta_1^2 \sigma^*(2\sigma^*-\alpha_{K,1})^2} + \frac{8(c+2)\eps }{2\sigma^*-\alpha_{K,1}} + \frac{c_\mu\eps}{\frakm^*} \right) r^2(\rho_K) \\
        &\leq \left( \frac{25 \kappa^{*1/2} \eps^2}{24 \theta_1^2 \sigma^*} + \frac{8(c+2)\eps }{\sigma^*} + \frac{c_\mu\eps}{\sigma^*} \right) r^2(\rho_K),
    \end{align*}
    in the last inequality we have used $\frakm^* > \sigma^*$ and $\alpha_{K,1} < \sigma^*,$ which holds by Lemma \ref{lemma.conditions}. This gives
    \begin{align*}
        \frac{2\sigma^*+\alpha_{K,2}}{2c} B_{1,1} &\leq \frac{3\sigma^*}{2c} \left( \frac{25 \kappa^{*1/2} \eps^2}{24 \theta_1^2 \sigma^*} + \frac{8(c+2)\eps }{\sigma^*} + \frac{c_\mu\eps}{\sigma^*} \right) r^2(2\rho_K) \\
        &= \left( \frac{25 \kappa^{*1/2} \eps^2}{16 \theta_1^2 c} + \frac{12(c+2)\eps}{c} + \frac{3c_\mu\eps}{2c} \right) r^2(2\rho_K).
    \end{align*}
    
    By construction, we have $\mu = (c_\mu\eps/\frakm^*)r^2(\rho_K)/\rho_K,$ so that
    \begin{align*}
        \frac{2\sigma^*+\alpha_{K,2}}{c} \mu\rho_K \leq \frac{3\sigma^*}{c} \frac{c_\mu\eps}{\frakm^*} r^2(\rho_K) \leq \frac{3 c_\mu\eps}{c} r^2(\rho_K),
    \end{align*}
    since $\frakm^* > \sigma^*$ and $\alpha_{K,2} < \sigma^*.$
    
    By Lemma~\ref{lemma.useful_bounds_lecue} we have $\alpha_M^2 \leq 4\eps r^2(2\rho_K),$ whereas by Lemma~\ref{lemma.l2_quantile} we bound
    \begin{align*}
        \alpha_Q^2 &\leq \eps \max\bigg(\|f-f^*\|_{2,\bX}^2 \frac{1488 \theta_1^4}{\eps^2}\frac{K}{n},\ r_Q^2(\rho,\gamma_Q),\ \|f-f^*\|_{2,\bX}^2\bigg) \\
        &\leq \eps \left(r^2(2\rho_K) \vee r_Q^2(2\rho_K,\gamma_Q)\right) \max\bigg(\frac{1488 \theta_1^4 K}{n\eps^2},\ 1 \bigg) \leq 4\theta_1^2\eps \left(r^2(2\rho_K) \vee r_Q^2(2\rho_K,\gamma_Q)\right)
    \end{align*}
    since $K \leq n\eps^2/C^2,$ $C^2 = 384 \theta_1^2 c_r^2 c_\alpha^2 k_{+}^{1/2}$ and $1488/384 < 4.$
    
    With $\alpha_{K,2} < \sigma^*$ and the ratio $\delta_{K,n}^2/r^2(\rho_K)$ in \eqref{def.alpha_2k_delta_Kn}, we find
    \begin{align*}
        \frac{8(2\sigma^*+\alpha_{K,2})}{c\sigma^*(2\sigma^*-\alpha_{K,2})^2} \delta_{K,n}^2 \leq \frac{24}{c \sigma^{*2}} \delta_{K,n}^2 \leq \frac{25 \kappa^{*1/2} \eps^2}{16 \theta_1^2 c} r^2(\rho_K).
    \end{align*}
    
    By putting together all the previous bounds we have
    \begin{align*}
        R(\widehat f_{K,\mu,\sigma_+}) - R(f^*) &\leq r^2(2 \rho_K) 
        + \left( \frac{25 \kappa^{*1/2} \eps^2}{16 \theta_1^2 c} + \frac{12(c+2)\eps}{c} + \frac{3c_\mu\eps}{2c} + \frac{3 c_\mu\eps}{c} + 4\eps \right) r^2(2\rho_K) \\
        &\quad + \frac{25 \kappa^{*1/2} \eps^2}{16 \theta_1^2 c} r^2(2\rho_K)
        + \frac{c_\alpha }{c \sigma^*} \left( 2 \sigma^* r^2(2\rho_K) + (1 + 4\eps) r^3(2\rho_K) \right) \\
        &\quad + \frac{4\theta_1^2 c_\alpha \eps}{c \sigma^*} r(2\rho_K) \left(r^2(2\rho_K) \vee r_Q^2(2\rho_K,\gamma_Q)\right).
    \end{align*}
    Using $c_\alpha r(2\rho_K) = \alpha_{K,2} < \sigma^*$ in the second and third lines of the latter display, we find
    \begin{align*}
        R(\widehat f_{K,\mu}) - R(f^*)&\leq r^2(2 \rho_K) 
        + \left( \frac{25 \kappa^{*1/2} \eps^2}{16 \theta_1^2 c} + \frac{12(c+2)\eps}{c} + \frac{3c_\mu\eps}{2c} + \frac{3 c_\mu\eps}{c} + 4\eps \right) r^2(2\rho_K) \\
        &\quad + \frac{25 \kappa^{*1/2} \eps^2}{16 \theta_1^2 c} r^2(2\rho_K)
        +  \left( \frac{c_\alpha }{c} 2 r^2(2\rho_K) + \frac{1 }{c } (1 + 4\eps) r^2(2\rho_K) \right) \\
        &\quad + \frac{4\theta_1^2 \eps}{c} \left(r^2(2\rho_K) \vee r_Q^2(2\rho_K,\gamma_Q)\right).
    \end{align*}
    With $c>1$ and $(c+2)/c < 3,$ this recovers
    \begin{align*}
        R(\widehat f_{K,\mu}) - R(f^*)&\leq \left(2 + 2 c_\alpha +  \left(44 + 5 c_\mu \right)\eps + \frac{25 \kappa^{*1/2}}{8 \theta_1^2}\eps^2\right) r^2(2\rho_K) \\
        &\quad + 4\theta_1^2\eps \left(r^2(2\rho_K) \vee r_Q^2(2\rho_K,\gamma_Q)\right),
    \end{align*}
    which completes the proof.
\end{proof}

\section{Proofs for the high-dimensional sparse linear regression} \label{sec.proofs:high_dim_reg}

\subsection{Proof of Theorem~\ref{thm.rates_AMOM_sigmaplus}}
\label{proof:thm.rates_AMOM_sigmaplus}

In Section~\ref{sec.proofs_thm_lasso}, we prove the following Theorem~\ref{thm.MOM_sqrt_lasso}. We show now how this theorem can be used to derive our Theorem~\ref{thm.rates_AMOM_sigmaplus}.

\begin{thm}\label{thm.MOM_sqrt_lasso}
    Assume that $P_{\bX, \xi} \in \mP_{[0,\sigma_+]}$.
    There exists universal constants $\widetilde c_\mu,$ $(\widetilde c_i)_{i=0,\ldots,5}$ that only depend on $\theta_0,\theta_1,\gamma_Q,\gamma_M$ such that the following holds.
    Assume that $|\mI| \geq n / 2,$
    $|\mO| \leq \widetilde c_0 s^* \log(ed/s^*)$, $n \geq s^* \log(ed/s^*)$ and $\bbeta^* \in \mF_{s^*}$.
    
    For every $\iota_K, \iota_\mu \in [1/2, 2]^2$,
    let $K = \lceil \iota_K \widetilde c_2 s^* \log(ed/s^*) \rceil$
    and let $(\widehat\bbeta_{K,\mu,\sigma_+}, \widehat\sigma_{K,\mu,\sigma_+})$ be the MOM$-K$ estimator defined in~\eqref{def.MOM_est} with penalization parameter
    \begin{align*}
        \mu := \iota_\mu \widetilde c_\mu \sqrt{\frac{1}{n} \log\left(\frac{e d}{s^*} \right)}.
    \end{align*}
    Then, for all $p\in[1,2],$ we have
    \begin{align}
    \begin{split}\label{eq.bound_linear_minimax}
        |\widehat\bbeta_{K,\mu,\sigma_+} - \bbeta^*|_p
        &\leq \widetilde c_3 \eps^{-1} \kappa^* \sigma^* {s^*}^\frac{1}{p}
        \sqrt{\frac{1}{n}
        \log \left( \frac{ ed}{s^*} \right)}, \\
        |\widehat\sigma_{K,\mu,\sigma_+} - \sigma^* |
        &\leq c_\alpha \widetilde c_3 \eps^{-1} \kappa^* \sigma^* {s^*}^{\frac{1}{2}}
        \sqrt{\frac{1}{n} \log \left( \frac{ed}{s^*} \right)}.
    \end{split}
    \end{align}
    with probability at least $1 - 4 \exp( - K / 8920).$
\end{thm}

With high probability, we have
\begin{align*}
    |\widehat\bbeta_{K,\mu} - \bbeta^*|_p
    &\leq \widetilde c_3 \eps^{-1} \kappa^* \sigma^* {s^*}^\frac{1}{p}
    \sqrt{\frac{1}{n}
    \log \left( \frac{ ed}{s^*} \right)}.
\end{align*}
We can explicit the value of $\varepsilon^{-1}$ as
\begin{align*}
    \eps^{-1} &= \frac{192 \theta_0^2 (c+2) \big(8 + 134\kappa_{+}^{1/2} ((1+\frac{\sigma_{+}}{\sigma^*})\vee \frac{6}{5}) \big)}{c-2}
    = C \big( (1 + \frac{\sigma_{+}}{\sigma^*}) \vee \frac{6}{5} \big).
\end{align*}
for a constant $C>0$, and therefore
\begin{align*}
    |\widehat\bbeta_{K,\mu} - \bbeta^*|_p
    &\lesssim \big( (1 + \frac{\sigma_+}{\sigma^*}) \vee \frac{6}{5} \big) \sigma^* {s^*}^\frac{1}{p}
    \sqrt{\frac{1}{n}
    \log \left( \frac{ ed}{s} \right)}.
\end{align*}
Since by assumption $\sigma^* < \sigma_+$, we deduce
\begin{align*}
    |\widehat\bbeta_{K,\mu} - \bbeta^*|_p
    &\lesssim \sigma_+ {s^*}^\frac{1}{p}
    \sqrt{\frac{1}{n}
    \log \left( \frac{ ed}{s} \right)}.
\end{align*}
The proof for the bound on $\widehat\sigma_{K,\mu,\sigma_+}$ follows the same computations as it involves a factor of $\varepsilon^{-1}$.

\subsection{Proof of Theorem~\ref{thm.MOM_sqrt_lasso}} \label{sec.proofs_thm_lasso}

In this section we use the results in Theorem~\ref{thm.main_theorem} and the computations in Section~\ref{sec.disc_complexity_sparse} for the sparse linear setting.
For any fixed $\eps\in(0,1),$ the function
\begin{align}\label{def.r_complexity_linear_eps}
    r_{\eps}^2(\rho) = C_{\gamma_P,\gamma_M}^2 
    \begin{cases}
        \max\Big\{\rho \frakm^* \sqrt{\frac{\log d}{n\eps^2}},\ \frac{\rho^2}{n\eps^2}\log\left(\frac{ed}{n\eps^2}\right) \Big\}, & \text{if $\rho \leq \frac{\frakm^* \sqrt{\log d}}{\sqrt{n\eps^2}},$} \\
        \max\Big\{\rho \frakm^* \sqrt{\frac{1}{n\eps^2} \log\left(\frac{e d^2 \frakm^{*2}}{\rho^2 n\eps^2} \right)},\ \frac{\rho^2}{n\eps^2}\log\left(\frac{ed}{n\eps^2}\right) \Big\}, & \text{if } \frac{\frakm^* \sqrt{\log d}}{\sqrt{n\eps^2}} \leq \rho \leq \frac{\frakm^* d}{\sqrt{n\eps^2}},
    \end{cases}
\end{align}
is a strict upper bound on $r^2(\rho)$ defined in~\eqref{def.r_complexity_linear}. By arguing as in the discussion above, the smallest solution of the sparsity equation is of the form
\begin{align*}
    \rho^* = C_{\gamma_P,\gamma_M}^* \frakm^* s^* \sqrt{\frac{1}{n\eps^2} \log\left(\frac{ed}{s^*} \right)},\quad r_{\eps}^2(\rho^*) = C_{\gamma_P,\gamma_M}^{*2} \frac{\frakm^{*2} s^*}{n\eps^2} \log\left(\frac{ed}{s^*} \right).
\end{align*}
For any fixed constant $C>0,$ let $K^*$ be the smallest integer such that 
\begin{align*}
    K^* \geq \frac{n\eps^2}{C^2 \frakm^{*2}} r_\eps^2(\rho^*),
\end{align*}
this matches definition~\eqref{def.K*} in Theorem~\ref{thm.main_theorem} with $C^2 = 384 \theta_1^2$ and $r=r_{\eps}.$ By definition, this is equivalent to
\begin{align*}
    K^* \geq \frac{C_{\gamma_P,\gamma_M}^{*2}}{C^2} s \log\left(\frac{ed}{s} \right),
\end{align*}
which gives the heuristic that the minimum number of blocks is of order $K^* \sim s \log(ed/s).$ For any integer $K\geq K^*,$ we compute the radii $\rho_K$ solving
\begin{align*}
    K = \frac{n\eps^2}{C^2 \frakm^{*2}} r_\eps^2(\rho_K),
\end{align*}
which is a rearrangement of definition~\eqref{eq.rhoK} in Theorem~\ref{thm.main_theorem}. For all $\rho^* \leq \rho_K \lesssim \frakm^*\sqrt{n\eps^2},$ we have 
\begin{align*}
    r_{\eps}^2(\rho_K) = C_{\gamma_P,\gamma_M}^2 
    \rho_K \frakm^* \sqrt{\frac{1}{n\eps^2} \log\left(\frac{e d^2 \frakm^{*2}}{\rho_K^2 n\eps^2} \right)},
\end{align*}
and the implicit solutions $\rho_K$ are of the form
\begin{align*}
    \rho_K = C_K K \frakm^* \sqrt{\frac{1}{n\eps^2} \left[\log\left(\frac{ed^2}{K^2}\right)\right]^{-1} },
\end{align*}
with $C_K$ some absolute constant, for all $K \lesssim n\eps^2.$ In fact, let us compute
\begin{align*}
    \frac{n\eps^2}{K \frakm^{*2}} r_\eps^2(\rho_K) &= C_{\gamma_P,\gamma_M}^2 
     C_K \sqrt{ \left[\log\left(\frac{ed^2}{K^2}\right) \right]^{-1} \log\left(\frac{e d^2 }{C_K^2 K^2} \log\left(\frac{ed^2}{K^2}\right) \right)} \\
     &= C_{\gamma_P,\gamma_M}^2 
     C_K \sqrt{ \frac{\log\left(\frac{e d^2 }{K^2}\right) +  \log\log\left(\frac{ed^2}{K^2}\right) - \log\left(C_K^2\right)}{\log\left(\frac{ed^2}{K^2}\right)} },
\end{align*}
which we want to be equal to the given $C^2.$ Since $d\gg n$ and $K\lesssim n\eps^2,$ without loss of generality $C_K^2 \ll d/n,$ thus
\begin{align*}
    \frac{1}{2} < 1 - \frac{ \log\left(C_K^2\right)}{\log\left(\frac{ed^2}{K^2}\right)}  
    < \frac{\log\left(\frac{e d^2 }{K^2}\right) +  \log\log\left(\frac{ed^2}{K^2}\right) - \log\left(C_K^2\right)}{\log\left(\frac{ed^2}{K^2}\right)} 
    < 2 - \frac{ \log\left(C_K^2\right)}{\log\left(\frac{ed^2}{K^2}\right)} < 2,
\end{align*}
which allows for an absolute constant $C_K \in [C_{\gamma_P,\gamma_M}^2 / (\sqrt{2}C^2), \sqrt{2} C_{\gamma_P,\gamma_M}^2 / C^2]$ recovering the solution.

As mentioned earlier, we can write $K^* = \lceil \widetilde c s^*\log(ed/s^*)\rceil$ with $\widetilde c = C^{*2}_{\gamma_P,\gamma_M} / (384\theta_1^2)$ and, without loss of generality, $\widetilde c \geq 1.$ Assume that the number of outliers is smaller than $\widetilde c_0 s^*\log(ed/s^*)$ with $\widetilde c_0 = \widetilde c / 32,$ this results in $32|\mO| \leq K^*$ and the choice $K=K^*$ is valid in Theorem~\ref{thm.main_theorem}. Then set $\widetilde c_2 = 2 \widetilde c$ and apply Theorem~\ref{thm.main_theorem} separately for any choice $K = \lceil \iota_K \widetilde c_2 s^*\log(ed/s^*)\rceil$ for all $\iota_K \in [1/2,2].$ Then, for any $\iota_\mu\in[1/4,4],$ any penalization parameter of the form
\begin{align*}
    \mu
    &= \iota_\mu c_\mu \eps \frac{r_{\eps}^2(\rho_K)}{\frakm^* \rho_K}
    = \iota_\mu c_\mu C_{\gamma_P,\gamma_M}^2 \eps
    \sqrt{\frac{1}{n\eps^2}
    \log\left(\frac{e d^2 \frakm^{*2}}{\rho_K^2 n\eps^2} \right)}
    = \iota_\mu \widetilde c_\mu \sqrt{\frac{1}{n}
    \log\left(\frac{e d^2 }{K^2} \right)},
\end{align*}
with universal constant $\widetilde c_\mu = c_\mu C_{\gamma_P,\gamma_M}^2,$ is a compatible choice. Furthermore, one finds
\begin{align*}
    \mu
    &= \iota_\mu c_\mu C_{\gamma_P,\gamma_M}^2
    \sqrt{\frac{1}{n} \left(\log\left(\frac{e d^2 }{{s^*}^2} \right)
    - 2\log\log\left(\frac{ed}{s^*} \right)
    - 2\log(\iota_K \widetilde c_2) \right)}.
\end{align*}
We observe that, since $\iota_K \widetilde c_2 \geq 1,$ 
\begin{align*}
    \log\left(\frac{e d^2 }{s^2} \right)
    - 2\log\log\left(\frac{ed}{s^*} \right) - 2\log(\iota_K \widetilde c_2) \leq \log\left(\frac{e d^2 }{{s^*}^2} \right),
\end{align*}
and, with $\log(ed/s^*) \leq (\sqrt{e}d/s^*)^{1/2}$ and $\iota_K \widetilde c_2 \leq (ed/s^*)^{1/4},$
\begin{align*}
    \log\left(\frac{e d^2 }{{s^*}^2} \right) - 2\log\log\left(\frac{ed}{^*} \right) - 2\log(\iota_K \widetilde c_2) \geq \frac{1}{2} \log\left(\frac{e d^2 }{{s^*}^2} \right) - 2\log(\iota_K \widetilde c_2) \geq \frac{1}{4} \log\left(\frac{e d^2 }{{s^*}^2} \right).
\end{align*}
Therefore, any penalization parameter in the smaller interval
\begin{align*}
    \mu \in \left[ \frac{1}{2} \widetilde c_\mu \sqrt{\frac{1}{n} \log\left(\frac{ed^2}{{s^*}^2} \right)}, 2 \widetilde c_\mu \sqrt{\frac{1}{n} \log\left(\frac{ed^2}{{s^*}^2} \right)} \right],
\end{align*}
with absolute constant $\widetilde c_\mu = c_\mu C_{\gamma_P,\gamma_M}^2$, is valid. This matches the construction required by Theorem~\ref{thm.MOM_sqrt_lasso} for any $\iota_K,\iota_\mu \in [1/2,2]^2$ and shows that the penalization parameter $\mu$ can be chosen without knowledge of the moments of the noise.

The convergence rates in Theorem~\ref{thm.main_theorem} become
\begin{align*}
    |\widetilde \bbeta - \bbeta^*|_1 &\leq 2\rho_K = 2 C_K \eps^{-1} \frakm^* K \sqrt{\frac{1}{n} \left[\log\left(\frac{ed^2}{K^2}\right)\right]^{-1} }, \\
    |\widetilde \bbeta - \bbeta^*|_2 &\leq r_{\eps}(2\rho_K) \leq 2C \eps^{-1} \frakm^* \sqrt{\frac{K}{n}}, \\
    |\widehat\sigma_{K,\mu} - \sigma^*| &\leq c_\alpha r_{\eps}(2\rho_K) \leq 2c_\alpha C \eps^{-1} \frakm^* \sqrt{\frac{K}{n}}.
\end{align*}
Finally, for $K \simeq K^*,$ one gets
\begin{align*}
    |\widetilde \bbeta - \bbeta^*|_1
    &\leq 2\rho_{K^*}
    \lesssim 2C_{\gamma_P,\gamma_M}^* \eps^{-1}\frakm^* s^* \sqrt{\frac{1}{n} \log\left(\frac{ed}{s^*}\right)}, \\
    |\widetilde \bbeta - \bbeta^*|_2
    &\leq r_{\eps}(2\rho_{K^*})
    \lesssim 2C_{\gamma_P,\gamma_M}^* \eps^{-1}\frakm^* \sqrt{\frac{s^*}{n} \log\left(\frac{ed}{s^*}\right)}, \\
    |\widehat\sigma_{K,\mu} - \sigma^*|
    &\leq c_\alpha r(2\rho_{K^*})
    \lesssim 2 c_\alpha C_{\gamma_P,\gamma_M}^* \eps^{-1}\frakm^* \sqrt{\frac{s^*}{n} \log\left(\frac{ed}{s^*}\right)}.
\end{align*}
The bounds in~\eqref{eq.bound_linear_minimax} for $p\in[1,2]$ are obtained by applying the interpolation inequality $|\bbeta|_p \leq |\bbeta|_1^{-1+2/p} |\bbeta|_2^{2-2/p}.$ This concludes the proof.

\subsection{Proof of Corollary~\ref{cor.AMOM_estiSigma}}
\label{proof:cor.AMOM_estiSigma}

Recall the definition of signal-to-noise ratio
\begin{align*}
    SNR := \frac{\Var(f^*)}{\Var(\zeta)} = \frac{\Var(f^*)}{\sigma^{*2}},
\end{align*}
and denote
\begin{align*}
    A_Y^2 := \frac{\Var(Y^2)}{\Var(Y)^2},\quad B_Y^2 := \frac{\E[Y]^2}{\Var(Y)}.
\end{align*}

The following proposition allows us to bound above and below the estimator~$\widehat\sigma_{K,+}$ on an event with high probability.
\begin{prop}
    Assume that $\Var(Y) > 0$ and consider the quantities $A_Y,B_Y$ defined above. For any integer
    \begin{align*}
        K \in \left[8|\mO|,\ \frac{n\eps^2}{C^2} \wedge \frac{n}{177 A_Y^2} \wedge \frac{n}{706 B_Y^2} \right],
    \end{align*}
    there exists an event $\Omega(K)$ with probability at least $1 - 2\exp(-7K/3600)$ such that, on this event, the estimator
    \begin{align*}
        \widehat\sigma_{K,+}^2 := Q_{1/2,K}\left[ Y^2 \right] - \left(Q_{1/2,K}\left[ Y \right]\right)^2,
    \end{align*}
    satisfies $\sigma^{*2} \leq 8 \widehat\sigma_{K,+}^2
    \leq 16 \sigma^{*2}(SNR + 1).$
    \label{prop.MOM-VAR}
\end{prop}

Combining Proposition~\ref{prop.MOM-VAR} and Theorem~\ref{thm.rates_AMOM_sigmaplus} by replacing $\sigma_+$ by $\widehat\sigma_{K,+}$ and reasoning on the intersection of both events yields the conclusion.

We now prove Proposition~\ref{prop.MOM-VAR}.
\begin{proof}
    We start with
    \begin{align*}
        \Var(Y) = \Var(f^*(\bX) + \zeta) = \Var(f^*(\bX)) + \sigma^{*2} + 2\Cov(f^*(\bX),\zeta) = \Var(f^*(\bX)) + \sigma^{*2},
    \end{align*}
    where in the last step we have used that $f^*(\bX) = \bX^\top\bbeta^*$ is the orthogonal projection of the square-integrable random variable $Y = \bX^\top\bbeta^* + \zeta$ onto the closed and convex set of square-integrable random variables $\mA := \{\bX^\top\bbeta: \bbeta\in\R^d\}.$ Thus, $\Var(Y) = \sigma^{*2}(SNR + 1).$
    
    We apply Lemma~\ref{lemma.bound_card_Z} to the variable $Z=Y^2.$ We choose $\eta=1/2$ and $\gamma = 7/8,$ $x = 1/15,$ $\delta_{K,n}^2 = a_{K,n}^2 := 15 (K/n) \Var(Y^2),$ so that $\gamma (1 - 1/15 - x) \geq 1/2,$ in fact
    \begin{align*}
        \gamma \left(1 - \frac{1}{15} - x \right) = \frac{7}{8}\left(1 - \frac{1}{15} - \frac{1}{15} \right) = \frac{91}{120} > \frac{1}{2}.
    \end{align*}
    Therefore, on an event $\Omega_1(K)$ with probability at least $1-\exp(- 7 K / 3600),$ we have $Q_{1/2,K}\left[ Y^2 \right] \in [\E[Y^2] - a_{K,n}, \E[Y^2] + a_{K,n}].$
    
    We now repeat the argument for $Z=Y.$ We choose again $\eta=1/2$ and $\gamma = 7/8,$ $x =1/15,$ $\delta_{K,n}^2 = b_{k,n}^2 := 15 (K/n) \Var(Y),$ so that $\gamma (1 - 1/15 - x) \geq 1/2.$ Therefore, on an event $\Omega_2(K)$ with probability at least $1-\exp(- 7 K / 3600),$ we have $(Q_{1/2,K}\left[ Y \right])^2 \in [(\E[Y] - b_{K,n})^2, (\E[Y] + b_{K,n})^2].$
    
    We now work on the event $\Omega(K) = \Omega_1(K)\cap\Omega_2(K)$ which has probability at least $1-2\exp(- 7 K / 3600).$ We have
    \begin{align*}
        \widehat\sigma_{K,+}^2 &\in \Big[ \Var(Y) - a_{K,n} - 2\E[Y]b_{K,n} - b_{K,n}^2,\  \Var(Y) + a_{K,n} + 2\E[Y]b_{K,n} - b_{K,n}^2 \Big],
    \end{align*}
    with $a_{K,n}^2 = 15 (K/n) \Var(Y^2),\ b_{K,n}^2 = 15 (K/n) \Var(Y).$ We now show that
    \begin{align*}
        \frac{\sigma^{*2}}{4} \leq 2\widehat\sigma_{K,+}^2 \leq 4 \Var(Y),
    \end{align*}
    which would give the claim. We start with the lower bound, we want
    \begin{align*}
        1 \leq \frac{2\Var(Y) - 2a_{K,n} - 4\E[Y]b_{K,n} - 2b_{K,n}^2}{\sigma^{*2}/4},
    \end{align*}
    and we show the stronger
    \begin{align*}
        \max\left\{ \frac{2a_{K,n}}{\sigma^{*2}/4},\ \frac{4\E[Y]b_{K,n}}{\sigma^{*2}/4},\ \frac{2b_{K,n}^2}{\sigma^{*2}/4} \right\} &\leq \frac{1}{3}\left(\frac{2\Var(Y)}{\sigma^{*2}/4} - 1 \right).
    \end{align*}
    By construction, we have
    \begin{align*}
        \frac{8a_{K,n}}{\sigma^{*2}} &= \frac{\sqrt{\Var(Y^2)}}{\sigma^{*2}} \sqrt{\frac{960 K}{n}}, \\
        \frac{16\E[Y]b_{K,n}}{\sigma^{*2}} &= \frac{\E[Y]\sqrt{\Var(Y)}}{\sigma^{*2}} \sqrt{\frac{3840 K}{n}}, \\
        \frac{8b_{K,n}^2}{\sigma^{*2}} &= \frac{\Var(Y)}{\sigma^{*2}} \frac{120 K}{n},
    \end{align*}
    and the quantities $A_Y,B_Y$ are defined in such a way that $\sqrt{\Var(Y^2)} = A_Y \Var(Y)$ and $\E[Y] = B_Y \sqrt{\Var(Y)}.$ Therefore, it is enough that
    \begin{align*}
        A_Y (SNR + 1) \sqrt{\frac{8640 K}{n}} &\leq 8(SNR + 1) - 1, \\
        B_Y (SNR + 1) \sqrt{\frac{34560 K}{n}} &\leq 8(SNR + 1) - 1, \\
        (SNR + 1) \frac{360 K}{n} &\leq 8(SNR + 1) - 1.
    \end{align*}
    We now divide by $(SNR + 1)$ and use $1/(SNR + 1) \leq 1,$ the stronger condition
    \begin{align*}
        A_Y \sqrt{\frac{8640 K}{n}} &\leq 7, \\
        B_Y \sqrt{\frac{34560 K}{n}} &\leq 7, \\
        \frac{360 K}{n} &\leq 7,
    \end{align*}
    is then satisfied if $K \leq n / \max\{ 177 A_Y^2,\ 706 B_Y^2 ,\ 52\},$ which is true by assumption on the upper bound on the number of blocks. This completes the proof of $\sigma^{*2} \leq 8\widehat\sigma_{K,+}^2$ on the event $\Omega(K).$
    
    We now deal with $2\widehat\sigma_{K,+}^2 \leq 4 \Var(Y).$ Since the quantity $-b_{K,n}^2$ is negative, it is sufficient that $2\Var(Y) + 2a_{K,n} + 2\E[Y] b_{K,n} \leq 2 \Var(Y)$ and, dividing by $\sigma^{*2},$ 
    \begin{align*}
        \frac{2 a_{K,n}}{\sigma^{*2}} + \frac{2\E[Y]b_{K,n}}{\sigma^{*2}} \leq \frac{2\Var(Y)}{\sigma^{*2}}.
    \end{align*}
    We show the stronger inequalities
    \begin{align*}
        \frac{2 a_{K,n}}{\sigma^{*2}} &\leq \frac{\Var(Y)}{\sigma^{*2}}, \\
        \frac{2\E[Y]b_{K,n}}{\sigma^{*2}} &\leq \frac{\Var(Y)}{\sigma^{*2}},
    \end{align*}
    by arguing as for the previous step. It is sufficient that
    \begin{align*}
        A_Y (SNR + 1) \sqrt{\frac{60 K}{n}} &\leq (SNR + 1), \\
        B_Y (SNR + 1) \sqrt{\frac{60 K}{n}} &\leq (SNR + 1),
    \end{align*}
    which holds if $K \leq n / \max\{ 60 A_Y^2,\ 60 B_Y^2\},$ and the latter is true by assumption on the upper bound on the number of blocks. This completes the proof of $2\widehat\sigma_{K,+}^2 \leq 4\Var(Y)$ on the event $\Omega(K).$
\end{proof}

\section{Proofs for adaptivity to the sparsity level \texorpdfstring{$s$}{s}} \label{sec.proofs:adaptivity_s}

\subsection{A general algorithm for simultaneous adaptivity}

In this section, we prove a more general theorem, that will yield Theorem~\ref{thm.MOM_sqrt_lasso_adaptive} as a particular case.

\textbf{Algorithm for adaptation to sparsity.} The steps of the adaptive procedure are as follows.
\begin{itemize}
    \item Let $w_1, w_2, w_3$ be three functions $[1,d/e] \to \R_+$ and set $M := \lfloor \log_2(s_+) \rfloor.$
    \item For every $m\in\{1,\ldots,M+1\},$ compute $(\widehat\bbeta_{(2^m)}, \widehat\sigma_{(2^m)})$.
    \item Set 
    \begin{align*}
        \mM := \bigg\{ & m\in\{1,\ldots,M\} : \, \text{for all $k\geq m$, }
        |\widehat\bbeta_{(2^{k-1})} - \widehat\bbeta_{(2^k)}|_1 \leq C_1 \widehat \sigma w_1(2^k)
        , \\
        &|\widehat\bbeta_{(2^{k-1})} - \widehat\bbeta_{(2^k)}|_2 \leq C_2 \widehat \sigma w_2(2^k)
        \text{ and }
        |\widehat\sigma_{(2^{k-1})} - \widehat\sigma_{(2^k)}| \leq C_3 \widehat \sigma w_3(2^k)
        \bigg\}.
    \end{align*}
    \item Set $\widetilde m := \min \mM,$ with the convention that $\widetilde m := M+1$ if $\mM = \emptyset.$ 
    \item Define $\widetilde s := 2^{\widetilde m}$ and $(\widetilde\bbeta, \widetilde\sigma) := (\widehat\bbeta_{(\widetilde s)}, \widehat\sigma_{(\widetilde s)}).$
\end{itemize}

\begin{defi}
    Let $\Theta$ be a subset of $\R^d \times \R_+$ and $\| \cdot \|$ a norm on $\Theta$.
    For a given $s \in \{2,\ldots,d/(2e)\}$, we say that an estimator $\widehat \theta_{(s)} \in \Theta$ robustly converges to $\theta^* \in \Theta$ in norm $\| \cdot \|$ with bound $C_1 \sigma^* w(s)$ if
    \begin{align}
        &\inf_{\bbeta^*\in \mF_s, \, \sigma^* > 0} P_{\beta^*,P_{\bX,\zeta}}^{\otimes n}
        \left( \forall \mD' \in \mD(N), \|\widehat \theta_{(s)}(\mD') - \theta^*\| \leq C_1 \sigma^* w(s) \right)
        \geq 1 - \widetilde c_6 C_2 \left(\frac{s}{ed} \right)^{\widetilde c_5 s} - u_n,
        \label{eq:general_sup_s}
        \\
        &\inf_{\bbeta^*\in \widetilde \mF_{2s}, \, \sigma^* > 0} 
        P_{\beta^*,P_{\bX,\zeta}}^{\otimes n}
        \left( \forall \mD' \in \mD(N), \|\widehat \theta_{(s)}(\mD') - \theta^*\| \leq C_1 \sigma^* w(s)  \right)
        \geq 1 -  \widetilde c_6 C_2 \left(\frac{2s}{ed} \right)^{2 \widetilde c_5  s} - u_n.
        \label{eq:general_sup_2s}
    \end{align}
    and if the function $w(\cdot): [1, d/e] \rightarrow \R_+$
    satisfies the following conditions:
    \begin{enumerate}
        \item $w(\cdot)$ is increasing on $[1, d/e]$ ;
        
        \item There exists a constant $C' > 0$ such that, for all $m=1, \dots, \lfloor \log_2(s_+) \rfloor$, we have $$\sum_{k=1}^m w(2^k) \leq C' \cdot w(2^m) \; ;$$
        
        \item There exists a constant $C''> 0$ such that, for all $b=1, \dots, s_+$, $$w(2b) \leq C'' w(b).$$
    \end{enumerate}
    \label{defi:function_w}
\end{defi}

\begin{thm}[Joint adaptation of $(\widehat\bbeta,\widehat \sigma)$ to $s$]
    \label{thm.joint_adapt}
    Let $s_{+}\in\{2,\ldots,d/(2e)\}$ and for $s=1, \dots, 2 s_+$,
    let $(\widehat \bbeta_{(s)}, \widehat \sigma_{(s)})$ be a joint estimator of $(\bbeta^*, \sigma^*)$ such that
    \begin{enumerate}
        \item $\widehat \bbeta_{(s)}$ robustly converges to $\bbeta^*$ in $|\cdot|_1$-norm with bound $C_1 \sigma^* w_1(s)$;
        \item $\widehat \bbeta_{(s)}$ robustly converges to $\bbeta^*$ in $|\cdot|_2$-norm with bound $C_2 \sigma^* w_2(s)$;
        \item $\widehat \sigma_{(s)}$ robustly converges to $\sigma^*$ in $|\cdot|$-norm with bound $C_3 \sigma^* w_3(s)$;
    \end{enumerate}
    for some constants $N > 0$, $\widetilde c_6 > 0$ $C_1 > 0,$ $u_n > 0$ and for some functions $w_1, w_2, w_3$ such that $C_3 w_3(2s_+) \leq 1/2$.
    Then, there exists constants $\widetilde C_1, \widetilde C_2, \widetilde C_3$ such that,
    for all $s^*\in\{1,\ldots,s_{+}\}$ and $\bbeta^* \in \widetilde\mF_{s^*},$
    the aggregated estimator $(\widetilde \bbeta, \widetilde \sigma, \widetilde s)$ satisfies 
    \begin{align*}
        P_{\beta^*,P_{\bX,\zeta}}^{\otimes n}
        &\Big( \forall \mD' \in \mD(N),
        |\widetilde\bbeta - \bbeta^*|_1 \leq \widetilde C_1 \sigma^* w_1(s^*), \,
        |\widetilde\bbeta - \bbeta^*|_2 \leq \widetilde C_2 \sigma^* w_2(s^*), \,
        |\widetilde\sigma - \sigma^*| \leq \widetilde C_3 \sigma^* w_3(s^*)  \Big) \\
        &\geq 1 - 21(\log_2(s_{+})+1)^2 \Bigg( \widetilde c_5 \left(\frac{2s^*}{d} \right)^{2 \widetilde c_6 s^*} + u_n\Bigg)
        - 21 \widetilde c_6 \left(\frac{2^{M+1}}{d} \right)^{\widetilde c_5 2^{M+1}} - 21 u_n
    \end{align*}
    and
    \begin{align*}
        \P_{\bbeta^*}\left(\forall \mD' \in \mD(N), \widetilde s \leq s^* \right)
        \geq 1 - 6 (\log_2(s_{+})+1)^2 \Bigg( \widetilde c_6 \left(\frac{2s^*}{d} \right)^{2 \widetilde c_5 s^*} + u_n\Bigg)
        - 6 \, \widetilde c_6 \left(\frac{2^{M+1}}{d} \right)^{\widetilde c_5 2^{M+1}} - 6 u_n.
    \end{align*}
\end{thm}

We adapt the proof given in \cite[Section 7.3.1]{derumigny2018improved} to this new setting where the adaptation is done on both estimators simultaneously. Proof of Theorem~\ref{thm.joint_adapt} is given in Section~\ref{proof:thm.joint_adapt}.

\subsection{Proof of Theorem~\ref{thm.MOM_sqrt_lasso_adaptive}}
\label{sec.proofs_thm_lasso_adaptive}

To prove Theorem~\ref{thm.MOM_sqrt_lasso_adaptive}, we will apply Theorem~\ref{thm.joint_adapt}. We first check that its assumption are satisfied.
We choose the functions $w_1(s) = s \sqrt{(1/n) \log(ed/s)}$, $w_2(s) = w_3(s) = w_1(s) = s^{1/2} \sqrt{(1/n) \log(ed/s)}$. By Lemma 4.4 in \cite{derumigny2018improved}, $w_1$, $w_2$ and $w_3$ satisfy the 3 conditions in Definition~\ref{defi:function_w}.

It remains to check that the following bounds in probability \eqref{eq:general_sup_s} and \eqref{eq:general_sup_2s} hold for all $s^*=1,\ldots,s_{+}$.
Applying Theorem~\ref{thm.rates_AMOM_sigmaplus} gives
\begin{align*}
    \inf_{\bbeta^*\in \mF_{s^*}, \, \sigma^* > 0}
    P_{\beta^*,P_{\bX,\zeta}}^{\otimes n}
    \Bigg(
    \sup_{ \mD' \in \mD(\widetilde c_3 \frakr_\mO) }
    \bigg\{ \frakr_2^{-1}
    \big| \widehat\sigma (\mD') - \sigma^* \big|
    &\vee \hspace{-0.1cm}
    \sup_{p \in [1,2]} \frakr_p^{-1}
    \big| \widehat\bbeta (\mD') - \bbeta^* \big|_p \bigg\}
    \leq \widetilde c_4 \sigma_+ \Bigg)
    \geq 1 - 4 \Big( \frac{s^*}{ed} \Big)^{\widetilde c_5 s^*},
\end{align*}
proving that the bound \eqref{eq:general_sup_s} is satisfied.

Furthermore, we have
\begin{align*}
    K_{2s} &= \left\lceil \widetilde c_2 2s^* \log\left(\frac{ed}{2s^*}\right) \right\rceil
    = \left\lceil \widetilde c_2 2s^* \left(\log\left(\frac{ed}{s^*}\right) + \log(2) \right) \right\rceil =  \gamma(2s^*) K_{s^*}, \\ 
    \mu_{2s^*} &= \widetilde c_\mu \sqrt{\frac{1}{n} \log\left(\frac{ed}{2s^*} \right)} = \widetilde c_\mu \sqrt{\frac{1}{n} \log\left(\frac{ed}{s^*} \right) - \frac{\log(2)}{n}} = \widetilde\gamma(2s^*)\mu_s,
\end{align*}
with some $\gamma(2s^*), \widetilde\gamma(2s^*) \in [1/2, 2]^2.$
This gives $\widehat\bbeta_{K_{2s^*} / \gamma(2s^*),
\, \mu_{2s^*} / \widetilde\gamma(2s^*)} = \widehat\bbeta_{K_{s^*},\mu_{s^*}}$
and, applying Theorem~\ref{thm.rates_AMOM_sigmaplus} with $2s^*$ instead of $s^*,$ yields
\begin{align*}
    \inf_{\bbeta^*\in \mF_{2s^*}, \, \sigma^* > 0}
    P_{\beta^*,P_{\bX,\zeta}}^{\otimes n}
    \Bigg(& \forall \mD' \in \mD(\widetilde c_3 \frakr_\mO) , \,
    \bigg\{ 
    \big| \widehat\sigma (\mD') - \sigma^* \big|
    \leq \widetilde c_4 \sigma_+ \sqrt{\frac{2s^*}{n} \log \left( \frac{ed}{2s^*} \right)} \\
    &\text{ and } \forall p \in [1,2], \,
    \big| \widehat\bbeta (\mD') - \bbeta^* \big|_1
    \leq \widetilde c_4 \sigma_+ (2s^*)^{1/p}
    \sqrt{\frac{1}{n} \log \left( \frac{ed}{2s^*} \right)} \Bigg)
    \geq 1 - 4 \Big( \frac{2s^*}{ed} \Big)^{\widetilde c_5 2s^*},
\end{align*}
proving that the bound \eqref{eq:general_sup_2s} is satisfied with $\widetilde c_4$ multiplied by $4$.

\subsection{Proof of Theorem~\ref{thm.joint_adapt}}
\label{proof:thm.joint_adapt}

We choose $s \in [1, s_+]$ and assume that $\bbeta^* \in \mF_s$. Define $\P := \P_{\beta^*,\sigma^*}$ and $m_0 := \lfloor \log_2(s) \rfloor + 1$.
For $p=1,2$, define
$\widehat \theta_{(s)}^{(p)} := \widehat \bbeta_{(s)}$,
$\widetilde \theta^{(p)} := \widetilde \bbeta$,
$\theta^{(p),*} := \bbeta^*$ and $d_p$ be the distance on $\R$ induced by the norm $|\cdot|_p$.
Define $\widehat \theta_{(s)}^{(3)} = \widehat \sigma_{(s)}$,
$\widetilde \theta^{(3)} := \widetilde \sigma$,
$\theta^{(3),*} := \sigma^*$ and $d_3$ be the distance on $\R$ induced by the absolute value.

\textbf{Bound on $\widehat \sigma$ with high probability.}
Combining the definition $\widehat\sigma = \widehat\sigma_{2s_+}$ with the assumptions that $C_3 w_3(2s_+) \leq 1/2$ and that $\widehat \sigma_{(s)}$ robustly converges to $\sigma^*$ in $|\cdot|$-norm with bound $C_3 \sigma^* w_3(s)$, we get
\begin{align}
    \P\left(\forall \mD' \in \mD(N), \sigma^*/2 \leq \widehat\sigma \leq (3/2) \sigma^* \right)
    \geq 1 - \widetilde c_6 \left(\frac{2^{M+1}}{d} \right)^{\widetilde c_5 2^{M+1}} - u_n
    \label{eq:bound_hatSigma_s+}
\end{align}

\textbf{Bound on the probability $\P(\exists \mD' \in \mD(N), \, \tilde m \geq m_0 + 1)$.}
We have
\begin{align*}
    &\P(\exists \mD' \in \mD(N), \, \widetilde m \geq m_0 + 1)
    \leq \sum_{m = m_0+1}^M
    \P(\exists \mD' \in \mD(N), \tilde m = m_0 + 1) \\
    &\leq \sum_{m = m_0+1}^M \sum_{k=m}^M
    \P \bigg( \exists \mD' \in \mD(N), 
    |\widehat\bbeta_{(2^{k-1})} - \widehat\bbeta_{(2^k)}|_1
    > 4 C_1 \widehat \sigma w_1(2^k) \\
    &\hspace{8em}
    \text{ or }
    |\widehat\bbeta_{(2^{k-1})} - \widehat\bbeta_{(2^k)}|_2
    > 4 C_2 \widehat \sigma w_2(2^k)
    \text{ or }
    |\widehat\sigma_{(2^{k-1})} - \widehat\sigma_{(2^k)}|
    > 4 C_3 \widehat \sigma w_3(2^k)
    \bigg) \displaybreak[0] \\
    &\leq \sum_{m = m_0+1}^M \sum_{k=m}^M
    \P \bigg( \exists \mD' \in \mD(N), \exists p \in [3], d_p \big( 
    \widehat \theta_{(2^{k-1})}^{(p)}, \widehat \theta_{(2^{k})}^{(p)} \big)
    > 4 C_p \widehat \sigma w_p(2^k)
    \bigg) \displaybreak[0] \\
    &\leq \sum_{p=1}^3 \sum_{m = m_0+1}^M \sum_{k=m}^M
    \P \bigg( \exists \mD' \in \mD(N), 
    d_p \big( 
    \widehat \theta_{(2^{k-1})}^{(p)}, \widehat \theta_{(2^{k})}^{(p)} \big)
    > 4 C_p \widehat \sigma w_p(2^k)
    \bigg) \displaybreak[0] \\
    &\leq \sum_{p=1}^3 \sum_{m = m_0+1}^M \sum_{k=m}^M
    \P \bigg( \exists \mD' \in \mD(N), 
    d_p \big( \widehat \theta_{(2^{k-1})}^{(p)}, \theta^{(p),*} \big)
    > 4 C_p \widehat \sigma w_p(2^k)
    \bigg) \\
    & \hspace{3em} + \P \bigg( \exists \mD' \in \mD(N), 
    d_p \big( \widehat \theta_{(2^{k})}^{(p)}, \theta^{(p),*} \big)
    > 4 C_p \widehat \sigma w_p(2^k)
    \bigg) \displaybreak[0] \\
    &\leq 2 \sum_{p=1}^3 \sum_{m = m_0+1}^M \sum_{k=m-1}^M
    \P \bigg( \exists \mD' \in \mD(N), 
    d_p \big( \widehat \theta_{(2^{k-1})}^{(p)}, \theta^{(p),*} \big)
    > 4 C_p \widehat \sigma w_p(2^k)
    \bigg) \displaybreak[0] \\
    &\leq 2 \sum_{p=1}^3 \sum_{m = m_0+1}^M \sum_{k=m-1}^M
    \P \bigg( \exists \mD' \in \mD(N), 
    d_p \big( \widehat \theta_{(2^{k-1})}^{(p)}, \theta^{(p),*} \big)
    > 4 C_p \widehat \sigma w_p(2^k) ,
    \widehat \sigma \geq \frac{\sigma}{2} \bigg) \\
    & \hspace{3em}
    + 6 \, \P \bigg( \exists \mD' \in \mD(N), \hat \sigma < \frac{\sigma}{2} \bigg).
\end{align*}
Combining the previous equation with Equation~(\ref{eq:bound_hatSigma_s+}), and then with the assumption on the bound on the estimator $\widehat \theta_{(2^{k-1})}^{(p)}$ for the distance $d_p$, we get
\begin{align}
    \P(&\exists \mD' \in \mD(N), \, \widetilde m \geq m_0 + 1) \nonumber \\
    &\leq 2 \sum_{p=1}^3 \sum_{m = m_0+1}^M \sum_{k=m-1}^M
    \P \bigg( \exists \mD' \in \mD(N), d_p \big( \widehat \theta_{(2^{k-1})}^{(p)}, \theta^{(p),*} \big)
    > 2 C_p \widehat \sigma w_p(2^k) \bigg) \nonumber \\
    & \hspace{3cm} - 6 \widetilde c_6 \left(\frac{2^{M+1}}{d} \right)^{\widetilde c_5 2^{M+1}}
    - 6 u_n \nonumber \displaybreak[0] \\
    &\leq 6 M^2 \widetilde c_6 \left( \left( \frac{2s}{p} \right)^{2 \widetilde c_5 s} + u_n \right)
    - 6 \widetilde c_6 \left(\frac{2^{M+1}}{d} \right)^{2^{M+1} \widetilde c_5 } - 6 u_n \nonumber \displaybreak[0] \\
    &\leq 6 (\log_2(s_+) + 1)^2 \widetilde c_6 \left( \left( \frac{2s}{p} \right)^{2 \widetilde c_6  s} + u_n \right)
    - 6 \widetilde c_6 \left(\frac{2^{M+1}}{d} \right)^{\widetilde c_5 2^{M+1}} - 6 u_n.
    \label{eq:bound_tilde_m_m0}
\end{align}
This gives the bound on $\tilde s$ as claimed.

\medskip

\textbf{Bound on the deviation probability of $\widetilde \theta^{(p)}$.}
For any $a > 0$, we have 
\begin{align}
    \P \big(\exists \mD' \in \mD(N),  d_p(\widetilde \theta^{(p)}, \theta^{(p),*}) \geq a \big) 
    &\leq \P \big(\exists \mD' \in \mD(N),  d_p(\widetilde \theta^{(p)}, \theta^{(p),*}) \geq a, \widetilde m \leq m_0 \big) \nonumber \\
    &+ \P(\exists \mD' \in \mD(N),  \widetilde m \geq m_0 + 1).
    \label{eq:proba_decomp_risk_tilde_beta}
\end{align}
On the event $\{ \widetilde m \leq m_0 \}$, we have the decomposition 
\begin{equation}
    d_p(\widetilde \theta^{(p)}, \theta^{(p),*})
    \leq \sum_{k=\widetilde m + 1}^{m_0}
    d_p \left(\widehat \theta_{(2^{k-1})}^{(p)} , \widehat \theta_{(2^{k})}^{(p)} \right)
    + d_p(\widehat \theta_{(2^{m_0})}^{(p)}, \theta^{(p),*}).
    \label{eq:proba_decomp_distance_beta_tilde_beta}
\end{equation}
Using the assumption on the function $w_p$, we get that,
\begin{align}
    \sum_{k=\widetilde m + 1}^{m_0}
    d_p \left(\widehat \theta_{(2^{k-1})}^{(p)} , \widehat \theta_{(2^{k})}^{(p)} \right)
    &\leq \sum_{k=\widetilde m + 1}^{m_0} 4 \hat \sigma C_0 w(2^k) \nonumber \\
    &\leq 4 \hat \sigma C_p C' w_p(2^{m_0})
    \leq 4 \hat \sigma C_p C' C'' w_p(s).
    \label{eq:bound_sum_distance_hat_beta}
\end{align}
We have $2^{m_0} \leq 2s$, therefore applying Assumption (\ref{eq:general_sup_2s}), we have with $\P_{\beta^*,\,\sigma^*}$-probability at least
$1 - \widetilde c_5 \left( 2s/p \right)^{2 \widetilde c_6 s} - u_n$, for all $\mD' \in \mD(N)$,
\begin{align}
    d_p(\widehat \theta_{(2^{m_0})}^{(p)}, \theta^{(p),*})
    \leq C_p \widehat \sigma w(2s)
    \leq C_p C'' \widehat \sigma w(s).
    \label{eq:bound_distance_beta_hat_beta}
\end{align}
Combining Equations (\ref{eq:proba_decomp_distance_beta_tilde_beta}), (\ref{eq:bound_sum_distance_hat_beta}), (\ref{eq:bound_distance_beta_hat_beta}) and (\ref{eq:bound_hatSigma_s+}), we get with $\P_{\beta^*}$-probability
at least $1 - \widetilde c_5 (2s/p)^{2 \widetilde c_6 s} -
\widetilde c_5 (2^{M+1}/p)^{\widetilde c_6 2^{M+1}}  - 2u_n$, for all $\mD' \in \mD(N)$,
\begin{align}
    d_p(\widetilde \theta^{(p)}, \theta^{(p),*})
    \leq \left( 4 C_p C' C'' + (3/2) C_p C'' \right) \sigma w(s).
    \label{eq:bound_tilde_beta_true_beta}
\end{align}

\medskip

Combining Equation~(\ref{eq:bound_tilde_m_m0}) with Equations (\ref{eq:proba_decomp_risk_tilde_beta}) and (\ref{eq:bound_tilde_beta_true_beta}), we finally get that
\begin{align*}
    &\P \left(\exists \mD' \in \mD(N), d_p(\widetilde \theta^{(p)}, \theta^{(p),*})
    \geq \left( 4 C_p C' C'' + (3/2) C_p C'' \right) \sigma w_p(s) \right) \\
    & \hspace{1.5cm} \leq 7 (\log_2(s_+) + 1)^2 \left( \widetilde c_6 \left( \frac{2s}{p} \right)^{2 \widetilde c_5 s}
    + u_n \right) - 7 \, \widetilde c_6 \left(\frac{2^{M+1}}{d} \right)^{\widetilde c_5 2^{M+1}} - 7 u_n.
\end{align*}
By a union bound, we then obtain
\begin{align*}
    \P_{\bbeta^*, \, \sigma^*} & \left(\forall \mD' \in \mD(N),
    \forall p=1,2,3, \,
    d_p(\widetilde \theta^{(p)}, \theta^{(p),*})
    \geq \left( 4 C' C'' + (3/2) C'' \right) C_p \sigma w_p(s)  \right) \\
    &\geq 1 - 21 (\log_2(s_{+})+1)^2 \Bigg(\widetilde c_6 \left(\frac{2s}{d} \right)^{2 \widetilde c_5 s} + u_n\Bigg)
    - 21 \, \widetilde c_6 \left(\frac{2^{M+1}}{d} \right)^{\widetilde c_5 2^{M+1}} - 21 u_n.
\end{align*}
as claimed.

\section{Auxiliary results}\label{sec.proofs_aux}

In this section we give auxiliary results that are used in the proofs of the main results.

\begin{lem}[Lemma 6 in~\cite{lecue2020robustML}] \label{lemma.6}
    Let $\rho \geq 0$, $\Gamma_{f^*}(\rho) := \bigcup_{f \in \mF : \, \|f-f^*\|\leq \rho/20} \big( \partial ||\cdot|| \big)_f.$ For all $g \in \mF,$ we have
    \begin{align*}
        \|f^*\| - \|g\| \leq \frac{\rho}{10} 
        - \sup_{z^* \in \Gamma_{f^*}(\rho)} z^* (g - f^*).
    \end{align*}
\end{lem}

We recall here the definition of quantiles we used in Section \ref{sec.estimator}. For any $K\in\N,$ set $[K]=\{1,\ldots,K\}.$ For all $\alpha\in(0,1)$ the $\alpha-$quantile of a vector $\bx=(x_1,\ldots,x_K)\in\R^K$ is any element $Q_\alpha[\bx]$ of the set
\begin{align*}
    \mQ_\alpha[\bx] := \Big\{ u\in\R:\ \big|\{k\in[K]:x_k\geq u\} \big| \geq (1-\alpha)K,\  \big|\{k\in[K]:x_k\leq u\} \big| \geq \alpha K \Big\}.
\end{align*}
For all $t\in\R,$ we write $Q_\alpha[\bx]\geq t$ when there exists $J\subset[K]$ such that $|J|\geq(1-\alpha)K$ and, for all $j\in J,$ $x_j\geq t.$ We write $Q_\alpha[\bx]\leq t$ if there exists $J\subset[K]$ such that $|J|\geq \alpha K$ and, for all $j\in J,$ $x_j\leq t.$

\begin{lem}\label{lemma.quantile_prop}
    We have the following properties.
    \begin{enumerate}[label=\arabic*., ref=\arabic*]
        \item \textbf{Monotonicity} \\
        For all $\alpha\in(0,1),$ $\beta\in(0,\alpha]$ and $\bx\in\R^K,$ $Q_{\beta}[\bx] \leq Q_{\alpha}[\bx].$ \label{quant_monotonicity}
        \item \textbf{Opposite} \\
        For all $\alpha\in(0,1)$ and $\bx\in\R^K,$ $Q_{\alpha}[\bx] \geq -Q_{1-\alpha}[-\bx].$ \label{quant_opposite}
        \item \textbf{Linearity} \\
        For all $\alpha\in(0,1),$ $\bx\in\R^K$ and $a,b\in\R,$ $Q_{\alpha}[a\bx + b] = |a|Q_{\alpha}[\sgn(a)\bx] + b.$ \label{quant_linearity}
        \item \textbf{Difference} \\
        For all $\alpha,\beta \in (0, 1)$ and $\bx,\by\in\R^K,$ $Q_{\alpha}[\bx-\by]\leq Q_{\alpha+\beta}[\bx] - Q_{\beta}[\by].$ \label{quant_difference}
        \item \textbf{Triangular} \\
        For all $\alpha,\beta \in (0, 1)$ and $\bx,\by\in\R^K,$ $Q_{\alpha}[\bx+\by]\leq Q_{\alpha+\beta}[\bx] + Q_{1-\beta}[\by].$ \label{quant_triangular}
    \end{enumerate}
\end{lem}
\begin{proof}[Proof of Lemma \ref{lemma.quantile_prop}]
    We prove property \ref{quant_monotonicity}.
    Write $\bx=(x_j)_{j\in[K]}.$ The property $Q_{\beta}[\bx] \leq Q_{\alpha}[\bx]$ is true by construction, because $Q_{\alpha}[\bx] \leq u$ implies that there are at least $\alpha K \geq \beta K$ components such that $x_j \leq u.$
    
    We prove property \ref{quant_opposite}.
    Write $\bx=(x_j)_{j\in[K]}$ and $Q_{\alpha}[\bx] = u,$ then there are at least $(1-\alpha)K$ components such that $x_j\geq u$ and at least $\alpha K$ components such that $x_j\leq u.$ We now show that $u \geq -Q_{1-\alpha}[-\bx].$ This is equivalent to $Q_{1-\alpha}[-\bx] \geq -u,$ which requires at least $\alpha K$ components such that $-x_j \geq -u,$ that is, $x_j \leq u.$ The latter is true by construction.
   
    We prove property \ref{quant_linearity}.
    Write $\bx=(x_j)_{j\in[K]}.$ The property $Q_{\alpha}[a\bx + b] = Q_{\alpha}[a\bx] + b$ follows from the definition, that is, if $Q_{\alpha}[a\bx] = u$ then there are at least $(1-\alpha)K$ components such that $a x_j \geq u$ and at least $\alpha K$ components such that $a x_j \leq u.$ Thus, the same components also satisfy $a x_j + b \geq u + b$ or $a x_j + b \leq u + b.$ It remains to show that $Q_{\alpha}[a\bx] = |a|Q_{\alpha}[\sgn(a)\bx].$ Let $Q_{\alpha}[a\bx] = u.$ We show that we have at least $(1-\alpha)K$ components $\sgn(a)x_j \geq u/|a|$ and at least $\alpha K$ components $\sgn(a)x_j \leq u/|a|.$ The latter conditions are equivalent to $|a|\sgn(a)x_j \geq u$ and $|a|\sgn(a)x_j \leq u.$ This is enough to conclude since $a = \sgn(a)|a|$ and $Q_{\alpha}[a\bx] = u.$
   
    We prove property \ref{quant_difference}.
    Write $\bx=(x_j)_{j\in[K]},$ $\by=(y_i)_{i\in[K]}$ and $Q_{\alpha+\beta}[\bx] = u,$ $Q_{\beta}[\by]=l.$ By construction:
    \begin{itemize}
        \item there are at least $(1 - \alpha - \beta)K$ components $x_j \geq u;$
        \item there are at least $(\alpha + \beta)K$ components $x_j \leq u;$
        \item there are at least $(1 - \beta)K$ components $y_i \geq l;$
        \item there are at least $\beta K$ components $y_i \leq l.$
    \end{itemize}
    With $(\bx - \by) = (x_k - y_k)_{k\in[K]},$ we want to show that $Q_{\alpha}[\bx-\by] \leq u - l,$ which means there are $\alpha K$ components $x_k - y_k \leq u - l.$ We now count how many times this inequality fails. In order for a component to be $x_k - y_k \geq u - l,$ it is necessary that either $x_k\geq u,$ which can happen at most $(1 - \alpha - \beta)K$ times, or $y_k\leq l,$ which can happen at most $\beta K$ times. Therefore, the inequality $x_k - y_k \geq u - l$ is satisfied by at most $(1 - \alpha - \beta)K + \beta K = (1 - \alpha)K$ components, leaving at least $\alpha K$ components where $x_k - y_k \leq u - l.$ This is enough to conclude.
    
    We prove property \ref{quant_triangular} as a consequence of property \ref{quant_difference} and property \ref{quant_opposite}.
\end{proof}

In the following, we use the notation $[K] = \{1,\ldots,K\}$ and $[K]_I := \{ k \in [K] : B_k \subset \mI \}.$ We denote by $K_I$ the cardinality of $[K]_I.$

\begin{lem}\label{lemma.bound_card_Z}
    Let $Z = Z(\bX,Y)$ be a real-valued random variable. Let $\eta\in(0,1)$ and $\gamma, \delta_{K,n}, x > 0$ such that $ \gamma ( 1 - K Var(Z) / (n \delta_{K,n}^2) - x) \geq \max\{\eta, 1 - \eta\}.$ Let $K \in [|\mO|/(1-\gamma), n].$ There exists an event $\Omega = \Omega(Z, K)$ with $\P(\Omega)\geq 1 - \exp(- K \gamma x^2 / 2)$ such that, on this event
    \begin{equation*}
        \big| \{ k \in [K] : |\P_{B_k}(Z) - \E[Z]| \leq \delta_{K,n} \} \big| \geq \max\{\eta, 1 - \eta\} K,
    \end{equation*}
    thus the quantiles $Q_{\eta}[Z],Q_{1-\eta}[Z]$ belong to the interval $[\E[Z] - \delta_{K,n}, \E[Z] + \delta_{K,n}].$
\end{lem}
\begin{proof}[Proof of Lemma~\ref{lemma.bound_card_Z}]
    We have
    \begin{align*}
        | \{ k \in [K] &: |\P_{B_k} (Z) - \E[Z]| \leq \delta_{K,n} \} | \geq \sum_{k \in [K]_I} \Ind \{ |\P_{B_k} (Z) - \E[Z]| \leq \delta_{K,n} \} \\
        &= K_I - \sum_{k \in [K]_I} \P_{\bX} \{ |\P_{B_k} (Z) - \E[Z]| \geq \delta_{K,n} \} \\
        &\quad - \sum_{k \in [K]_I} \Big( \Ind \{ |\P_{B_k} (Z) - \E[Z]| \geq \delta_{K,n} \}
        - \P_{\bX} \{ |\P_{B_k} (Z) - \E[Z]| \geq \delta_{K,n} \} \Big).
    \end{align*}
    We bound the second term using Chebychev's inequality
    \begin{align*}
        \sum_{k \in [K]_I} \P_{\bX} \{ |\P_{B_k} (Z) - \E[Z]| \geq \delta_{K,n} \}
        \leq K_I \frac{Var[P_{B_k} (Z) - \E[Z]]}{\delta_{K,n}^2}
        = K_I \frac{Var[Z]}{|B_k| \delta_{K,n}^2}
        = K_I \frac{K Var[Z]}{n \delta_{K,n}^2}.
    \end{align*}
    We bound the last term using Hoeffding's inequality
    \begin{align*}
        \sum_{k \in [K]_I} \Big( \Ind \{ |\P_{B_k} (Z) - \E[Z]| \geq \delta_{K,n} \}
        - \P_{\bX} \{ |\P_{B_k} (Z) - \E[Z]| \geq \delta_{K,n} \} \Big)
        \leq x K_I,
    \end{align*}
    on an event $\Omega(Z, K)$ of probability greater than $1 - \exp(- x^2 K_I / 2).$
    Combining the previous inequalities, we get that on $\Omega(Z, K),$
    \begin{align*}
        | \{ k \in [K]_I : |\P_{B_k} (Z) - \E[Z]| \leq \delta_{K,n} \} |
        &\geq K_I \left( 1 - \frac{K Var[Z]}{n \delta_{K,n}^2} - x \right)
        \geq K \gamma \left( 1 - \frac{K Var[Z]}{n \delta_{K,n}^2} - x \right),
    \end{align*}
    and the last term is bigger than $\max\{\eta, 1-\eta\}K$ by assumption. By definition, this also means that the quantiles $Q_{\eta}[Z],Q_{1-\eta}[Z]$ belong to the interval $[\E[Z] - \delta_{K,n}, \E[Z] + \delta_{K,n}].$
\end{proof}

\begin{lem}\label{lemma.Q1/8_zeta2}
    Let $K \in [16 |\mO|, n].$ On an event $\Omega(K)$ with probability $\P(\Omega(K))\geq 1 - \exp(- K / 4320),$ the quantiles $Q_{1/8,K}[\zeta^2], Q_{7/8,K}[\zeta^2]$ belong to the interval $[\sigma^{*2} - \delta_{K,n}, \sigma^{*2} + \delta_{K,n}],$ with $\delta_{K,n}$ defined in~\eqref{def.alpha_2k_delta_Kn}.
\end{lem}
\begin{proof}[Proof of Lemma~\ref{lemma.Q1/8_zeta2}]
    We use Lemma~\ref{lemma.bound_card_Z} with $\eta=1/8,$ $Z=\zeta^2,$ $\Var(Z) = \E[\zeta^4]-\E[\zeta^2]^2 = \sigma^{*4}(\kappa^*-1),$ $\eta=1/8,$ $\gamma = 15/16,$ $x=1/45,$ and $\delta_{K,n}^2 \geq 25(K/n)\Var(Z).$ Then, 
    \begin{align*}
        \gamma \left( 1 - x - \frac{K Var(Z)}{n \delta_{K,n}^2} \right) \geq \frac{15}{16}\left( 1 - \frac{1}{45} - \frac{1}{25} \right) = \frac{15}{16} - \frac{7}{120} > \frac{7}{8} = 1 - \eta.
    \end{align*}
    The probability of the corresponding event is $\P(\Omega(K))\geq 1 - \exp(- K \gamma x^2/ 2) = 1 - \exp(- K / 4320).$
\end{proof}

\begin{lem}[Lemma 3 in~\cite{lecue2020robustML}] \label{lemma.3} 
    Grant Assumption~\ref{ass.main}. Fix $\eta \in (0,1)$ and $\rho\in(0,+\infty].$ Let $\alpha,\gamma,\gamma_P,x$ be positive real numbers such that $\gamma(1-\alpha-x-16\gamma_P\theta_0) \geq 1-\eta.$ Assume that $K$ is an integer in $[|\mO|/(1-\gamma),n\alpha/4\theta_0^2].$
    Then, there exists an event $\Omega_Q(K)$ with probability $\P(\Omega_Q(K)) \geq 1 - 4 \exp(- K\gamma x^2 / 2)$ and, on this event: for all $f \in \mF$ with $\|f - f^*\| \leq \rho,$ if $\|f - f^*\|_{2,\bX} \geq r_P(\rho, \gamma_P)$ then
    \begin{align*}
        \left|\left\{k\in[K] : \P_{B_k}(f-f^*)^2 \geq (4 \theta_0)^{-2} \|f - f^*\|_{2,\bX}^2 \right\} \right| \geq (1-\eta)K
    \end{align*}
    In particular, $Q_{\eta, K}[ (f - f^*)^2] \geq (4 \theta_0)^{-2} \|f - f^*\|_{2,\bX}^2.$
\end{lem}

\begin{lem}[Lemma 4 in~\cite{lecue2020robustML}] \label{lemma.4} 
    Grant Assumption~\ref{ass.main}. Fix $\eta \in (0,1)$ and $\rho\in(0,+\infty].$ Let $\alpha,\gamma,\gamma_M,x$ be positive real numbers such that $\gamma(1-\alpha-x-8\gamma_M/\eps) \geq 1-\eta.$ Assume that $K$ is an integer in $[|\mO|/(1-\gamma),n].$
    Then, there exists an event $\Omega_M(K)$ with probability $\P(\Omega_M(K)) \geq 1 - \exp(- K\gamma x^2 / 2)$ and, on this event: for all $f \in \mF$ with $\|f - f^*\| \leq \rho,$
    \begin{align*}
        \left|\left\{k\in[K] : |(\P_{B_k} - \E)(2\zeta(f-f^*)| \leq \alpha_M^2 \right\} \right| \geq (1-\eta)K,
    \end{align*}
    with 
    \begin{align*}
        \alpha_M^2 := \eps \max\left(\frac{16\theta_m^2}{\eps^2 \alpha} \frac{K}{n},\ r_M^2(\rho,\gamma_M),\ \|f-f^*\|_{2,\bX}^2 \right).
    \end{align*}
\end{lem}

\begin{lem}\label{lemma.useful_bounds_lecue}
    Let $K\in \left[32|\mO|,\ n/(372\theta_0^2) \right].$ There exists an event $\Omega(K)$ of probability bigger than $1 - 2 \exp(- K / 8928)$ such that, for all $\rho \in \{ \rho_K, 2\rho_K \}$, and all $f \in \mF$ such that $\|f - f^*\| \leq \rho$, we have
    \begin{enumerate}
        \item if $\|f - f^*\|_{2,\bX} \geq r_P(\rho, \gamma_P)$, then $Q_{1/16, K} \big( (f - f^*)^2 \big)
        \geq (4 \theta_0)^{-2} \|f - f^*\|_{2,\bX}^2;$
         \item $Q_{15/16, K} \big[-2 \zeta(f-f^*) \big] \leq \E[-2\zeta(f-f^*)(\bX)] + \alpha_M^2,$
         \item $Q_{1/16,K}[-2 \zeta (f - f^*)] \geq \E[-2 \zeta (f - f^*)(\bX)] - \alpha_M^2.$
        \item $Q_{15/16, K} \big[2 \zeta(f-f^*) \big] \leq \alpha_M^2,$
    \end{enumerate}
    with
    \begin{align*}
        \alpha_M^2 := \eps \max\left(\frac{1488 \theta_m^2}{\eps^2} \frac{K}{n},\ r_M^2(\rho,\gamma_M),\ \|f-f^*\|_{2,\bX}^2 \right), \quad \theta_m = \theta_1 \frakm^*.
    \end{align*}
    Furthermore, for $r(\cdot)$ as in Theorem~\ref{thm.main_theorem} and $\|f - f^*\|_{2,\bX} \leq r(\rho),$ we find $\alpha_M^2 \leq 4\eps r^2(\rho).$
\end{lem}
\begin{proof}[Proof of Lemma~\ref{lemma.useful_bounds_lecue}] 
    The first property follows from applying Lemma~\ref{lemma.3} with $\eta=1/16,$ $\rho\in\{\rho_K, 2\rho_K\},$ $\alpha = x = 1/93,$ $\gamma = 31/32,$ $\gamma_P = 1 / (1488 \theta_0)$ and checking that $\gamma(1-\alpha-x-16\gamma_P\theta_0) \geq 1-\eta.$ With our choices, we find
    \begin{align*}
        \frac{31}{32}\left(1 - \frac{1}{93} - \frac{1}{93} - \frac{16}{1488} \right) = \frac{31}{32} \left( 1 - \frac{1}{31} \right) = \frac{30}{32} = \frac{15}{16}.
    \end{align*}
    The corresponding event $\Omega_1$ has probability at least $1 - \exp(-K\gamma x^2/2) = 1 - \exp(- K / 8928).$
    
    The second and third properties follow from applying Lemma~\ref{lemma.4} with $\eta=1/16,$ $\rho\in{\rho_K, 2\rho_K},$ $\alpha = x = 1/93,$ $\gamma = 31/32,$ $\gamma_M = \eps/744$ and checking that $\gamma(1-\alpha-x-8\gamma_M/\eps) \geq 1-\eta.$ With our choices, we find
    \begin{align*}
        \frac{31}{32}\left(1 - \frac{1}{93} - \frac{1}{93} - \frac{8}{744} \right) = \frac{31}{32} \left( 1 - \frac{1}{31} \right) = \frac{30}{32} = \frac{15}{16}.
    \end{align*}
    The corresponding event $\Omega_2$ has probability at least $1 - \exp(-K\gamma x^2/2) = 1 - \exp(- K / 8928).$
    
    The fourth property holds on the same event $\Omega_2$ given above, and is a consequence of the nearest point theorem and the convexity of the function class $\mF,$ which guarantee that $\E[2\zeta(f-f^*)(\bX)]\leq 0.$
    
    Given all the above, the probability of the event $\Omega(K) = \Omega_1 \cap \Omega_2$ is at least $1 - \P(\Omega_1) - \P(\Omega_1) = 1 - 2\exp(- K / 8928).$
    
    We finally bound, with $r^2(\rho_K) = 384\theta_m^2 K/(n\eps^2),$
    \begin{align*}
        \frac{\alpha_M^2}{r^2(2\rho_K)} \leq \frac{\alpha_M^2}{r^2(\rho_K)} = \eps \max\left(\frac{1488 \theta_m^2}{\eps^2} \frac{K}{n} \frac{1}{r^2(\rho_K)},\ 1 \right) = \eps \frac{1488}{384} < 4\eps.
    \end{align*}
\end{proof}

\begin{lem}\label{lemma.l2_quantile}
    Let $K \in [32|\mO|, n/(372\theta_0^2)].$ There exists an event $\Omega_Q(K)$ of probability bigger than $1 - \exp(- K / 8928)$ such that, for all $\rho \in \{ \rho_K, 2\rho_K \},$ and all $f \in \mF$ such that $\|f - f^*\| \leq \rho,$ we have 
    \begin{align*}
        Q_{15/16,K}\big[(f-f^*)^2 \big] \leq \|f-f^*\|_{2,\bX}^2 + \alpha_Q^2,
    \end{align*}
    with
    \begin{align*}
        \alpha_Q^2 := \eps \max\bigg(\|f-f^*\|_{2,\bX}^2 \frac{1488 \theta_1^4}{\eps^2}\frac{K}{n},\ r_Q^2(\rho,\gamma_Q),\ \|f-f^*\|_{2,\bX}^2\bigg).
    \end{align*}
\end{lem}
\begin{proof}[Proof of Lemma~\ref{lemma.l2_quantile}]
    Take $\eta = 1/16,$ $\gamma = 31/32,$ $\alpha = x = 1/93$ and $\gamma_Q = \eps / 372.$ We follow the steps of the proof of Lemma 4 in~\cite{lecue2020robustML}. For all $f \in \mF$ and $\rho > 0,$ set $\B(f, \rho) = \{ g \in \mF: \|g - f \| \leq \rho\}.$ For all $k\in [K],$ set $\mD_k=(\bX_i,Y_i)_{i\in B_k}$ and
    \begin{align*}
        g_f(\mD_k) &:= (\P_{B_k} - \E)[(f-f^*)^2], \\
        \alpha_Q^2(f) &:= \eps \max\bigg(\|f-f^*\|_{2,\bX}^2 \frac{4\theta_1^4}{\eps^2\alpha}\cdot\frac{K}{n}, r_Q^2(\rho,\gamma_Q), \|f-f^*\|_{2,\bX}^2 \bigg).
    \end{align*}
    Let $[K]_I = \{k\in[K]:B_k\subset\mI\}$ and consider any $k\in[K]_I.$
    An application of Markov inequality gives
    \begin{align*}
        \P\big(2 |g_f(\mD_k)| &\geq \alpha_Q^2(f) \big) \leq \frac{4 \E\Big[|g_f(\mD_k)|^2\Big]}{\alpha_Q^2(f)\cdot \alpha_Q^2(f)}.
    \end{align*}
    The denominator of the last term in the previous display can be bounded below using both $\alpha_Q^2(f) \geq \eps \|f-f^*\|_{2,\bX}^2$ and $\alpha_Q^2(f) \geq \|f-f^*\|_{2,\bX}^2 4\theta_1^4 K / (\eps\alpha n).$ This gives
    \begin{align*}
        \P\big(2 |g_f(\mD_k)| \geq \alpha_Q^2(f) \big) &\leq \frac{4 \E\Big[ \big((\P_{B_k}-\P_{\bX})(f-f^*)^2\big)^2\Big]}{\|f-f^*\|_{2,\bX}^2 \frac{ 4\theta_1^4}{\alpha} \frac{K}{n} \|f-f^*\|_{2,\bX}^2} \\
        &\leq \frac{\sum_{i\in B_k} \Var\big( (f-f^*)^2(\bX_i) \big)}{|B_k|^2 \frac{\theta_1^4}{\alpha} \frac{K}{n} \|f-f^*\|_{2,\bX}^4} \\
        &\leq \frac{\E[(f-f^*)^4(\bX)]}{|B_k| \frac{\theta_1^4}{\alpha} \frac{K}{n} \|f-f^*\|_{2,\bX}^4} \\
        &\leq \frac{\alpha \|f-f^*\|_{4,\bX}^4}{\theta_1^4 \|f-f^*\|_{2,\bX}^4} \\
        &\leq \alpha,
    \end{align*}
    since $\|f-f^*\|_{4,\bX} \leq \theta_1 \|f-f^*\|_{2,\bX}$ by Assumption~\ref{ass.main}. The following bound follows exactly from the proof of Lemma 4 in~\cite{lecue2020robustML}. Take $J = \cup_{k\in[K]_I} B_k$ and write $r_Q(\rho) = r_Q(\rho,\gamma_Q).$ Take $\B(f^*,\rho,r_Q(\rho))$ the set of functions $f\in\B(f^*,\rho)$ such that $\|f-f^*\|_{2,\bX} \leq r_Q(\rho).$ We have
    \begin{align*}
        \E \bigg[ \sup_{f\in \B(f^*,\rho)} \sum_{k\in [K]_I} \xi_k \frac{g_f(\mD_k)}{\alpha_Q^2(f)}\bigg] \leq \frac{2}{\eps r_Q^2(\rho)} \E \bigg[\sup_{f\in \B(f^*,\rho,r_Q(\rho))} \Big|\sum_{k\in [K]_I} \xi_k (\P_{B_k}-\E)(f-f^*)^2 \Big| \bigg].
    \end{align*}
    Furthermore, we can apply the symmetrization argument in the proof of Lemma 4 in~\cite{lecue2020robustML}. Together with the definition of $r_Q(\cdot),$ we find
    \begin{align*}
        \E \bigg[\sup_{f\in \B(f^*,\rho)} \sum_{k\in [K]_I} \xi_k \frac{g_f(\mD_k)}{\alpha_Q^2(f)}\bigg] \leq \frac{4K}{\eps n} \gamma_Q |[K]_I| \frac{n}{K} = \frac{4\gamma_Q}{\eps}|[K]_I|.
    \end{align*}
    Now we utilize the function $\psi$ found in the proof of Lemma 4 in~\cite{lecue2020robustML}. On an event $\Omega(K)$ with probability at least $1 - \exp(- K\gamma x^2 / 2) = 1 - \exp(- K/8928),$
    \begin{align*}
        \sum_{k\in [K]_I} &\Ind\big(|g_f(\mD_k)| < \alpha_Q^2(f)\big) \\
        &\geq (1-\alpha) |[K]_I| - 2\E \bigg[\sup_{f\in \B(f^*,\rho)} \sum_{k\in [K]_I} \psi\bigg( \frac{|g_f(\mD_k)|}{\alpha_Q^2(f)}\bigg)\bigg] + |[K]_I| x\\
        &\geq (1-\alpha) |[K]_I| - 2\E \bigg[\sup_{f\in \B(f^*,\rho)} \sum_{k\in [K]_I} \xi_k \frac{|g_f(\mD_k)|}{\alpha_Q^2(f)}\bigg] - |[K]_I| x \\
        &\geq |[K]_I|\bigg(1 - \alpha - x - \frac{4\gamma_Q}{\eps}\bigg) \\
        &\geq \gamma K \bigg(1 - \alpha - x - \frac{4\gamma_Q}{\eps}\bigg).
    \end{align*}
    We now check that the latter is bigger than $(1-\eta)K.$ With our choices, this gives
    \begin{align*}
        \frac{31}{32}\left(1 - \frac{1}{93} - \frac{1}{93} - \frac{4}{372} \right) = \frac{31}{32} \left( 1 - \frac{1}{31} \right) = \frac{30}{32} = \frac{15}{16},
    \end{align*}
    which is what we want. As a consequence, $Q_{15/16,K}[(f-f^*)^2] \leq \|f-f^*\|_{2,\bX}^2 + \alpha_Q^2(f).$
\end{proof}

In the next result we use the event $\Omega(K) := \Omega_1(K) \, \cap \, \Omega_2(K) \, \cap \, \Omega_3(K)$ with $\Omega_1(K),\Omega_2(K)$ and $\Omega_3(K)$ respectively defined as the events in Lemma~\ref{lemma.Q1/8_zeta2}, Lemma~\ref{lemma.useful_bounds_lecue} and Lemma~\ref{lemma.l2_quantile}. The event $\Omega(K)$ has probability at least $1 - 4\exp(-K/8920).$ We also denote by $r(\cdot)$ any function satisfying $r(\rho) \geq \max\{r_P(\rho,\gamma_P), r_M(\rho,\gamma_M) \}.$ For any integer $K$ and $c_\rho\in\{1,2\},$ we will use the notation $\alpha_{K,c_\rho} := c_\alpha r(c_\rho \rho)$ and  $\delta_{K,n}^2 := 25 \frakm^{*4} K/n.$

\begin{lem} \label{lemma.bound_P2zeta_f_fstar}
    Let $C^2 = 384 \theta_1^2 c_r^2 c_\alpha^2 \kappa_{+}^{1/2}$ and
    \begin{align*}
        K \in \left[32|\mO|, \frac{n}{372\theta_0^2} \wedge \frac{n}{25 \kappa_{+}} \wedge \frac{n\eps^2}{C^2} \right].
    \end{align*}
    On the event $\Omega(K)$ defined above, for all $f \in \mF$ such that
    $\|f - f^*\| \leq c_\rho \rho_K,$ $\|f-f^*\|_{2,\bX} \leq r(c_\rho \rho_K)$ and $|\sigma - \sigma^*| \leq \alpha_{K,c_\rho},$
    \begin{align*}
        \E[-2 &\zeta (f - f^*)(\bX)] \leq \frac{2\sigma^*+\alpha_{K,c_\rho}}{2c} T_{K,\mu}(f^*,\sigma^*,f,\sigma) + \frac{2\sigma^*+\alpha_{K,c_\rho}}{2c} \mu\rho + \alpha_M^2 \\
        &\quad + \frac{8 (2\sigma^* + \alpha_{K,c_\rho})}{c \sigma^* (2\sigma^*-\alpha_{K,c_\rho})^2} \delta_{K,n}^2 + \frac{\alpha_{K,c_\rho}}{c (2\sigma^*-\alpha_{K,c_\rho})} \left( 2 \sigma^* r(c_\rho \rho_K) + r^2(c_\rho \rho_K) + \alpha_Q^2 + \alpha_M^2 \right) ,
    \end{align*}
    where $\alpha_M^2, \alpha_Q^2$ are given in Lemma~\ref{lemma.useful_bounds_lecue} and Lemma~\ref{lemma.l2_quantile}.
\end{lem}
\begin{proof}[Proof of Lemma~\ref{lemma.bound_P2zeta_f_fstar}.]
   We start by applying Lemma~\ref{lemma.useful_bounds_lecue}, which gives
   \begin{align*}
        \E[-2 \zeta (f - f^*)(\bX)] \leq Q_{1/4,K}[-2 \zeta (f - f^*)] + \alpha_M^2 \leq Q_{1/4,K}[(f-f^*)^2 - 2 \zeta (f - f^*)] + \alpha_M^2,
    \end{align*}
    the second inequality follows from the fact that $(f-f^*)^2$ is positive. Using the definition of $T_{K,\mu}(f^*,\sigma^*,f,\sigma)$ in~\eqref{def.T_functional} and the quantile properties in Lemma~\ref{lemma.quantile_prop}, we can rewrite 
    \begin{align*}
        \E[-2 \zeta &(f - f^*)(\bX)] \\
        &\leq Q_{1/4,K}[(f-f^*)^2 - 2 \zeta (f - f^*)] + \alpha_M^2 \\
        &= \frac{\sigma+\sigma^*}{2c} Q_{1/4,K} \bigg[ 2c\frac{\ell_{f} - \ell_{f^*}}{\sigma+\sigma^*} \bigg] + \alpha_M^2 \\
        &= \frac{\sigma+\sigma^*}{2c} Q_{1/4,K} \bigg[ R_c(\ell_{f^*},\sigma^*,\ell_f,\sigma) - (\sigma-\sigma^*) \bigg( 1 - 2 \frac{\ell_f+\ell_{f^*}}{(\sigma+\sigma^*)^2} \bigg)\bigg] + \alpha_M^2 \\
        &\leq \frac{\sigma+\sigma^*}{2c} \left( Q_{1/2,K} \Big[ R_c(\ell_{f^*},\sigma^*,\ell_f,\sigma) \Big] - Q_{1/4,K} \bigg[ (\sigma-\sigma^*) \bigg( 1 - 2 \frac{\ell_f+\ell_{f^*}}{(\sigma+\sigma^*)^2} \bigg)\bigg] \right) + \alpha_M^2 \\
        &\leq \frac{\sigma+\sigma^*}{2c} \bigg( Q_{1/2,K} \Big[ R_c(\ell_{f^*},\sigma^*,\ell_f,\sigma) \Big] + \mu(\|f\|-\|f^*\|) \bigg) + \frac{\sigma+\sigma^*}{2c}\mu\rho + \alpha_M^2 \\
        &\quad  - \frac{\sigma+\sigma^*}{2c} Q_{1/4,K} \bigg[(\sigma-\sigma^*) \bigg( 1 - 2 \frac{\ell_f+\ell_{f^*}}{(\sigma+\sigma^*)^2} \bigg)\bigg] \\
        &= \frac{\sigma+\sigma^*}{2c} T_{K,\mu}(f^*,\sigma^*,f,\sigma) + \frac{\sigma+\sigma^*}{2c} \Big( \mu\rho - Q_{1/4,K}\bigg[(\sigma-\sigma^*) \bigg( 1 - 2 \frac{\ell_f+\ell_{f^*}}{(\sigma+\sigma^*)^2} \bigg)\bigg] \Big) + \alpha_M^2.
    \end{align*}
    Since $\sigma+\sigma^* \leq 2\sigma^* + \alpha_{K,c_\rho},$ it remains to show that
    \begin{align}\label{eq.lem_emp_risk_toshow}
        - \frac{\sigma+\sigma^*}{2c} &Q_{1/4,K} \bigg[ (\sigma-\sigma^*) \bigg( 1 - 2 \frac{\ell_f+\ell_{f^*}}{(\sigma+\sigma^*)^2} \bigg)\bigg] \\
        &\leq \frac{8 (2\sigma^* + \alpha_{K,c_\rho})}{c \sigma^* (2\sigma^*-\alpha_{K,c_\rho})^2} \delta_{K,n}^2 + \frac{\alpha_{K,c_\rho}}{c (2\sigma^*-\alpha_{K,c_\rho})} \left( 2 \sigma^* r(c_\rho \rho_K) + r^2(c_\rho \rho_K) + \alpha_Q^2 + \alpha_M^2 \right). \nonumber
    \end{align}
    
    \medskip
    
    First, by the quantile properties in Lemma~\ref{lemma.quantile_prop}, we have
    \begin{align*}
        - \frac{\sigma+\sigma^*}{2c} Q_{1/4,K} \bigg[ (\sigma-\sigma^*) \bigg( 1 - 2 \frac{\ell_f+\ell_{f^*}}{(\sigma+\sigma^*)^2} \bigg)\bigg] &\leq \frac{\sigma+\sigma^*}{2c} Q_{3/4,K} \bigg[ (\sigma-\sigma^*) \bigg( 2 \frac{\ell_f+\ell_{f^*}}{(\sigma+\sigma^*)^2} -1 \bigg)\bigg].
    \end{align*}
    By expanding $\ell_f=\ell_{f^*}+\ell_f-\ell_{f^*},$ we get
    \begin{align*}
        \frac{\sigma+\sigma^*}{2c} Q_{3/4,K} &\bigg[ (\sigma-\sigma^*) \bigg( 2 \frac{\ell_f+\ell_{f^*}}{(\sigma+\sigma^*)^2} -1 \bigg)\bigg] \\
        &= \frac{\sigma+\sigma^*}{2c} Q_{3/4,K}\Bigg[ (\sigma-\sigma^*) \bigg( \frac{4 \ell_{f^*}}{(\sigma+\sigma^*)^2} - 1\bigg) + (\sigma-\sigma^*) \frac{2(\ell_f-\ell_{f^*})}{(\sigma+\sigma^*)^2} \Bigg] \\
        &\leq \frac{\sigma+\sigma^*}{2c} Q_{7/8,K}\Bigg[ (\sigma-\sigma^*) \bigg( \frac{4 \ell_{f^*}}{(\sigma+\sigma^*)^2} - 1\bigg) \Bigg] + \frac{Q_{7/8,K}\left[(\sigma-\sigma^*) (\ell_f-\ell_{f^*}) \right]}{c (\sigma+\sigma^*)} .
    \end{align*}
    Since the term $(\sigma-\sigma^*)$ has different signs for $\sigma<\sigma^*$ and $\sigma>\sigma^*,$ we need to account for this in the bounds. We focus first on the term
    \begin{align*}
        &Q_{7/8,K}\Bigg[ (\sigma-\sigma^*) \bigg( \frac{4 \ell_{f^*}}{(\sigma+\sigma^*)^2} - 1\bigg) \Bigg] \\
        &\quad \leq \max \bigg\{ \sup_{\sigma\in(\sigma^*,\sigma^*+\alpha_{K,c_\rho}]} (\sigma-\sigma^*) \bigg( \frac{4 Q_{7/8,K}[\ell_{f^*}]}{(\sigma+\sigma^*)^2} - 1\bigg),\  \sup_{\sigma\in[\sigma^*-\alpha_{K,c_\rho}, \sigma^*)} (\sigma^*-\sigma) \bigg(1 - \frac{4 Q_{7/8,K}[\ell_{f^*}]}{(\sigma+\sigma^*)^2} \bigg) \bigg\}.
    \end{align*}
    Thanks to Lemma~\ref{lemma.Q1/8_zeta2}, the quantile $Q_{7/8,K}[\ell_{f^*}] = Q_{7/8,K}[\zeta^2]$ is in the interval $[\sigma^{*2}-\delta_{K,n},\sigma^{*2}+\delta_{K,n}],$ therefore
    \begin{align}
    \begin{split}\label{eq.Q7/8_control}
        &Q_{7/8,K}\Bigg[ (\sigma-\sigma^*) \bigg( \frac{4 \ell_{f^*}}{(\sigma+\sigma^*)^2} - 1\bigg) \Bigg] \\
        &\leq \max \bigg\{ \sup_{\sigma\in(\sigma^*,\sigma^*+\alpha_{K,c_\rho}]} (\sigma-\sigma^*) \bigg( \frac{4 (\sigma^{*2}+\delta_{K,n})}{(\sigma+\sigma^*)^2} - 1\bigg),\  \sup_{\sigma\in[\sigma^*-\alpha_{K,c_\rho}, \sigma^*)} (\sigma^*-\sigma) \bigg(1 - \frac{4 (\sigma^{*2}-\delta_{K,n})}{(\sigma+\sigma^*)^2} \bigg) \bigg\}.
    \end{split}
    \end{align}
    We denote $a_{+}^2 = \sigma^{*2}+\delta_{K,n}$ and $a_{-}^2 = \sigma^{*2}-\delta_{K,n}.$ The first function in the latter display is positive (or zero) for $\sigma \in [\sigma^*, 2a_{+} - \sigma^*].$ Let $\sigma_{a_{+}}$ be the point achieving the maximum, then $\sigma_{a_{+}}$ belongs to the same interval and $|\sigma_{a_{+}} - \sigma^*| \leq 2a_{+} - 2\sigma^* = 2\sigma^*(\sqrt{1+\delta_{K,n}/\sigma^{*2}} - 1).$ By construction, the quantity $\delta_{K,n}/\sigma^{*2}$ is smaller than one, since
    \begin{align*}
        \frac{\delta_{K,n}^2}{\sigma^{*4}} = \frac{25 \mu^{*4} K}{\sigma^{*4} n} = \frac{25 \kappa^* K}{n} \leq  \frac{25 \kappa_{+} K}{n} \leq 1
    \end{align*}
    and $K \leq n/(25\kappa_{+}).$ For all $x\in(0,1),$ the inequality $\sqrt{1 + x} \leq 1 + x$ holds, so that 
    \begin{align*}
        |\sigma_{a_{+}} - \sigma^*| \leq 2\sigma^*\left(\sqrt{1+\frac{\delta_{K,n}}{\sigma^{*2}} } - 1\right) \leq 2\sigma^*\left(1+\frac{\delta_{K,n}}{\sigma^{*2}} - 1\right) = \frac{2\delta_{K,n}}{\sigma^*}.
    \end{align*}
    Now we repeat the same argument for the second function in \eqref{eq.Q7/8_control}, using $\sqrt{1 - x} \geq 1 - x$ for all $x\in(0,1),$ thus getting a point $\sigma_{a_{-}}$ achieving the maximum such that $|\sigma_{a_{-}} - \sigma^*| \leq 2\delta_{K,n}/\sigma^*.$ By Lemma~\ref{lemma.conditions}, we have $2\delta_{K,n}/\sigma^* < \alpha_{K,c_\rho} < \sigma^*.$ 
    With $\delta_a = 2\delta_{K,n}/\sigma^*,$ this yields
    \begin{align*}
        Q_{7/8,K}&\Bigg[ (\sigma-\sigma^*) \bigg( \frac{4 \ell_{f^*}}{(\sigma+\sigma^*)^2} - 1\bigg) \Bigg] \\
        &\leq \max \bigg\{ (\sigma^* - \sigma_{a_{-}}) \bigg( 1 - \frac{4 a_{-}^2}{(\sigma_{a_{-}}+\sigma^*)^2} \bigg),\ (\sigma_{a_{+}}-\sigma^*) \bigg( \frac{4 a_{+}^2}{(\sigma_{a_{+}}+\sigma^*)^2} - 1\bigg) \bigg\} \\
        &\leq \frac{2\delta_{K,n}}{\sigma^*} \max \bigg\{ 1 - \frac{4 \sigma^{*2}- 4\delta_{K,n}}{(2\sigma^*-\delta_a)^2},\ \frac{4\sigma^{*2} + 4\delta_{K,n}}{(2\sigma^*+ \delta_a)^2} - 1 \bigg\} \\
        &= \frac{2\delta_{K,n}}{\sigma^*} \max \bigg\{ \frac{ 4\sigma^*\delta_a + \delta_a^2 +  4\delta_{K,n}}{(2\sigma^*-\delta_a)^2},\ \frac{4\delta_{K,n} - 4\sigma^*\delta_a - \delta_a^2}{(2\sigma^*+ \delta_a)^2} \bigg\} \\
        &\leq \frac{16 \delta_{K,n}^2}{\sigma^*(2\sigma^*-\delta_a)^2} \\
        &\leq \frac{16 \delta_{K,n}^2}{\sigma^*(2\sigma^*-\alpha_{K,c_\rho})^2}.
    \end{align*}
    
    One last term needs to be bounded in order to obtain~\eqref{eq.lem_emp_risk_toshow}. We only consider the case when $\sigma \in [\sigma^*,\sigma^*+\alpha_{K,c_\rho}],$ the case $\sigma \in [\sigma^*-\alpha_{K,c_\rho},\sigma^*]$ follows the same steps. With $\ell_{f^*}-\ell_f = 2\zeta(f-f^*) - (f-f^*)^2,$ we get
    \begin{align*}
        \frac{1}{c (\sigma+\sigma^*)} &Q_{7/8,K}\left[(\sigma-\sigma^*) (\ell_f-\ell_{f^*}) \right] = \frac{(\sigma-\sigma^*)}{c (\sigma+\sigma^*)} Q_{7/8,K}\left[(f-f^*)^2 - 2\zeta(f-f^*) \right] \\
        &\leq \frac{\alpha_{K,c_\rho}}{c (2\sigma^*-\alpha_{K,c_\rho})} \left( Q_{15/16,K}\left[(f-f^*)^2\right] + Q_{15/16,K}\left[-2\zeta(f-f^*) \right] \right) \\
        &\leq \frac{\alpha_{K,c_\rho}}{c (2\sigma^*-\alpha_{K,c_\rho})} \left( \|f-f^*\|_{2,\bX}^2 + \alpha_Q^2 + \E[-2\zeta(f-f^*)(\bX)] + \alpha_M^2 \right),
    \end{align*}
    the last inequality follows from Lemma~\ref{lemma.useful_bounds_lecue} and Lemma~\ref{lemma.l2_quantile}. By the Cauchy-Schwarz inequality, $\E[-2\zeta(f-f^*)(\bX)] \leq 2 \sigma^* \|f-f^*\|_{2,\bX} \leq 2 \sigma^* r(c_\rho \rho_K).$ By putting everything together, we conclude
    \begin{align*}
        \E[-2 &\zeta (f - f^*)(\bX)] \leq \frac{2\sigma^*+\alpha_{K,c_\rho}}{2c} T_{K,\mu}(f^*,\sigma^*,f,\sigma) + \frac{2\sigma^*+\alpha_{K,c_\rho}}{2c} \mu\rho + \alpha_M^2 \\
        &\quad + \frac{8 (2\sigma^* + \alpha_{K,c_\rho})}{c \sigma^* (2\sigma^*-\alpha_{K,c_\rho})^2} \delta_{K,n}^2 + \frac{\alpha_{K,c_\rho}}{c (2\sigma^*-\alpha_{K,c_\rho})} \left( 2 \sigma^* r(c_\rho \rho_K) + r^2(c_\rho \rho_K) + \alpha_Q^2 + \alpha_M^2 \right),
    \end{align*}
    which gives the claim.
\end{proof}

\end{appendices}

\bibliographystyle{acm}       
\bibliography{bibo}          

\begin{thebibliography}{10}

\bibitem{Alon1999}
{\sc Alon, N., Matias, Y., and Szegedy, M.}
\newblock The space complexity of approximating the frequency moments.
\newblock {\em Journal of Computer and System Sciences 58}, 1 (1999), 137 --
  147.

\bibitem{bellec2017}
{\sc {Bellec}, P.~C., {Lecu{\'e}}, G., and {Tsybakov}, A.~B.}
\newblock {Towards the study of least squares estimators with convex penalty}.
\newblock {\em arXiv e-prints\/} (Jan. 2017), arXiv:1701.09120.

\bibitem{bellec2018}
{\sc Bellec, P.~C., Lecué, G., and Tsybakov, A.~B.}
\newblock Slope meets lasso: Improved oracle bounds and optimality.
\newblock {\em Ann. Statist. 46}, 6B (12 2018), 3603--3642.

\bibitem{bellec2016}
{\sc {Bellec}, P.~C., and {Tsybakov}, A.~B.}
\newblock {Bounds on the prediction error of penalized least squares estimators
  with convex penalty}.
\newblock {\em arXiv e-prints\/} (Sept. 2016), arXiv:1609.06675.

\bibitem{belloni2010}
{\sc Belloni, A., Chernozhukov, V., and Wang, L.}
\newblock Square-root lasso: pivotal recovery of sparse signals via conic
  programming.
\newblock {\em Biometrika 98}, 4 (2011), 791--806.

\bibitem{belloni2014pivotal}
{\sc Belloni, A., Chernozhukov, V., and Wang, L.}
\newblock Pivotal estimation via square-root lasso in nonparametric regression.
\newblock {\em Annals of Statistics 42}, 2 (2014), 757--788.

\bibitem{comminges2018adaptive}
{\sc Comminges, L., Collier, O., Ndaoud, M., and Tsybakov, A.~B.}
\newblock Adaptive robust estimation in sparse vector model.
\newblock {\em arXiv preprint arXiv:1802.04230\/} (2018).

\bibitem{derumigny2018improved}
{\sc Derumigny, A.}
\newblock Improved bounds for square-root {L}asso and square-root {S}lope.
\newblock {\em Electronic Journal of Statistics 12}, 1 (2018), 741--766.

\bibitem{derumigny2019thesis}
{\sc Derumigny, A.}
\newblock {\em Some statistical results in high-dimensional dependence
  modeling}.
\newblock PhD thesis, Universit{\'e} Paris-Saclay (ComUE), 2019.

\bibitem{devroye2016sub}
{\sc Devroye, L., Lerasle, M., Lugosi, G., and Oliveira, R.~I.}
\newblock Sub-gaussian mean estimators.
\newblock {\em The Annals of Statistics 44}, 6 (2016), 2695--2725.

\bibitem{giraud2014introduction}
{\sc Giraud, C.}
\newblock {\em Introduction to high-dimensional statistics}, vol.~138.
\newblock CRC Press, 2014.

\bibitem{Jerrum1986}
{\sc Jerrum, M.~R., Valiant, L.~G., and Vazirani, V.~V.}
\newblock Random generation of combinatorial structures from a uniform
  distribution.
\newblock {\em Theoretical Computer Science 43\/} (1986), 169 -- 188.

\bibitem{lecue2013learning}
{\sc Lecu{\'e}, G., and Mendelson, S.}
\newblock Learning subgaussian classes: upper and minimax bounds (2013).
\newblock {\em Topics in Learning Theory-Societe Mathematique de France,(S.
  Boucheron and N. Vayatis Eds.)\/} (2013).

\bibitem{lecue2018regularization}
{\sc Lecu{\'e}, G., and Mendelson, S.}
\newblock Regularization and the small-ball method i: sparse recovery.
\newblock {\em The Annals of Statistics 46}, 2 (2018), 611--641.

\bibitem{lecue2020robustML}
{\sc Lecué, G., and Lerasle, M.}
\newblock Robust machine learning by median-of-means: Theory and practice.
\newblock {\em Ann. Statist. 48}, 2 (04 2020), 906--931.

\bibitem{Levin2005}
{\sc {Levin}, L.~A.}
\newblock {Notes for Miscellaneous Lectures}.
\newblock {\em arXiv e-prints\/} (Mar. 2005), cs/0503039.

\bibitem{lugosi2019regularization}
{\sc Lugosi, G., and Mendelson, S.}
\newblock Regularization, sparse recovery, and median-of-means tournaments.
\newblock {\em Bernoulli 25}, 3 (2019), 2075--2106.

\bibitem{lugosi2017regularization}
{\sc Lugosi, G., and Mendelson, S.}
\newblock Risk minimization by median-of-means tournaments.
\newblock {\em J. Eur. Math. Soc. 22\/} (2020), 925--965.

\bibitem{mendelson2014upper}
{\sc Mendelson, S.}
\newblock Upper bounds on product and multiplier empirical processes.
\newblock {\em Stochastic Processes and their Applications 126}, 12 (2016),
  3652--3680.

\bibitem{mendelson2017multiplier}
{\sc Mendelson, S.}
\newblock On multiplier processes under weak moment assumptions.
\newblock In {\em Geometric Aspects of Functional Analysis}. Springer, 2017,
  pp.~301--318.

\bibitem{NemYud1985}
{\sc Nemirovskij, A.~S., and Yudin, D.~B.}
\newblock {\em Problem complexity and method efficiency in optimization}.
\newblock Wiley-Interscience, 1983.

\end{thebibliography}

\end{document}